\documentclass[reqno,a4paper,12pt]{amsart}
\usepackage{amsmath,amsthm,amssymb,mathrsfs}
\usepackage[
	left = 2.4cm,
	right = 2.4cm,
	top = 2.2 cm,
	bottom = 2.2 cm
]{geometry}
\usepackage{graphicx,xcolor,tikz,mathtools,ulem}
\usepackage{bm,stmaryrd}
\usepackage{enumitem}
\usepackage[toc,page]{appendix}
\usepackage[
	pdfencoding=auto,
	colorlinks = false,
    hidelinks,
	citecolor = blue,
	linkcolor = blue,
	anchorcolor = blue,
	urlcolor = blue, 
	filecolor=blue,
	pdfkeywords={}
]{hyperref}
\usepackage{esint} % Average integral symbol

\usepackage{amsaddr}

\usepackage{todonotes}
\makeatletter
\define@key{todonotes}{Dennis}[]{%
    \setkeys{todonotes}{author=D,inline,color=green!20}}%
\define@key{todonotes}{Yadong}[]{%
    \setkeys{todonotes}{author=Y,inline,color=cyan!20}}%
\makeatother

\theoremstyle{plain}
\newtheorem{theorem}{Theorem}[section]

\newtheorem{definition}[theorem]{Definition}
\newtheorem{lemma}[theorem]{Lemma}
\newtheorem{proposition}[theorem]{Proposition}
\newtheorem{remark}[theorem]{Remark}
\theoremstyle{definition}
\newtheorem{assumption}{Assumption}

\numberwithin{equation}{section}

\newcommand{\bb}[1]{\mathbb{#1}}
\newcommand{\bbr}{\bb{R}}
\newcommand{\bbi}{\bb{I}}
\newcommand{\bbn}{\bb{N}}

\newcommand{\bbp}{\bb{P}}
\newcommand{\bbb}{\bb{B}}
\newcommand{\bba}{\bb{A}}
\newcommand{\bbc}{\bb{C}}
\newcommand{\bbd}{\bb{D}}

\newcommand{\bbf}{\bb{F}}
\newcommand{\bbs}{\bb{S}}
\newcommand{\bbw}{\bb{W}}
\newcommand{\bu}{\mathbf{u}}
\newcommand{\bv}{\mathbf{v}}
\newcommand{\bw}{\mathbf{w}}
\newcommand{\bn}{\mathbf{n}}

\newcommand{\bJ}{\mathbf{J}}
\newcommand{\bT}{\mathbf{T}}
\newcommand{\btau}{\bm{\tau}}
\newcommand{\cD}{\mathcal{D}}
\newcommand{\cE}{\mathcal{E}}
\newcommand{\cF}{\mathcal{F}}
\newcommand{\cT}{\mathcal{T}}
\newcommand{\cL}{\mathcal{L}}
\newcommand{\cK}{\mathcal{K}}
\newcommand{\sR}{\mathscr{R}}
\newcommand{\tin}{\quad \text{in }}
\newcommand{\ton}{\quad \text{on }}

\renewcommand{\d}{\mathrm{d}}
\newcommand{\dx}{\,\d x}
\newcommand{\dt}{\,\d t}
\newcommand{\dxdt}{\,\d x \d t}
\newcommand{\dtau}{\,\d \tau}
\newcommand{\ddt}{\frac{\d}{\d t}}
\newcommand{\ptial}[1]{ \partial_{#1} }
\newcommand{\pt}{\ptial{t}}
\newcommand{\onehalf}{\frac{1}{2}}

\newcommand{\tr}{\mathrm{Tr}}

\newcommand{\tran}[1]{ #1^{\top}}
\newcommand{\inv}[1]{ #1^{-1}}

\newcommand{\abs}[1]{\left\vert #1 \right \vert}

\newcommand{\norm}[1]{\left\Vert #1 \right \Vert}
\newcommand{\normm}[1]{\Vert #1 \Vert}

\newcommand{\inner}[2]{\left\langle #1 , #2 \right\rangle}

\newcommand{\innerr}[2]{\left( #1 , #2 \right)}

\newcommand{\arxiv}[1]{arXiv: \href{https://arxiv.org/abs/#1}{#1}}

\DeclareMathOperator*{\Div}{\mathrm{div}}

\allowdisplaybreaks[4]

\begin{document}

 \title[Two-phase viscoelastic fluids]{On a diffuse interface model for incompressible viscoelastic two-phase flows}

    \author[Y. Liu \& D. Trautwein]{
            \small
            Yadong Liu$^{\ast}$$^\dagger$ and 
            Dennis Trautwein$^\dagger$
            }
        \address{
	   $^\ast$School of Mathematical Sciences, 
	   Nanjing Normal University, 
	   Nanjing 210023, P.~R.~China
        }
        \email{ydliu@njnu.edu.cn (Yadong.Liu@ur.de)}
        \address{
            $^\dagger$Fakult\"at f\"ur Mathematik, Universit\"at Regensburg, 93053 Regensburg, Germany
        }
        
        \email{Dennis.Trautwein@ur.de}

	\date{\today}
	

	\begin{abstract}
		This paper concerns a diffuse interface model for the flow of two incompressible viscoelastic fluids in a bounded domain.
		More specifically, the fluids are assumed to be macroscopically immiscible, but with a small transition region, where the two components are partially mixed. Considering the elasticity of both components, one ends up with a coupled Oldroyd-B/Cahn--Hilliard type system, which describes the behavior of two-phase viscoelastic fluids. 
		We prove the existence of weak solutions to the system in two dimensions for general (unmatched) mass densities, variable viscosities, different shear moduli, and a class of physically relevant and singular free energy densities that guarantee that the order parameter stays in the physically reasonable interval. The proof relies on a combination of a regularization of the original system and a new hybrid implicit time discretization for the regularized system together with the analysis of an Oldroyd-B type equation.	
	\end{abstract}
	
	\maketitle
	
% 	\setcounter{tocdepth}{1}
% 	\tableofcontents
	
\section{Introduction}
    \label{sec:introduction}
	In this article, we study a so-called \textit{diffuse interface model} (also called \textit{phase field model}) for two incompressible, viscoelastic fluids of different mass densities, viscosities and shear moduli. In the model, a partial mixing of the macroscopically immiscible fluids is considered and elastic effects are taken into account. 
	
	This model is quite new and was developed recently in Mokbel--Abels--Aland \cite{MAA2018}, where they proposed a novel phase-field model for a fluid-structure interaction problem to handle very large deformations as well as topology changes like the contact of a solid to a wall. Under certain assumptions on the system parameters, the model is {capable of describing} a thermodynamically consistent, frame indifferent,
% 	\footnote{Is it still frame indifferent with r.h.s.~in \eqref{eqs:CH-q} ? \\
% 	Yes I think \eqref{eqs:CH-q} is frame indifferent. Since the left Cauchy--Green tensor $\bbb$ is frame invariant, i.e.,
% 	\begin{equation*}
% 	    \bbb_r = \bbf_r \tran{\bbf_r} = \bb{Q} \bbf \tran{\bbf} \tran{\bb{Q}} = \bb{Q} \bbb \tran{\bb{Q}}.
% 	\end{equation*}
% 	Below the Remark 2.2 in [AGG '12], there is a sentence that $\frac{\partial \rho}{\partial \phi} \frac{\abs{\bu}^2}{2}$ is not objective. This is because the velocity is in general a non-objective vector. The velocity is objective only in some specific case. \\
% 	I think probably the main problem could be \eqref{eqs:fluid-B}.
%  \\ Thank you!
% 	}
	incompressible two-phase flow with viscoelasticity of Oldroyd-B type.
% 	\textbf{thermodynamically consistent, frame indifferent, incompressible two-phase flow with viscoelasticity of Oldroyd-B type; \underline{without} dissipation $+\Delta \mathbb{B}$ as it is a regularization for the analysis, but can also be included in the derivation due to nonlocal storage of energy or nonlocal entropy production mechanisms, see [Malek, Prusa, Skrivan, Süli '18]}
% 	\footnote{Here: Model without $+\Delta \bbb$, as it was derived so in [Mokbel, Abels, Aland] from physics and is thermodyn. consistent, frame indifferent, incompressible, ...\\Sounds good.}

	Let $ T > 0 $, $ Q_T \coloneqq \Omega \times (0,T) $ with $ \Omega \subset \bbr^d $, $d\in\{2,3\}$, a sufficiently smooth bounded domain and~$ {S_T \coloneqq \partial \Omega \times (0,T) }$. We consider the following system of Oldroyd-B/Cahn--Hilliard type:
	\begin{subequations}
		\label{eqs:Model}
		\begin{alignat}{3}
			\label{eqs:fluid-momentum}
			\begin{split}
				\pt (\rho(\phi) \bu) + \Div ( \rho(\phi) \bu \otimes \bu ) & + \Div ( \bu \otimes \bJ ) + \nabla \pi \\
				- \Div \big( \bbs(\nabla \bu, \bbb, \phi) \big) & = {q \nabla \phi + \frac{\mu(\phi)}{2} \nabla \tr(\bbb - \ln \bbb - \bbi)}
			\end{split} && \tin Q_T, \\
			\Div \bu & = 0 && \tin Q_T, \\
			\label{eqs:fluid-B}
			\pt \bbb + \bu \cdot \nabla \bbb + \frac{\alpha(\phi)}{\lambda(\phi)} (\bbb - \bbi) & = \bbb \tran{\nabla} \bu + \nabla \bu \bbb
			+{ \frac{\kappa}{\mu(\phi)} \Delta \bbb} 
			&& \tin Q_T, \\
			\label{eqs:CH-phi}
			\pt \phi + \bu \cdot \nabla \phi & = \Div (m(\phi) \nabla q) && \tin Q_T, \\
			\label{eqs:CH-q}
			q - \tilde{\sigma} \Big( \frac{1}{\epsilon} W'(\phi) - \epsilon \Delta \phi \Big) & = \frac{\mu'(\phi)}{2} \tr(\bbb - \ln \bbb - \bbi) && \tin Q_T,
		\end{alignat}
    {where $\mathrm{Tr} \, \ln \mathbb{B} = \ln \det \mathbb{B}$ for positive definite matrices}. 
    %\todo{removed `by applying ...}
    %by applying the logarithmic function to the spectrum of $\mathbb{B}$. 
    Moreover, $ \bJ $ denotes the relative mass flux associated with the diffusion of the mixture components given by
	\begin{equation*}
	    \bJ \coloneqq - \rho'(\phi) m(\phi) \nabla q
	\end{equation*}
	and the stress tensor $ \bbs $ is defined by
	\begin{equation*}
		\bbs(\nabla \bu, \bbb, \phi) \coloneqq \nu(\phi)(\nabla \bu + \tran{\nabla} \bu) + \mu(\phi)(\bbb - \bbi).
	\end{equation*}
	Let $\phi_i$ be the volume fraction of fluid $i$, $i \in \{1,2\}$. Define $ \phi \coloneqq \phi_2 - \phi_1 $ as the order parameter related to the concentrations of the two fluids. {Namely, the values $\phi = \pm1$ indicate the unmixed ``pure'' phases of the fluid}. %\todo{Changed.}
 Based on the order parameter~${\phi, \bu}$ and~$ \rho(\phi) $ are the unknown (volume-averaged) velocity and the density of the mixture of the two fluids given by
	\begin{equation*}
	    \bu \coloneqq \frac{1 - \phi}{2} \bu_1 + \frac{1 + \phi}{2} \bu_2, \quad 
	    \rho(\phi) \coloneqq \frac{1 - \phi}{2} \rho_1 + \frac{1 + \phi}{2} \rho_2,
	   % \frac{\rho_2 - \rho_1}{2} \phi + \frac{\rho_1 + \rho_2}{2},
	\end{equation*}
	where $ \bu_i $, $ \rho_i $, $i \in \{1,2\}$ are the specific velocities and densities of fluid $i$.
	Moreover, {the tensor $ \bbb $ is linked to the macroscopic deformation of the viscoelastic fluid. In the context of multiple virtual configurations, the tensor $\bbb$ can be interpreted as the left Cauchy--Green tensor associated with the elastic part of the full deformation, cf. Appendix \ref{sec:derivation}}. Here $ \pi $ is the pressure, and~$ q $ denotes the chemical potential associated to $ \phi $. In system \eqref{eqs:Model}, $ \nu(\phi) > 0 $ denotes the viscosity coefficient, $ m(\phi) > 0 $ is a (non-degenerate) mobility coefficient, $ \mu(\phi) > 0 $ is the shear {modulus}, {and $\kappa>0$ is constant related to the stress diffusion}. The ratio $ \lambda(\phi)/\alpha(\phi) > 0 $ refers to the relaxation time of elasticity, which is supposed to be phase-dependent in the case of fluid-structure interaction \cite{MAA2018}.
 % Here we consider the two-phase viscoelastic flow and, for simplicity, we assume $ \alpha, \lambda > 0 $ to be constant. 
	Furthermore,~$ W(\phi) $ is the homogeneous free energy density for the mixture, which is of double-well type. One of the typical examples is the logarithmic potential
	\begin{equation*}
	    W(\phi) = \frac{\theta}{2}\big((1 + \phi) \ln(1 + \phi) + (1 - \phi) \ln(1 - \phi)\big) - \frac{\theta_c}{2} \phi^2,
	\end{equation*}
	defined in $ [-1, 1]$, which leads to a physically relevant value $\phi \in [-1,1]$, with $ 0 < \theta < \theta_c $ being the absolute temperature and the critical temperature of the mixture. 
	The system is closed with the boundary and initial conditions
	\begin{alignat}{3}
		\bu & = \mathbf{0} && \ton S_T, \\ 
		\ptial{\bn} \phi = \ptial{\bn} q & = 0 && \ton S_T, \\
        \ptial{\bn} \bbb & = \mathbb{O} && \ton S_T, \\ 
		(\bu, \bbb, \phi)(0) & = (\bu_0, \bbb_0, \phi_0) && \tin \Omega,
	\end{alignat}
	where $ \partial_\bn \coloneqq \bn \cdot \nabla $ and $ \bn $ denotes the exterior normal at $ \partial \Omega $. {Moreover, $\mathbf{0}\in\bbr^d$ and $\mathbb{O}\in\bbr^{d\times d}$ are the zero vector and zero matrix, respectively.} %\todo{added this sentence for future reference}
 %Note that \eqref{eqs:fluid-B} is a transport equation for which no further boundary condition is required, as $\bn \cdot\bu|_{\partial\Omega}=0$ ensures that characteristics do not enter the domain.
	\end{subequations}

    {
    % \begin{itemize}
    %     \item add goal: existence of weak solutions
    %     \item ``To further understand the analysis, we present an energy identity.''  (or something like that)
    % \end{itemize}
    The goal of this contribution is to investigate the global existence of weak solutions %in suitable senses 
    in two dimensions.} 
	Before stating the main results, let us mention that the total energy of the system \eqref{eqs:Model} denoted by $ \cE(t) $ consists of the kinetic energy, the {elastic energy} %\todo{Changed entropy to energy} 
 and the Ginzburg--Landau free energy as
	\begin{equation}
		\label{eqs:energy}
		\cE(\bu, \phi, \bbb) \coloneqq \underbrace{\int_\Omega \frac{\rho(\phi)}{2} \abs{\bu}^2 \dx}_{\text{Kinetic energy}} 
		+ \underbrace{\int_\Omega \frac{\mu(\phi)}{2} \tr(\bbb - \ln \bbb - \bbi) \dx}_{\text{Elastic energy}} 
		+ \underbrace{\int_\Omega \tilde{\sigma} \Big( \frac{\epsilon}{2} \abs{\nabla \phi}^2 + \frac{1}{\epsilon} W(\phi) \Big) \dx}_{\text{Free energy}},
	\end{equation}
    where $\tilde{\sigma} > 0$ is related to the surface tension at the interface and $ \epsilon > 0 $ corresponds to the thickness of the interface.
	Moreover, every sufficiently smooth solution of \eqref{eqs:Model} satisfies the energy dissipation differential inequality
    %\dennis{equality only with $-\Delta \bbb : \Delta \bbb^{-1}$, due to $... \geq \frac{1}{d} \abs{\nabla\tr\ln\bbb}^2$}
	\begin{equation}
    \label{eqs:energy-dissipation-structure}
    \begin{aligned}
        \ddt \cE(\bu, \phi, \bbb)(t) 
        & = - \int_\Omega \frac{\nu(\phi)}{2} \abs{\nabla \bu + \tran{\nabla} \bu}^2 \dx
        {+ \frac{\kappa}{2} \int_\Omega \nabla \bbb : \nabla \bbb^{-1} \dx }
        %{- \frac{\kappa}{2d} \int_\Omega \abs{\nabla \tr \ln \bbb}^2 \dx }
        \\
        & \quad - \int_{\Omega \setminus \{\lambda = 0\}} \frac{\mu(\phi)\alpha(\phi)}{2 \lambda(\phi)} \tr(\bbb + \inv{\bbb} - 2 \bbi) \dx
        - \int_\Omega m(\phi) \abs{\nabla q}^2 \dx
        \leq 0,
    \end{aligned}
	\end{equation}
    {where $-\nabla \bbb : \nabla \bbb^{-1} \geq \frac{1}{d} \abs{\nabla\tr\ln\bbb}^2$ {is non-negative}, see \cite[Lemma 3.1]{BLS2017}.} %\todo{yes, we keep this here}
	This is referred to \cite{MAA2018} with $\kappa =0$, where the model was derived via local energy dissipation laws.  
	As they did not provide a detailed comprehension of the elasticity, we give a complete derivation with a more general energy density in Appendix \ref{sec:derivation}.
	Note that the model \eqref{eqs:Model} can be seen as a viscoelastic fluid counterpart of the celebrated Abels--Garcke--Gr\"un (AGG) model \cite{AGG2012} where a diffuse interface model was proposed for a two-phase flow of two incompressible fluids with different densities using methods from rational continuum mechanics, which satisfies local and global dissipation inequalities and is frame indifferent.

 % \todo{Put remarks after the main theorem}

    For the purpose of readability, in the following we will let~${ \alpha(\cdot) = \lambda(\cdot) = \epsilon = \tilde{\sigma} = 1 }$ since they have no significant contribution to the analysis.
    
\subsection{State of the Art} 
\label{sec:state-of-the-art}
Over the last decades mathematical analysis of fluid dynamics has been developed with an abundant amount of literature. In particular, the problem of two-phase fluids with a free interface fascinated the profound attention of mathematicians. However, the free interface is not easy to track in actual applications and mathematicians still do not have a thorough comprehension of the problems concerning singularities of interfaces and topological changes during the evolution of interfaces. Thus, an effective approximation of the interface has been introduced with the Ginzburg--Landau free energy since, e.g., \cite{GPV1996,HH1977}, where the authors proposed a so-called ``model H'' by the phase-field method in terms of the phase variable that indicates the specific phase of the whole system. Later on, this model was investigated and employed explosively in many areas, especially in numerical simulations, since the model has the advantage that topological changes can happen. However, model H was endowed with a basic assumption that the densities of both components are the same, which is not always physically reasonable. To overcome this disadvantage, Abels--Garcke--Gr\"un \cite{AGG2012} derived the ``AGG model'', which considers unmatched densities and is thermodynamically consistent and frame indifferent. For the analysis of weak solutions and strong solutions to the AGG model, we refer to \cite{ADG2013,ADG2013degenerate,AGG2022,Giorgini2021}. An alternative and thermodynamically consistent model was derived before by Lowengrub--Truskinovski \cite{LT1998} by using the mass-averaged (barycentric) velocity. From the mathematical viewpoint, the latter has the disadvantage that the mass-averaged velocity is not divergence free, which is the case in the AGG model based on a volume-averaged velocity. In view of the physical advantages of the AGG model, there came many variants subsequently concerning different aspects. Readers are referred, for example, to \cite{Frigeri2016} for the nonlocal Cahn--Hilliard--Navier--Stokes equations with unmatched densities, to \cite{GGW2019,GK2022} for diffuse interface models including moving contact lines, to \cite{Sieber2020} for polymeric fluids, to \cite{KMS2021} for magnetohydrodynamic two-phase flow with different densities, and to \cite{MAA2018} for a diffuse interface model simulating a fluid-structure interaction problem with viscoelasticity of Oldroyd-B type, which is exactly of our interest.

Now, let us recall some facts about viscoelasticity, in particular of Oldroyd-B type. Depending on the deformation gradient $\mathbf{F}$ that is defined in Eulerian coordinates or the corresponding left Cauchy--Green tensor $\bbb \coloneqq \mathbf{F} \mathbf{F}^\top$, viscoelasticity can be recognized as two regimes. 
{Note that this definition of $\bbb$ here differs slightly from the one used in the current context. It corresponds to the definition of the model \eqref{eqs:Model} with an infinite relaxation time $\lambda(\phi)$.}
In the context of $\mathbf{F}$, the incompressible viscoelastic fluid model reads as
\begin{equation}
    \label{eqs:OldroydB-F}
    \left\{
        \begin{aligned}
            & \pt \bu + \bu \cdot \nabla \bu - \mu \Delta \bu + \nabla p = \Div (\mathbf{F} \mathbf{F}^\top), \\
            & \pt \mathbf{F} + \bu \cdot \nabla \mathbf{F} = \nabla \bu \mathbf{F}, \\
            & \Div \bu = 0,
        \end{aligned}
    \right.
\end{equation}
whose derivation is referred to, e.g., Giga--Kirshtein--Liu \cite{GKL2018} by the energy variational approach. The analysis of \eqref{eqs:OldroydB-F} has been quite popular in the last decade, and the global well-posedness and long time behavior have been established in, e.g., \cite{HW2015,LeiLZ2008,LinLZ2005,LZ2008} and their citations. However, due to the lack of compactness for the deformation gradient, problems concerning the existence of weak solutions are still open even in two dimensions. We refer to two recent progressive papers by Hu--Lin \cite{HL2016Lp} for the existence of 2D small weak solutions in $ L^p $, where the so-called \textit{effective viscous flux} method was employed. Based on this viscoelastic approach, some very interesting models that describe the motion of two-phase flows have been introduced in very recent years. Readers are referred to Agosti--Colli--Garcke--Rocca \cite{ACGR2022} for a phase-field model coupled with viscoelasticity with large deformations, and Kim--Tawri--Temam \cite{KTT2022} for a diffuse interface model describing the interaction between blood flow and a thrombus with Hookean elasticity during the stage of atherosclerotic lesion in human artery.

When it comes to the tensor $\bbb$, %$\bbb=\mathbf{F} \tran{\mathbf{F}}$, 
with relaxation and $\btau \coloneqq \bbb - \bbi$, we obtain the original Oldroyd-B model as
\begin{equation}
    \label{eqs:OldroydB-B}
    \left\{
        \begin{aligned}
            & \pt \bu + \bu \cdot \nabla \bu - \mu \Delta \bu + \nabla p = \Div \btau, \\
            & \pt \btau + \bu \cdot \nabla \btau - \nabla \bu \btau - \btau \nabla^\top \bu + a \btau = \frac{b}{2} (\nabla \bu + \tran{\nabla} \bu), \\
            & \Div \bu = 0,
        \end{aligned}
    \right.
\end{equation}
where $a,b \geq 0$ and $1/a$ denotes the relaxation time and $b/a$ is the viscosity of the polymers. For a derivation, we refer to \cite{MP2018,Oldroyd1950,RT2021}. It can also be recovered from the micro-macroscopic FENE (Finite Extensible Nonlinear Elastic) dumbbell model, see, e.g., \cite{GKL2018}. Concerning the analysis of \eqref{eqs:OldroydB-B}, the first existence result of weak solutions goes back to Lions--Masmoudi \cite{LM2000}, where they considered a co-rotational version of the stress tensor equation to get better estimates and compactness. Due to lack of compactness, the existence of weak solutions to \eqref{eqs:OldroydB-B} has been an open problem for a long time. We mention a breakthrough on a related model, for which Masmoudi \cite{Masmoudi2013} proved the existence of weak solutions to the FENE dumbbell model with many weak convergence techniques, based on the control of the propagation of strong convergence of some well-chosen quantity by studying a transport
equation for its defect measure. Instead of solving \eqref{eqs:OldroydB-B} directly, in recent years mathematicians tried to include some regularization terms for the Oldroyd-B equation, for example the diffusive stress $\Delta \btau$, in order to obtain higher regularity and compactness, which is however still limited to two dimensions. Barrett--Boyaval \cite{BB2011} proved the existence of weak solutions to a regularized Oldroyd-B model with stress diffusion in 2D by employing an approximating finite element scheme, together with an entropy regularization, while the global regularity of 2D diffusive Oldroyd-B equations was obtained by Constantin--Kliegl \cite{CK2012}. Note that in \cite{BB2011}, the authors reformulated an equivalent system in terms of the {tensor} $\bbb$ instead of $\btau$, and made use of a physically relevant elastic energy $\tr(\bbb - \ln \bbb - \bbi)$ and its estimate, which was derived before by Hu--Leli\`evre \cite{HL2007}. One of the advantages of the formulation with respect to $\bbb$ is that one can expect the positive definiteness of it, while in the other formulation this is not clear at all.
Later, a compressible counterpart of the (diffusive) Oldroyd-B model was proposed and analyzed in \cite{BLS2017} by means of a combination of multilayer approximations from \cite{BB2011} and the compressible Navier--Stokes equation \cite{FN2017}, based on the Galerkin scheme. For other related viscoelastic models, one refers to \cite{BLL2022,LMNR2017} for Peterlin type models, and to \cite{BBM2021} for a mixing of the Oldroyd-B model and the Giesekus model. {A recent existence result for the three dimensional Giesekus model can be found in \cite{BLM2024}.}

Along with the development of diffuse interface models, two-phase viscoelasticity came up in recent years. Gr\"un--Metzger \cite{GM2016} derived a micro-macro model for two-phase flow of dilute polymeric solutions, where a Fokker–Planck type equation describes orientation and elongation of polymers. Sieber \cite{Sieber2020} then extended the model in \cite{GM2016} to the case with different mass densities for polymeric fluids. Moreover, a viscoelastic phase separation model of Peterlin type was derived in \cite{B+2021} and investigated in \cite{BL2022a,BL2022b}. By virtue of the Oldroyd-B model, recently Garcke--Kov\'acs--Trautwein \cite{GKT2022} established a viscoelastic Cahn--Hilliard model to describe tumor growth and obtained the existence of weak solutions in the case of matched mass densities with the help of finite element approximations. 

In contrast to the growing literature on two-phase flows and viscoelastic fluids, the analytical results about two-phase viscoelasticity are quite limited, especially the Oldroyd-B type models with unmatched densities and variable shear moduli. 
To the best of our knowledge, the only related results of diffuse interface model including viscoelasticity of Oldroyd-B type are \cite{GKT2022,Sieber2020} we mentioned above, where they showed the existence of weak solutions for 2D two-phase Oldroyd-B type fluids with %constant viscosities\footnote{} and 
polynomial potential, while none of them concerned the case of unmatched densities and phase-depending shear moduli, which in fact, is rather physical and interesting. The aim of the present paper is to provide a deep understanding on the theory of incompressible two-phase flow with viscoelasticity of Oldroyd-B type with more physical assumptions, in company with the AGG model \cite{ADG2013,AGG2012}. More precisely, we are going to prove the existence of weak solutions to the incompressible Oldroyd-B/Cahn--Hilliard model in the presence of variable densities, shear moduli, viscosities, and a singular potential.

% \subsection{Reformulation with a Modified Pressure and Stress Diffusion}

% The full system of our interest now reads
% 	\begin{subequations}
% 		\label{eqs:Model-modified}
% 		\begin{alignat}{3}
% 			\label{eqs:modified-fluid-momentum}
% 			\begin{split}
% 				\pt (\rho(\phi) \bu) + \Div ( \rho(\phi) \bu \otimes \bu ) & + \Div ( \bu \otimes \bJ ) + \nabla \pi \\
% 				- \Div \big( \bbs(\nabla \bu, \bbb, \phi) \big) & = q \nabla \phi + \frac{\mu(\phi)}{2} \nabla \tr(\bbb - \ln \bbb - \bbi)
% 			\end{split} && \tin Q_T, \\
% 			\Div \bu & = 0 && \tin Q_T, \\
% 			\label{eqs:modified-fluid-B}
% 			\pt \bbb + \bu \cdot \nabla \bbb + (\bbb - \bbi) & = \bbb \tran{\nabla} \bu + \nabla \bu \bbb + \frac{\kappa}{\mu(\phi)} \Delta \bbb && \tin Q_T, \\
% 			\pt \phi + \bu \cdot \nabla \phi & = \Div (m(\phi) \nabla q) && \tin Q_T, \\
% 			q - W'(\phi) + \Delta \phi & = \frac{\mu'(\phi)}{2} \tr(\bbb - \ln \bbb - \bbi) && \tin Q_T, \\
% 			\bu = 0,\ \ptial{\bn} \bbb & = \mathbf{0} && \ton S_T, \\ 
% 			\ptial{\bn} \phi = \ptial{\bn} q & = 0 && \ton S_T, \\
% 			(\bu, \bbb, \phi)(0) & = (\bu_0, \bbb_0, \phi_0) && \tin \Omega.
% 		\end{alignat}
% 	\end{subequations}

\subsection{Main results}
\label{sec:main-result}
	In this section, we state the main result of this manuscript. Namely, the global existence of weak solutions to \eqref{eqs:Model} in two dimensions, which describes the incompressible viscoelastic two-phase flow with general (unmatched) mass density, variable viscosity, different shear moduli and a singular free energy density in two dimensions. We note that the notation will be explained in the beginning of Section \ref{sec:preliminaries}.
    In the following we summarize the assumptions that are necessary to formulate the notation of a weak solution.
{
\begin{assumption}
	\label{ass:weak}
	We assume that $ \Omega \subset \bbr^2 $ is a bounded domain with smooth boundary. Moreover, we impose the following conditions.
	\begin{enumerate}[label=\textbf{(H\arabic*)}] 
		\item \label{ass:rho} {The density of the model is given by $ \rho(\phi) = \onehalf(\rho_1 + \rho_2) + \onehalf(\rho_2 - \rho_1) \phi $}. 
  %\todo{removed 'which is derived in ...'}
  %which is derived in \cite{AGG2012,MAA2018}. 
  Here $ \rho_i > 0 $ denote the constant unmatched densities of two fluids and $ \phi $ is the difference of the volume fractions of the fluids.
		\item \label{ass:lower-upper-bounds} We assume $ m, \mu \in C^1(\bbr) $, $ \nu \in C(\bbr) $ and they have the corresponding constant lower and upper bounds, i.e., $ 0 < \underline{m} \leq m \leq \overline{m} $, $ 0 < \underline{\nu} \leq \nu \leq \overline{\nu} $, $ 0 < \underline{\mu} \leq \mu \leq \overline{\mu} $ and $ 0 < \underline{\mu}' \leq \mu' \leq \overline{\mu}' $.
		\item \label{ass:free-energy} The free energy density is assumed to be a general function $ W \in C([-1,1]) \cap C^2((-1,1)) $ that satisfies
		\begin{equation*}
			\lim_{s \rightarrow -1} W'(s) = - \infty, \quad
			\lim_{s \rightarrow 1} W'(s) = + \infty, \quad
			W''(s) \geq - \omega \text{ for some } \omega \in \bbr.
		\end{equation*}
	\end{enumerate}
\end{assumption}
% \begin{remark}
%     The Assumption \ref{ass:free-energy} allows for a nonconvex potential $W$, which has the domain of definition $[-1,1]$. {Then, a finite value of $W(\phi)$ automatically induces a physically relevant $\phi \in (-1,1)$, provided that $\abs{W'(\phi)} < \infty$.} %\todo{Change formulation}
%     \yadong{Or just simply delete the remark?}
%     \dennis{Yes delete it}
%     One typical example is the so-called logarithmic potential as given in Section \ref{sec:introduction}.
% \end{remark}
Moreover, we record the definition of weak solutions to \eqref{eqs:Model}.}
{
\begin{definition}
\label{def:model}
Let $ T > 0 $ and $ (\bu_0, \bbb_0, \phi_0) \in L_\sigma^2(\Omega) \times L^2(\Omega; \bbr_{\mathrm{sym}}^{2 \times 2}) \times W^{1,2}(\Omega) $ with $ \bbb_0 $ positive definite, $\tr \ln \bbb_0 \in L^1(\Omega)$ and $ \abs{\phi_0} \leq 1 $ almost everywhere in $ \Omega $. In addition, let Assumption \ref{ass:weak} hold true.
We call the quadruple $ (\bu, \phi, q, \bbb) $ a \textit{finite energy} weak solution to \eqref{eqs:Model} with initial data $ (\bu_0, \bbb_0, \phi_0) $, provided that
\begin{enumerate}
\item the quadruple $ (\bu, \phi, q, \bbb) $ satisfies 
% 		\footnote{If we should add $ \abs{\phi} < 1 $, a.e.~in $Q_T$?\\@y: We should add this.Thank you!}
\begin{gather*}
\bu \in C_w([0,T]; L_\sigma^2(\Omega)) 
\cap L^2(0,T; W_0^{1,2}(\Omega; \bbr^2)); \\
%			\cap W^{1,\frac{4}{3}}(0,T; [W_0^{1,2}(\Omega; \bbr^2)]'); \\
\phi \in C_w([0,T]; W^{1,2}(\Omega)) 
\cap L^2(0,T; W^{2,2}(\Omega)) \text{ with } \phi\in(-1,1) \text{ a.e.~in } Q_T; \\
W'(\phi) \in L^2(0,T; L^2(\Omega)), \quad
%			\cap W^{1,2}(0,T; [W_0^{1,2}(\Omega)]'); \\
q \in L^2(0,T; W^{1,2}(\Omega)); \\
\bbb \text{ is symmetric positive definite a.e.~in } Q_T; \\
\bbb \in C_w([0,T]; L^2(\Omega; \bbr_{\mathrm{sym}}^{2 \times 2})) 
\cap L^2(0,T; W^{1,2}(\Omega; \bbr_{\mathrm{sym}}^{2 \times 2})); \\
\tr \ln \bbb \in L^\infty(0,T; L^1(\Omega)) 
\cap L^2(0,T; W^{1,2}(\Omega)); 
%			\cap W^{1,2}(0,T; [W_0^{1,2}(\Omega; \bbr_{\mathrm{sym}}^{2 \times 2})]');
\end{gather*}
\item for all $ t \in (0,T) $ and all $ \bw \in C^\infty([0,T]; C_0^\infty(\Omega; \bbr^2)) $ with $ \Div \bw = 0 $, we have
%\footnote{I changed $\d t$ to $\d\tau$, to not mix up with $t\in(0,T)$\\Thank you! That is my mistake.}
\begin{align}
\int_0^t \int_\Omega & \Big( \rho(\phi) \bu \cdot \pt \bw 
+ (\rho(\phi) \bu \otimes \bu) : \nabla \bw 
- \rho'(\phi) (\bu \otimes m(\phi) \nabla q) : \nabla \bw \Big) \dx \dtau \nonumber\\
& \quad - \int_0^t \int_\Omega \Big( 
\nu(\phi) (\nabla \bu + \tran{\nabla} \bu) : \nabla \bw 
+ \mu(\phi) (\bbb - \bbi) : \nabla \bw
\Big) \dx \dtau 
\label{eqs:weak-full-u-formulation} \\
& = - \int_0^t \int_\Omega q \nabla \phi \cdot \bw \dx \dtau
- \int_0^t \int_\Omega \frac{\mu(\phi)}{2} \nabla \tr(\bbb - \ln \bbb - \bbi) \cdot \bw \dx \dtau \nonumber \\
& \quad + \int_\Omega \rho(\phi(\cdot, t)) \bu(\cdot, t) \cdot \bw(\cdot, t) \dx
- \int_\Omega \rho(\phi_0) \bu_0 \cdot \bw(\cdot, 0) \dx; \nonumber
\end{align}
\item for all $ t \in (0,T) $ and all $ \xi \in C^\infty([0,T]; C^1(\overline{\Omega}; \bbr)) $, we have
\begin{align}
\int_0^t \int_\Omega & \phi \big(  \pt \xi 
+ \bu \cdot \nabla \xi \big) \dx \dtau
- \int_0^t \int_\Omega m(\phi) \nabla q \cdot \nabla \xi \dx \dtau 
\label{eqs:weak-full-phi-formulation} \\
& = \int_\Omega \phi(\cdot, t) \xi(\cdot, t) \dx
- \int_\Omega \phi_0 \xi(\cdot, 0) \dx; \nonumber
\end{align}
\item for a.e.~$ (x,t) \in Q_T $, we have
\begin{equation*}
q = W'(\phi) - \Delta \phi
+ \frac{\mu'(\phi)}{2} \tr(\bbb - \ln \bbb - \bbi);
\end{equation*}
\item for all $ t \in (0,T) $ and all $ \bbc \in C^\infty(\overline{Q_T}; \bbr_{\mathrm{sym}}^{2 \times 2}) $, we have
\begin{align}
\int_0^t \int_\Omega & \Big( \bbb : \pt \bbc 
+ (\bu \otimes \bbb) : \nabla \bbc \Big) \dx \dtau \nonumber\\
& \quad + \int_0^t \int_\Omega \Big(\big(\nabla \bu \bbb + \bbb \tran{\nabla} \bu \big) : \bbc 
- \kappa \nabla \bbb : \nabla \frac{\bbc}{\mu(\phi)} \Big) \dx \dtau 
\label{eqs:weak-full-B-formulation} \\
& = \int_0^t \int_\Omega (\bbb : \bbc - \tr \bbc) \dx \dtau
+ \int_\Omega \bbb(\cdot, t) : \bbc(\cdot, t) \dx
- \int_\Omega \bbb_0 : \bbc(\cdot, 0) \dx; \nonumber
\end{align}
\item for a.e.~$ t \in (0,T) $, the following energy estimate holds
\begin{align}
\cE(t)
& + \onehalf \int_0^t \norm{\sqrt{\nu(\phi)}(\nabla \bu + \tran{\nabla} \bu)(\tau)}_{L^2}^2 \dtau
\nonumber \\
& 
+ \int_0^t \bigg(\norm{\frac{\mu(\phi)}{2} \tr(\bbb
+ \inv{\bbb} - 2\bbi)(\tau)}_{L^1} 
+ \frac{\kappa}{2} \norm{\nabla \tr \ln \bbb}_{L^2}^2\bigg) \dtau 
\label{eqs:energy-dissipation-inequality} \\
& 
+ \int_0^t \norm{\sqrt{m(\phi)}\nabla q(\tau)}_{L^2}^2 \dtau \leq {\cE_0},
\nonumber
\end{align}
where { $ \cE(t) = \cE\big(\bu(t), \phi(t), \bbb(t)\big) $ is the energy defined in \eqref{eqs:energy} and $ \cE_0 \coloneqq \cE(\phi_0,\bu_0,\bbb_0) $}.
%\todo{Define $\cE(\phi,\bu,\bbb)$ instead of $\cE(t)$.}
%			\begin{equation*}
%				\cE(t) = \int_\Omega \frac{\rho(\phi)}{2} \abs{\bu(t)}^2(t) \dx
%				+ \int_\Omega \frac{\mu(\phi)}{2} \tr(\bbb - \ln \bbb - \bbi)(t) \dx
%				+ \int_\Omega \frac{\epsilon}{2} \abs{\nabla \phi(t)}^2 + \frac{1}{\epsilon} W(\phi(t)) \dx. 
%			\end{equation*}
\end{enumerate}
\end{definition}
}
    
    {Now, we state the main theorem of this work.}
	\begin{theorem}[Proved in Section \ref{sec:proof-weak-limiting}]
		\label{thm:main}
		Let $d=2$ and let Assumption \ref{ass:weak} hold true. Assume that the initial data satisfy $ (\bu_0, \bbb_0, \phi_0) \in L_\sigma^2(\Omega;\bbr^2) \times L^2(\Omega; \bbr_{\mathrm{sym}}^{2 \times 2}) \times W^{1,2}(\Omega) $ with $ \bbb_0 $ positive definite a.e.~in $\Omega$, $\tr\ln\bbb_0 \in L^1(\Omega)$, and $ \abs{\phi_0} \leq 1 $ a.e.~in $\Omega$ and $ \fint_\Omega \phi_0 \dx \in (-1, 1) $. Then for any $ T \in (0, \infty) $, there exists a weak solution $ (\bu, \bbb, \phi, q) $ of \eqref{eqs:Model} in the sense of Definition \ref{def:model}. {Moreover, the tensor $ \bbb $ fulfills the following estimate}
		\begin{equation*}
		    \norm{\bbb(t)}_{L^2}^2
	        + \int_0^t \norm{\bbb(\tau)}_{L^2}^2 \dtau
	        + \int_0^t \norm{\nabla \bbb(\tau)}_{L^2}^2 \dtau
	        \leq C (\cE_0, \norm{\bbb_0}_{L^2}^2)
		\end{equation*}
		for a.e.~$ t \in (0,T) $.
	\end{theorem}

\begin{remark}[Modified pressure]
    {The unknown function $\pi$ from \eqref{eqs:Model} is a modified pressure and it is related to the `standard' pressure $p$ (cf.~Appendix \ref{sec:derivation}) via the relation
    \begin{align*}
        \pi = p + \frac{\tilde\sigma\epsilon}{2} \abs{\nabla \phi}^2 + \frac{\tilde\sigma}{\epsilon} W(\phi) + \frac{\mu(\phi)}{2} \tr(\bbb - \ln \bbb - \bbi).
    \end{align*}
    This formulation is more beneficial for the analysis of weak solutions.}
    % \dennis{comment out the remainder of this remark?}
        {In fact, from the derivation of \eqref{eqs:Model}, cf.~Appendix \ref{sec:derivation}, we know that the original momentum equation is 
        \begin{align*}
            \pt (\rho(\phi) \bu) & + \Div ( \rho(\phi) \bu \otimes \bu ) + \Div ( \bu \otimes \bJ ) + \nabla p \\
			& - \Div \big(\bbs(\nabla \bu, \bbb, \phi)\big) = - \epsilon \tilde{\sigma} \Div \big( \nabla \phi \otimes \nabla \phi \big),
        \end{align*}
        where the function $p:Q_T \to \bbr$ is the pressure, $\epsilon \tilde{\sigma} \Div \big( \nabla \phi \otimes \nabla \phi \big)$ is referred to as the capillary force.
        Then by defining a new scalar pressure
        \begin{equation*}
            \pi \coloneqq p + \frac{\tilde\sigma\epsilon}{2} \abs{\nabla \phi}^2 + \frac{\tilde\sigma}{\epsilon} W(\phi) + \frac{\mu(\phi)}{2} \tr(\bbb - \ln \bbb - \bbi),
        \end{equation*}
        we obtain
        \begin{equation*}
		  \nabla \pi = \nabla p + q \nabla \phi + \tilde\sigma \epsilon \Div \big( \nabla \phi \otimes \nabla \phi \big)
		  + \frac{\mu(\phi)}{2} \nabla \tr(\bbb - \ln \bbb - \bbi).
    	\end{equation*}
     In \cite{BLS2017}, a regularization term $ \alpha \nabla \tr \ln \bbb $ was added in the momentum equation but here we get it for free, since the shear modulus $ \mu(\phi) $ is unmatched among the two phases.
     }
    \end{remark}
    \begin{remark}[Stress diffusion]
        {For {mathematical} reasons, we introduced a diffusive regularization $\frac{\kappa}{\mu(\phi)}\Delta \bbb$ in \eqref{eqs:fluid-B} such that the evolution law of the {tensor} $\bbb$ is of parabolic nature. The system is closed with an additional no-flux boundary condition~${\ptial{\bn} \bbb={\mathbb{O}}}$, see also in \cite{ACGR2022, BB2011, BBM2021, CK2012, GKT2022}. }
        %\todo{changed to zero matrix, see also (1.1)}
        
        {In comparison to, e.g., \cite{BB2011}, the special structure of the factor $ \frac{\kappa}{\mu(\phi)} $ in \eqref{eqs:fluid-B} is of the main importance in the presence of a phase-depending shear modulus~$\mu(\phi)$ to keep the energy dissipation structure in \eqref{eqs:energy-dissipation-structure}.
        In general, the additional diffusive term $ \frac{\kappa}{\mu(\phi)} \Delta \bbb $ can also be motivated by means of a nonlocal energy storage or nonlocal entropy production mechanisms, see M\'{a}lek--Pr\r{u}\v{s}a--Sk\v{r}ivan--S\"uli \cite{MPSS2018}.
        }
    \end{remark} 
\begin{remark}
    In this manuscript, we consider an isotropic free energy and a regular mobility. These are simplifications as our main focus lies on the unmatched densities, viscosities and shear moduli. We comment that more general cases can be obtained by suitable modifications in our current framework. See also Remark \ref{rem:generalcases} for further discussion.
\end{remark}
% \begin{remark}
%     For simplicity of the model, $ \alpha, \lambda $ are assumed to be constant. The existence result also holds true, if $ \alpha, \lambda: \mathbb{R}\to\mathbb{R} $ are continuous, positive and bounded functions that depend on the phase-field variable $ \phi $. \todo{no remark environment, better inclusion in normal block of text}
% \end{remark}
\begin{remark}[Extension to 3D]
{
% Here, the existence of weak solutions holds true only in the two dimensional case, due to the lack of compactness of $\bbb$. If we further consider a modified Helmholtz free energy with the additional term $\int_\Omega \abs{\bbb - \bbi}^2 \dx$, cf. \cite{BBM2021}, from the derivation we would have a quadratic term in the $\bbb$-equation, providing us with more compactness. Hence, one may employ the same argument in this paper to exploit the existence of weak solutions in three dimensions.
{The existence of weak solutions is only valid for the two-dimensional case, primarily due to the lack of compactness of the tensor $\bbb$ in three dimensions. 
Recently, for a simpler model of the incompressible Navier--Stokes equations coupled with the diffusive Oldroyd-B equation, authors of \cite{BBM2021} proved the 3D result by means of a modified Helmholtz free energy, which includes an additional quadratic term, e.g., $\int_\Omega \big( \frac12 \mu \tr(\bbb - \ln\bbb - \bbi) + \frac14 \mu \abs{\bbb - \bbi}^2 \big)\, \mathrm{d}x$. Incorporating this quadratic term improves the compactness properties, as energy estimates directly provide an $L^2$ estimate for $\bbb$ in two and three dimensions. 
We also refer to the recent work \cite{GT2024} that presents an existence result for a related diffuse interface model in three dimensions that includes quadratic contributions to the elastic energy.
This suggests that future work could extend the arguments presented here combined with the strategies from \cite{BBM2021} to establish the existence of weak solutions in a three-dimensional setting.}
}
\end{remark}
%\section{Weak solutions}	
%	\textbf{Strategy}: construct the approximate system with parameter $ \delta > 0 $. 
%	
%	Prove the existence of approximate solution: we already have the result of Abels--Garcke--Gr\"un model with unmatched density [Abels--Depner--Garcke, JMFM '13], which is still valid if one add a given data $ \mathbf{f} \in L_t^2L_x^2 $. Solve the $ \bbb $-equation and using fixed-point argument. Not working! Due to the non-convexity of $ W(\phi) $ we considered and $ \mu(\phi), \alpha(\phi), \lambda(\phi) $ depends on $ \phi $. 
%	
%	Establish the uniform estimate with respect to $ \delta > 0 $.
	
\subsection{Strategy of the Proof (Technical Discussions)}
\label{sec:technical-discussions}
% \footnote{With more details}
Note that the system \eqref{eqs:Model} has a similar structure as the AGG model \cite{AGG2012}, which was solved with the help of an implicit time discretization scheme in \cite{ADG2013}. 
In order to combine this with the existence of weak solutions to an Oldroyd-B fluid system proved in \cite{BB2011}, {we introduce a regularization} to \eqref{eqs:Model} (see $\sR_\eta$ below, also in \eqref{eqs:Model_reg}), such that we can solve the regularized system in a suitable way and have good uniform \textit{a priori} estimates (see Section \ref{sec:a-priori-reg}), inspired by \cite{Abels2009ARMA,BS2018,LR2014}. More specifically, the advantage of the regularization is twofold. Due to the presence of a phase-dependent shear modulus, the regularity of the terms with the coefficient $\mu(\phi)$ will cause problems if one tries the standard testing procedure. For example, multiplying the second term in \eqref{eqs:fluid-B} with $\frac{\mu(\phi)}{2} (\bbi - \inv{\bbb})$ and integrating by parts over $\Omega$ leads to
\begin{align*}
    \int_\Omega \bu \cdot \nabla \bbb : \frac{\mu(\phi)}{2} (\bbi - \inv{\bbb}) \dx
    =  \int_\Omega \bu \cdot \nabla \left(\frac{\mu(\phi)}{2} \tr (\bbb - \ln \bbb - \bbi)\right) \dx \\
    - \int_\Omega \frac{\mu'(\phi)}{2} \bu \cdot \nabla \phi \tr(\bbb - \ln \bbb - \bbi) \dx
\end{align*}
where $ \bu \in L^\infty(0,T; L_\sigma^2(\Omega)) \cap L^2(0,T; W_0^{1,2}(\Omega; \bbr^d)) $. The first term on the right-hand side vanishes for such $\bu$ (as it is solenoidal), while it is not clear if the second integral is well-defined or not under the current regularity setting. %Same as the term of $\pt \bbb$ with time derivative.
With the regularized terms $\mu(\sR_\eta \phi)$ and $\sR_\eta \bu$ in the regularized system, it is no longer an issue since $\sR_\eta \phi$ and $\sR_\eta \bu$ are sufficiently smooth for any $\eta > 0$.
{Moreover, under such regularization, we separate the proof into two parts, focusing on the challenges of the individual subsystems rather than addressing the combined difficulties all at once. Note that this kind of regularization at the same time promotes the solvability of the regularized Oldroyd-B equation, see \eqref{eqs:B-equation} below. }

Concerning the solvability of the regularized Oldroyd-B equation in three dimensions, we make use of the standard Galerkin approximation to account for the parabolic nature. In addition, inspired by \cite{BB2011}, an entropy regularization for $\ln \bbb$ is introduced, which is of importance since the {tensor} $\bbb$ is not necessarily positive definite almost everywhere in $Q_T$ and hence the physical energy $\tr(\bbb - \ln \bbb - \bbi)$ is not necessarily well-defined in presence of the logarithmic term. Thanks to the regularization $\sR_\eta$, we obtain uniform estimates and then pass to the limit for both the entropy regularization and the Galerkin approximation. Finally, we conclude the positive definiteness of the limit function (i.e.~the {tensor} $\bbb$) with a contradiction argument.

To solve the regularized system in 3D, one natural idea would be to employ a time discretization scheme for the full system as in \cite{ADG2013}. However, it is not reasonable to repeat all the regularization and approximation techniques for practicality and readability. Thus, in this manuscript, we propose a novel scheme, we called ``\textit{hybrid time discretization}'' (see Section \ref{sec:time-discretization}), to obtain a suitable approximation of the full system, combined with the solvability of the Oldroyd-B system. Specifically, for the AGG part $(\bu, \phi)$ we keep the implicit time discretization as in \cite{ADG2013}, namely the piecewise constant discrete solution with time-averaged terms regarding $\bbb$, while the Oldroyd-B part is solved in each time interval continuously in time with the time-averaged solution from the previous time step as the initial datum. % and data being constant in time. 
Due to the well-posedness of the Oldroyd-B part, one is able to construct a continuous mapping of $\bbb$ in terms of $(\bu, \phi)$ in the discrete AGG system, which is solvable in a similar fashion as in \cite{ADG2013}. The first essential ingredient of the hybrid time discretization scheme is a uniform estimate, for which we {exploit} the energy estimate of $\bbb$ on a discrete time interval and cancel out all the mixing terms. This is where we use the time-averaged approximation of terms containing $\bbb$, see \eqref{eqs:AGG-energy-discrete-B} and \eqref{eqs:B-identity-discrete}. The second crucial element is the compactness of these time-averaged terms, as the limit passage in these terms requires additional arguments. %Namely, one needs to pass to the limit of these terms, which in general are not so obvious. 
For this, we provide compactness results for weakly-$*$ convergent and weakly convergent sequences, respectively, with the help of a convolution (in time) with Dirac sequences and the Lebesgue differentiation theorem, see Appendix \ref{sec:compactness-B}. Note that one can not apply the Aubin--Lions lemma for the approximate {tensor}, say $\widetilde\bbb^N$, since it is endowed with jumps across the time intervals and one can not expect the existence of $\pt \widetilde\bbb$ over the whole time interval. This is overcome by a compactness argument with time translations, cf. Section \ref{sec:proof-reg}.

The final step is the passage to the limit in the regularized system as the regularization parameter $\eta \to 0$. This can be realized {in two dimensions} with a compactness argument from the concise uniform \textit{a priori} estimate \eqref{eqs:formalEstimate_reg} and a stronger estimate \eqref{eqs:formalEstimate_reg_strong} for $\bbb$, which are both independent of $\eta > 0$. The stronger estimate is necessary since the energy estimate \eqref{eqs:formalEstimate_reg} only provides merely $L^1$ information of $\bbb$. Note that here the stronger estimate for $\bbb$ relies on the Gagliardo--Nirenberg inequality which requires the restriction to two dimensions. Then, in light of the compactness of $\bu$, which is derived with a Helmholtz projection and the convergence of the kinetic energy, one obtains the existence of a weak solution to the original full system \eqref{eqs:Model} by passing to the limit as $\eta \to 0$.

\subsection{Outline}
The rest of the paper unfolds as follows.
In Section \ref{sec:preliminaries}, we explain the notations and provide some auxiliary lemmata. 
In Section \ref{sec:B} we study the tensor-valued equation for the {tensor $\bbb$} with $\eta$-regularization in three dimensions. The well-posedness is of particular significance and is obtained by a Galerkin approximation and an energy regularization. 
Section \ref{sec:regularized} is devoted to the analysis of the regularized system, which is approximated by a hybrid time discretization scheme. Here the discrete solution is solved by a fixed-point argument.
In Section \ref{sec:proof-of-weak-solution} we finish the proof of Theorem \ref{thm:main} (the existence of weak solutions to \eqref{eqs:Model}) by passing to the limit in the regularization, i.e., $\eta \to 0$, together with 
a uniform estimate derived in Section \ref{sec:a-priori-reg} and a stronger estimate for $\bbb$ in Section \ref{sec:a-priori-reg-strong}.
Additionally, in Appendix \ref{sec:derivation} we provide a thermodynamically consistant derivation of the diffuse interface model \eqref{eqs:Model} via local dissipation laws. In particular, we include the derivation of the equation for the {tensor $\bbb$}.

\section{Preliminaries}
\label{sec:preliminaries}
\subsection{Conventions and Auxiliary Results}
% \footnote{We should add the notations here}.
First, the set of real numbers natural numbers (excluding zero) are denoted by $\bbr$ and $\bbn$, and we write $\bbr_+ = \{x > 0\}$ and $\bbn_0 = \bbn \cup \{0\}$. 

%As usual, the letter $C$ in the present paper denotes a generic positive constant which may change its value from line to line, even in the same line, unless we give a special declaration.

\subsubsection{Vector and tensor analysis}
Let $\ptial{j} = \ptial{x_j}$ be the spatial derivative in $x_j$ direction, $j = 1,\dots,d$. 
For a vector $\bu = (\bu_j)_{j = 1}^d \in \bbr^d$, the divergence and gradient are given by $\Div \bu = \sum_{j = 1}^d \ptial{j} \bu_j$, $\nabla \bu = (\ptial{j} \bu^i)_{i,j=1}^d$. Additionally, the transpose of the gradient will be denoted by $\nabla^\top \bu \coloneqq (\nabla \bu)^\top$. Given vectors $\bu = (\bu_j)_{j = 1}^d, \bv = (\bv_j)_{j = 1}^d \in \bbr^d$, $d\in\bbn$, we define
\begin{equation*}
    \bu \cdot \bv = \sum_{i = 1}^d \bu_i \bv_i, \quad 
    (\bu \otimes \bv)_{ij} = \bu_i \bv_j.
\end{equation*}
Given matrices $\bba = (\bba_{ij})_{i,j = 1}^d, \bbb = (\bbb_{ij})_{i,j = 1}^d \in \bbr^{d \times d}$, we know $(\bba \bbb)_{ik} = \sum_{j = 1}^d \bba_{ij} \bbb_{jk}$. Then we denote by
\begin{equation*}
    \bba : \bbb = \tr(\tran{\bbb} \bba) = \sum_{j,k=1}^d \bbb_{jk} \bba_{jk}
\end{equation*}
the Frobenius inner product and $\abs{\bba} \coloneqq \sqrt{\bba : \bba}$ the induced modulus. 
For a matrix $\bba = (\bba_{ij})_{i,j = 1}^d \in \bbr^{d \times d}$, we introduce the divergence as $(\Div \bba)_i = \sum_{j = 1}^d \ptial{j} \bba_{ij}$, and the gradient as $\nabla \bba = (\ptial{j} \bba)_{j = 1}^d$. Then, we define
\begin{gather*}
    \nabla \bu : \nabla \bv = \sum_{i = 1}^d \ptial{i} \bu \cdot \ptial{i} \bv = \sum_{i,j = 1}^d \ptial{i} \bu_j \ptial{i} \bv_j, \ 
    \nabla \bba : \nabla \bbb = \sum_{i = 1}^d \ptial{i} \bba : \ptial{i} \bbb = \sum_{i,j,k = 1}^d \ptial{i} \bba_{jk} \ptial{i} \bbb_{jk},\\
    (\bu \otimes \bba)_{ijk} = \bu_i \bba_{jk}, \ 
    (\bu \otimes \bba) : \nabla \bbb = \sum_{i,j,k = 1}^d \bu_i \bba_{jk} \ptial{i} \bbb_{jk},
\end{gather*}
Moreover, let $\bbr_{\mathrm{sym}}^{d \times d}$ be the set of symmetric matrices in $\bbr^{d \times d}$. For any $\bbb \in \bbr_{\mathrm{sym}}^{d \times d}$, there exists a diagonal decomposition
%\todo{removed `Now'}
\begin{equation}
    \label{eqs:diagonal}
    \bbb = \tran{\bb{O}} \bbd \bb{O}, %\text{ satisfying } \tr \bbb = \tr \bbd,
\end{equation}
{with $ \bbd = \mathrm{diag} \{\lambda_1, \lambda_2, \dots, \lambda_d\} $, where $\lambda_k > 0$, $k = 1,\dots,d$, are the eigenvalues of $\bbb$, and $\bb{O} \in \bbr^{d \times d}$ is an orthogonal matrix.}
{Let $g:\bbr \to \bbr$ be a scalar function. 
{We define
% Then for such a real symmetric matrix $\bbb \in \bbr^{d \times d}$, we define $g(\bbb)$ and its derivative $g'(\bbb)$ by
\begin{equation}
    \label{eqs:gB}
    g(\bbb) = \bb{O}^\top g(\bbd) \bb{O}, 
    %\quad 
    %g'(\bbb) = \bb{O}^\top g'(\bbd) \bb{O},
\end{equation}
where $g(\bbd) \coloneqq \mathrm{diag} \{g(\lambda_1), g(\lambda_2), \dots, g(\lambda_d)\}$.}
%\todo{Changed as suggested.}
%and $g'(\bbd) \coloneqq \mathrm{diag} \{g'(\lambda_1), g'(\lambda_2), \dots, g'(\lambda_d)\}$.
Although the diagonal decomposition \eqref{eqs:diagonal} is not unique, the function $g(\bbb)$ is uniquely defined by \eqref{eqs:gB}. %\todo{removed `due to the property ...'}%due to the property $\tr( g(\bbb) ) = \tr( g(\bbd) )$.
}
{For any positive definite matrix $ \bbb \in \bbr_{\mathrm{sym}}^{d \times d} $, we define its real logarithm $\ln \bbb$ as a symmetric matrix such that $ e^{\ln \bbb} = \bbb$, {where $e^{\mathbb M}$ denotes the standard matrix exponential of a matrix ${\mathbb M} \in \mathbb{R}^{d \times d}$}. In fact, in view of \eqref{eqs:diagonal} and \eqref{eqs:gB}, we have
\begin{equation*}
    \ln \bbb = \tran{\bb{O}} \mathrm{diag} \{\ln \lambda_1, \ln \lambda_2, \dots, \ln \lambda_d\} \bb{O},
\end{equation*}
which also implies $\tr \ln \bbb = \ln \det \bbb$.}

We recall a set of properties for matrix-valued functions.
%\dennis{Where do we actually use (2.3)--(2.5) from Lemma 2.1?}
\begin{lemma}[Properties of matrix-valued functions]
	Let {$ \bbb : U \to \bbr_{\mathrm{sym}}^{d \times d} $ be positive definite a.e.~in $U$}, with some open $U\subset\bbr^n$, $n\leq d$, $d,n \in \bbn$, be differentiable with respect to a first-order differential operator $ \partial \cdot $. Then
 % \footnote{maybe $\bbb:U\to\bbr^{d\times d}$ and $\bv \in U\to\bbr^d$ with some $U\subset\bbr^n$ open, $n\in\mathbb{N}$ (not necessarily $d=n$). Then $\partial\cdot$ can be a partial space/time derivative. I mean a derivative of $\bbb\in\bbr^{d\times d}$ does not make so much sense. Or: we just do it for the time derivative and say that space derivatives work analogously?\\ % I see. Maybe we should use your idea. $\bbb:U\to\bbr^{d\times d}$ and $\bv \in U\to\bbr^d$ with some $U\subset\bbr^n$ open}
	\begin{gather}
% 		\tr \ln \bbb = \ln \det \bbb, \\
        \label{eqs:energy-positive}
		%\bbb + \inv{\bbb} - 2 \bbi \text{ is symmetric and } 
        {\tr(\bbb - \ln \bbb - \bbi) \geq 0, }
        \\
        \label{eqs:dissipation-positive}
        {\tr(\bbb + \inv{\bbb} - 2 \bbi) \geq 0,}
        \\
        \label{eqs:partial-B-ln}
		\partial \bbb : \inv{\bbb} = \tr\big(\inv{\bbb} \partial \bbb\big) = \partial \tr \ln \bbb, \\
  %       \label{eqs:partial-B-3}
		% \partial \ln \bbb : \bbb = \tr\big(\bbb \partial \ln \bbb\big) = \partial \tr \bbb, \\
		\label{eqs:partial-B-entropy}
		\partial \bbb : (\bbi - \inv{\bbb})= \partial \tr(\bbb - \ln \bbb - \bbi).
	\end{gather}
    Moreover, {let $ \bbb , \bbc \in \bbr_{\mathrm{sym}}^{d \times d} $}, $ \bv \in C^1(U;\bbr^{d}) $,
    %\todo{changed.} 
    then it holds
	\begin{equation}
		\label{eqs:symmetric-product}
		\big(\nabla \bv \bbb + \bbb \tran{\nabla} \bv \big) : \bbc = 2 \big(\bbc \bbb\big): \nabla \bv.
	\end{equation}
\end{lemma}
\begin{proof}
{
The proof of \eqref{eqs:energy-positive} follows directly from \eqref{eqs:diagonal}, \eqref{eqs:gB} and the scalar inequality $s - \ln(s) - 1 \geq 0$ for $s>0$, i.e.,
\begin{align*}
    \tr(\bbb - \ln \bbb - \bbi) 
    = \sum_{i=1}^d ( \lambda_i - \ln(\lambda_i) - 1)
    \geq 0,
\end{align*}
where $\lambda_1, \ldots, \lambda_d>0$ are the eigenvalues of $\bbb$.
Similarly, the desired results \eqref{eqs:dissipation-positive}--\eqref{eqs:partial-B-entropy} are based on \eqref{eqs:diagonal}, \eqref{eqs:gB} and the scalar results $s + s^{-1} - 2 \geq 0$, $(\partial s) s^{-1} = \partial \ln(s)$ and $(\partial s) (1 - s^{-1}) = \partial (s - \ln(s) - 1)$ for all positive scalar functions $s$, {where $\partial$ is a first order differential operator.} %\todo{added sentence.}
% The proof of \eqref{eqs:dissipation-positive} can be directly traced back to the scalar inequalities $s - \ln s - 1 \geq 0$ and $s + s^{-1} - 2 \geq 0$ for all $s>0$, using \eqref{eqs:gB}.
% Similarly, the desired results \eqref{eqs:partial-B-ln}--\eqref{eqs:partial-B-entropy} are based on \eqref{eqs:gB} and the scalar identities $(\partial s) s^{-1} = \partial \ln(s)$ and $(\partial s) (1 - s^{-1}) = \partial (s - \ln(s) - 1)$ for all $s>0$. 
The identity \eqref{eqs:symmetric-product} follows directly from the standard properties of the Frobenius inner product and the symmetry of $\bbb$ and $\bbc$.}
%\footnote{Maybe only reference \cite{BB2011, BLS2017}\\Yes.} 
% We refer to \cite{BB2011,BLS2017}.
% 		\begin{align*}
% 			\big(\nabla \bv \bbb + \bbb \tran{\nabla} \bv \big) : \bbc
% 			& = \big( \ptial{i} \bv^j  B_j^k + B_i^j \ptial{k} \bv^j \big) C_i^k \\
% 			& = \ptial{i} \bv^j  B_j^k C_i^k + B_k^j \ptial{i} \bv^j C_k^i \\
% 			& = \ptial{i} \bv^j  \big( B_j^k C_i^k + B_k^j C_k^i \big) 
% 			= \ptial{i} \bv^j  \big( B_k^j C_i^k + B_k^j C_i^k \big) 
% 			= 2 \big(\bbc \bbb\big): \nabla \bv.
% 		\end{align*}
\end{proof}

\subsubsection{Function spaces}
Let $\Omega \subset \bbr^d$ with $d \in \{2,3\}$ be an open set with Lipschitz boundary $\partial \Omega$. The Lebesgue and Sobolev spaces for functions defined on $\Omega$ with values in $\bbr$ are denoted by $L^p(\Omega)$, $W^{k,p}(\Omega)$ for $1 \leq p \leq \infty$ and $k \in \bbn_0$.
%, as well as $W_0^{k,p}(\Omega) = \overline{C_0^\infty(\Omega)}^{W^{k,p}(\Omega)}$. 
In particular, $W^{0,p} (\Omega) = L^p(\Omega)$. {Moreover, $ C_0^\infty(\Omega) $ is defined as the space of all infinitely differentiable functions with compact support over $ \Omega $, and $W_0^{k,p}(\Omega) = \overline{C_0^\infty(\Omega)}^{W^{k,p}(\Omega)}$ is defined as the closure of $C_0^\infty(\Omega) \cap W^{k,p}(\Omega)$ with respect to the $W^{k,p}$ norm.} %Their induced norms are denoted by $\normm{\cdot}_{X}$ with $X \in \{L^p(\Omega), W^{k,p}(\Omega)\}$ and $X \in \{L^p, W^{k,p}\}$ for simplicity. 
Analogously, we denote the Lebesgue and Sobolev spaces for vector/tensor-valued functions with values in $X$ for some $1 < n \in \bbn$ by $L^p(\Omega; X)$, $W^{k,p}(\Omega; X)$, where $X \in \{\bbr^n, \bbr^{n \times n}\}$. For further applications, we give the dual space as $W^{-k,p}(\Omega) \coloneqq [W_0^{k,p}(\Omega)]'$, solenoidal spaces as
$L_\sigma^p(\Omega) \coloneqq \{\bu \in L^p(\Omega; \bbr^d): \Div \bu = 0, \bn \cdot \bu|_{\partial \Omega} = 0\}$ and $W_{0,\sigma}^{k,p}(\Omega) \coloneqq W_0^{k,p}(\Omega) \cap L_\sigma^p(\Omega)$.
%In the context of Banach spaces, we define the Banach space-valued spaces in a similar fashion. 
Let $I \subset \bbr$ be an interval and $X$ be a Banach space. The space $C(\bar{I}; X)$ denotes the set of continuous functions from $I$ to $X$. Moreover, $C_w(\bar{I}; X)$ is the space of functions that are continuous on $I$ with respect to the weak topology of $X$. The corresponding Banach space-valued Lebesgue and Sobolev spaces are $L^p(I; X)$, $W^{k,p}(I; X)$ for $1 \leq p \leq \infty$ and $k \in \bbn_0$. Throughout the paper, we sometimes use the notation $W_t^{k,p} W_x^{m,q} := W^{k,p}(0,t; W^{m,q}(\Omega))$ with $ k,m \in \bbn_0, 1 \leq p,q \leq \infty$. 

    \begin{lemma}
        \label{lem:energy-dissipation}
		Let $ E:[0,T) \rightarrow [0, \infty) $, $ 0 < T \leq \infty $, be a lower semicontinuous function and let $ D: (0,T) \rightarrow [0, \infty) $ be an integrable function. Then
		\begin{equation*}
			E(0) \varsigma(0) + \int_0^T E(t) \varsigma'(t) \dt \geq \int_0^T D(t) \varsigma(t) \dt
		\end{equation*}
		holds for all $ \varsigma \in W^{1,1}(0,T) $ with $ \varsigma(T) = 0 $ and $ \varsigma \geq 0 $ if and only if
		\begin{equation*}
			E(t) + \int_s^t D(\tau) \dtau \leq E(s)
		\end{equation*}
		holds for all $ s \leq t < T $ and almost all $ 0 \leq s < T $ including $ s = 0 $.
	\end{lemma}
	\begin{proof}
		For the proof, one is referred to Abels \cite[Lemma 4.3]{Abels2009CMP}.
	\end{proof}

	\begin{lemma}
		\label{lem:C_w}
		Let $ X,Y $ be two Banach spaces such that $ Y \hookrightarrow X $ and $ X' \rightarrow Y' $ densely and let $ 0 < T < \infty $. Then 
		\begin{equation*}
			L^\infty(0,T; Y) \cap C([0,T]; X)
			\hookrightarrow C_w([0,T]; Y).
		\end{equation*}
	\end{lemma}
	\begin{proof}
		See Abels \cite[Lemma 4.1]{Abels2009CMP}.
	\end{proof}
	
	% Let $ (X_0, X_1) $ be a compatible couple of Banach spaces, that is, there is a Hausdorff topological vector space $ Z $ such that $ X_0, X_1 \hookrightarrow Z $, and let $ (\cdot, \cdot)_{[\theta]} $, $ (\cdot, \cdot)_{(\theta, r)} $, $ \theta \in [0,1] $, $ r \in [1, \infty] $, denote the complex and real interpolation functor, respectively, cf.~Abels \cite{Abels2009ARMA}, Bergh--L\"{o}fstr\"{o}m \cite{BL1976}. Then we have the following lemma.
	% \begin{lemma}[{\cite[Theorem 5.1.2]{BL1976}}]
	% 	\label{lem:interpolation-L4}
	% 	Let $ (X_0, X_1) $ be a compatible couple of Banach spaces, $ I \subset \bbr $ and $ \Omega \subset \bbr^d $, $ d \in \bbn $. For all $ 1 \leq p_0 < \infty $, $ 1 \leq p_1 \leq \infty $, and $ \theta \in (0,1) $,
	% 	\begin{equation*}
	% 		(L^{p_0}(I; X_0), L^{p_1}(I; X_1))_{[\theta]}
	% 		= L^p(I; (X_0, X_1)_{[\theta]}),
	% 	\end{equation*}
	% 	where $ \frac{1}{p} = \frac{1 - \theta}{p_0} + \frac{\theta}{p_1} $. In particular, for $ p_0 = 2 $, $ p_1 = \infty $, $ p = 4 $, $ X_0 = W^{1,2}(\Omega) $ and $ X_1 = L^2(\Omega) $, it holds
	% 	\begin{equation}
	% 		\label{eqs:interpolation-L4}
	% 		(L^2(I; W^{1,2}(\Omega)), L^\infty(I; L^2(\Omega))_{[\onehalf]}
	% 		= L^4(I; H^{\onehalf, 2}(\Omega))
	% 		\hookrightarrow L^4(I; L^s(\Omega)),
	% 	\end{equation}
	% 	where $ s = 3 $ if $ d = 3 $ and $ s = 4 $ if $ d = 2 $. The space $H^{s, 2}$ for $0 < s < 1$ is the Bessel potential space.
	% \end{lemma}
{
The following lemma is a direct consequence of H\"older's inequality, the interpolation {$ H^{1/2}(\Omega) =(L^2(\Omega),H^1(\Omega))_{1/2} $ and the embedding $ H^{1/2}(\Omega) \hookrightarrow L^4(\Omega) $} in two dimensions. %\todo{Added $\Omega$.}
\begin{lemma}%[Embedding]
    Let $ I \subset \bbr $ and $ \Omega \subset \bbr^2 $ be a domain with smooth boundary. For $f \in L^2(I; W^{1,2}(\Omega)) \cap L^\infty(I; L^2(\Omega))$, it holds
    \begin{equation}
    	\label{eqs:interpolation-L4}
        \int_I \norm{f(t)}_{L^4(\Omega)}^4 \dt 
        \leq C \left(\int_I \norm{f(t)}_{W^{1,2}(\Omega)}^2 \dt\right)
        \left(\sup_{t \in I} \norm{f(t)}_{L^2(\Omega)} \right)^2 .
    \end{equation}
\end{lemma}}
	
\subsubsection{Mollifiers}
	For a {radially symmetric} function $ \psi \in C_0^\infty(\bbr^d) $ with $ \int_{\bbr^d} \psi \dx = 1 $ %, $ \psi(x) = \psi(-x) \geq 0 $ 
 and $ \mathrm{supp} \,\psi \subset \bbr^d $, define $ \psi_\eta(x) \coloneqq \eta^{-d} \psi(\frac{x}{\eta}) $, where $ \eta > 0 $. Then we introduce a regularization operator $ \sR_\eta $ by
    \begin{equation}
    	\label{eqs:mollification}
    	\sR_\eta \bw \coloneqq \psi_\eta * \bw 
    	= \int_{\bbr^d} \psi_\eta(x - y) \bw(y) \,\d y
            = \int_{\Omega} \psi_\eta(x - y) \bw(y) \,\d y,
    \end{equation}
    {for locally integrable $\mathbf{w}$}. % defined on $ \Omega \subset \bbr^d $which are extended by $ 0 $ outside of $ \Omega $. 
    Moreover, the following properties are satisfied.
    \begin{lemma}[Mollification]
        \label{lem:mollification}
        Let $ X $ be a Banach space. If $ \bv \in L^1_{loc}(\bbr^d; X) $, then we have $ \sR_\eta \bv \in C^\infty(\bbr^d; X) $. Furthermore, the following holds:
        \begin{enumerate}
            \item For $ \bu, \bv \in L^1_{loc}(\bbr^d) $, it holds that
                \begin{equation}
                	\label{eqs:mollification-commute}
                	\int_{\bbr^d} \sR_\eta \bu \cdot \bv  \dx = \int_{\bbr^d} \bu \cdot \sR_\eta \bv \dx.
                \end{equation}
            \item If $ \bv \in L_{loc}^p(\bbr^d; X) $, $ 1 \leq p < \infty $, then $ \sR_\eta \bv \in L_{loc}^p(\bbr^d; X) $ and $ \sR_\eta \bv \rightarrow \bv $ in $ L_{loc}^p(\bbr^d; X) $ as $ \eta \rightarrow 0 $.
            \item If $ \bv \in L^p(\bbr^d; X) $, $ 1 \leq p < \infty $, then $ \sR_\eta \bv \in L^p(\bbr^d; X) $, $ \norm{\sR_\eta \bv}_{L^p(\bbr^d; X)} \leq \norm{\bv}_{L^p(\bbr^d; X)} $ and $ \sR_\eta \bv \rightarrow \bv $ in $ L^p(\bbr^d; X) $ as $ \eta \rightarrow 0 $.
            \item If $ \bv \in L^p(\bbr^d; X) $, $ 1 \leq p < \infty $, then $ \norm{\sR_\eta \bv}_{L^\infty(\bbr^d; X)} \leq C(\eta) \norm{\bv}_{L^p(\bbr^d; X)} $ with $ C(\eta) > 0 $ depending on $ \eta $.
            \item If $ \bv \in L^\infty(\bbr^d; X) $, then $ \sR_\eta \bv \in L^\infty(\bbr^d; X) $, and $ \norm{\sR_\eta \bv}_{L^\infty(\bbr^d; X)} \leq \norm{\bv}_{L^\infty(\bbr^d; X)} $.
        \end{enumerate}
    \end{lemma}
    \begin{proof}
        The first statement is a direct consequence of the {radial symmetry} of $\psi$, while the fourth one can be derived by taking the supremum norm of $ \psi $. The remaining statements are referred to \cite[Theorem 11.3]{FN2017}.
    \end{proof}

\section{Matrix-Valued Equation with Stress Diffusion}
	\label{sec:B}

In this section, we are going to solve the tensor-valued equation for $ \bbb $ in dimension $d\in\{2,3\}$, 
\begin{equation}
		\label{eqs:B-equation}
		\begin{alignedat}{3}
			\pt \bbb + \sR_\eta \bv \cdot \nabla \bbb - \bbb \tran{\nabla} \sR_\eta \bv - \nabla \sR_\eta \bv \bbb + (\bbb - \bbi) & = \frac{\kappa}{\mu_\eta(\xi)} \Delta \bbb,  &&  \tin Q_T \\
			\ptial{\bn} \bbb & = {\mathbb{O}}, && \ton S_T \\
			\bbb(0) & = \bbb_0, && \tin \Omega,
		\end{alignedat}
\end{equation}
%\todo{changed to zero matrix, see also (1.1)}
with $ \bbb_0 \in L^2(\Omega;  \bbr^{d \times d}_{\mathrm{sym}})$ positive definite a.e.~in $\Omega$, $d\in\{2,3\}$, and given data $ (\bv, \xi) $ satisfying
%\footnote{Later, we choose $(\bv,\xi)$ constant (in time), so the existence result in Section 3 is more general?\\Yes. This is more general, but with a bit strong assumption, which for constant is obvious.\\ Ok, thank you. Should we do a remark?\\Yes, this is a good suggestion. I made it after the theorem.}
 %\footnote{We can keep this assumption on $(\bv,\xi)$ and solve for $\bbb=\bbb(\bv,\xi)$ on a given time interval $(0,T)$. Then we can keep Section 3 as it is, without modification. Later, we specify piecewise constant (in time) functions $(\bv,\xi)=(\bu^N,\phi^N)$ on the time intervals $(0,T) \coloneqq (t_k,t_{k+1})$, to construct a sequence of solutions and then define $\bbb^N$ on the full time interval.}
	\begin{equation}
		\label{eqs:B-v-phi}
		\bv \in %L^\infty(0,T;L^2(\Omega;\bbr^d)) \cap 
  C([0,T]; W_{0,\sigma}^{1, 2}(\Omega;\bbr^d)), 
            \quad
		\xi \in C([0,T]; W^{2,2}(\Omega)) \cap W^{1,2}(0,T; L^2(\Omega)).
	\end{equation}
	Here, $ \sR_\eta $ is the mollifier defined in \eqref{eqs:mollification} and $ \mu_\eta(\phi) \coloneqq \mu(\sR_\eta \phi) $. As usual, we give the precise definition of \textit{finite energy} weak solutions of the system \eqref{eqs:B-equation}.
\begin{definition}
\label{def:weak-B}
Let $d\in\{2,3\}$, $ T > 0 $, $ \bbb_0 \in L^2(\Omega; \bbr_{\mathrm{sym}}^{d \times d}) $ positive definite a.e.~in $ \Omega $, $\tr \ln \bbb_0 \in L^1(\Omega)$, and $ (\bv, \xi) $ satisfy \eqref{eqs:B-v-phi}. We call $ \bbb $ a finite energy weak solution of \eqref{eqs:B-equation} with data $ (\bbb_0, \bv, \xi) $, provided that%\footnote{replaced $\d t$ by $\d\tau$, to not mix up with $t\in(0,T)$\\Thank you!}
\begin{enumerate}
\item the {tensor} $ \bbb $ satisfies
\begin{gather*}
\bbb \text{ is symmetric positive definite a.e.~in } Q_T; \\
\bbb \in C_w([0,T]; L^2(\Omega; \bbr_{\mathrm{sym}}^{d \times d})) \cap L^2(0, T; W^{1,2}(\Omega; \bbr_{\mathrm{sym}}^{d \times d})); \\
\pt \bbb \in L^2(0, T; [W^{1,2}(\Omega; \bbr_{\mathrm{sym}}^{d \times d})]'); \\
\tr \ln \bbb \in L^\infty(0,T; L^1(\Omega)) 
\cap L^2(0,T; W^{1,2}(\Omega)); 
\end{gather*}
% \footnote{Add the $\tr\ln$ regularity? Here I don't think we can prove $\tr\ln\bbb \in C_w([0,T];L^1(\Omega))$, and basically we don't need it. Because the weak formulation does not require initial data of $\tr\ln\bbb_0$.}
\item for all $ t \in (0,T) $ and all $ \bbc \in C^\infty(\overline{Q_T}; \bbr_{\mathrm{sym}}^{d \times d}) $, we have 
\begin{align}
\int_0^t \int_\Omega & \Big( \bbb : \pt \bbc 
    - {(\sR_\eta \bv \cdot \nabla) \bbb : \bbc} \Big) \dx \dtau \nonumber\\
& \quad + \int_0^t \int_\Omega \Big(\big(\nabla \sR_\eta \bv \bbb + \bbb \tran{\nabla} \sR_\eta \bv \big) : \bbc 
    - \kappa \nabla \bbb : \nabla \frac{\bbc}{\mu_\eta(\xi)} \Big) \dx \dtau 
    \label{eqs:weak-B-formulation} \\
& = \int_0^t \int_\Omega (\bbb : \bbc - \tr \bbc) \dx \dtau
    + \int_\Omega \bbb(\cdot, t) : \bbc(\cdot, t) \dx
    - \int_\Omega \bbb_0 : \bbc(\cdot, 0) \dx; \nonumber
\end{align}
\item for a.e.~$ t \in (0,T) $ we have {the following estimate}
\begin{equation}
\label{eqs:weak-B-energyestimate}
\begin{aligned}
    & \norm{\tr(\bbb - \ln \bbb)(t)}_{L^1}
        + \norm{\bbb(t)}_{L^2}^2 \\
    & \quad + \int_0^t \Big(\norm{\bbb(\tau)}_{L^2}^2
        + \kappa \norm{\nabla \bbb(\tau)}_{L^2}^2\Big) \dtau \\
    & \quad + \int_0^t \Big(\norm{\tr(\bbb
        + \inv{\bbb} - 2\bbi)(\tau)}_{L^1} 
        + \kappa \norm{\nabla \tr \ln \bbb}_{L^2}^2\Big) \dtau \\
    & \leq C,
\end{aligned}
\end{equation}
{where the constant $C>0$ depends on $\norm{\tr\ln\bbb_0}_{L^1}$, $\norm{\bbb_0}_{L^2}$, $\kappa$, bounds of $\mu, \mu'$, $\bv, \xi$, $\eta$, and $T$.}
\end{enumerate}
\end{definition}
	Then main result of this section reads as follows:
	\begin{theorem}
		\label{thm:B-equation}
		Let $ T > 0 $ and $ \Omega \subset\bbr^d$, $d\in\{2,3\}$, be a smooth bounded domain. Assume that $ (\bv, \xi) $ satisfies \eqref{eqs:B-v-phi}, $ \bbb_0 \in L^2(\Omega; \bbr_{\mathrm{sym}}^{d \times d}) $ positive definite a.e.~in $ \Omega $ and $\tr \ln \bbb_0 \in L^1(\Omega)$. Then, there exists a finite energy weak solution of \eqref{eqs:B-equation} in the sense of Definition \ref{def:weak-B}. Moreover, the solution is unique and depends continuously on $ (\bv, \xi, \bbb_0) $.
	\end{theorem}
    \begin{remark}
        This theorem is valid for $(\bv, \xi)$ satisfying \eqref{eqs:B-v-phi}, which is a quite general and rather strong assumption. Later, we will only consider the case with piecewise-in-time constant functions satisfying certain spatial regularity, which, for sure, fulfill \eqref{eqs:B-v-phi} on each subinterval.
    \end{remark}

    {
    % We now dedicate the remainder of this section to the step-by-step proof of Theorem \ref{thm:B-equation}. We begin by introducing a two-layer approximation in Section \ref{sec:Two-layer-approximation} and establish uniform bounds for the approximate solution in Section \ref{sec:uniform-B}. Utilizing standard ODE theory, we justify the maximal existence time for solutions to the approximate system in Section \ref{sec:B-max-existence-time}. With these bounds in place, we apply compactness arguments to justify the passage to the limit in the Galerkin approximation in Section \ref{sec:limit_B_l}. After that, we derive uniform entropy estimates on the regularized level in Section \ref{sec:unifrom_B} and perform the corresponding limit passing in Section \ref{sec:limit_B_delta}. The positive definiteness of $\bbb$ is shown in Section \ref{sec:B-positive-delta}. Finally, we conclude the proof of Theorem \ref{thm:B-equation} by providing the necessary statements regarding the limit system in Section \ref{sec:B-proof}.
    We now dedicate the remainder of this section to the step-by-step proof of Theorem \ref{thm:B-equation}. We begin by introducing a two-layer approximation in Section \ref{sec:Two-layer-approximation} and establish uniform bounds for the approximate solution in Section \ref{sec:uniform-B}. Utilizing standard ODE theory, we justify the maximal existence time for solutions to the approximate system in Section \ref{sec:B-max-existence-time}. 
    With these bounds in place, we apply compactness arguments to justify the passage to the limit in the Galerkin approximation in Section \ref{sec:limit_B_l}. Subsequently, we derive uniform energy estimates at the regularized level in Section \ref{sec:unifrom_B} and perform the corresponding limit passage in Section \ref{sec:limit_B_delta}. The positive definiteness of $\bbb$ is established in Section \ref{sec:B-positive-delta}. Finally, we conclude the proof of Theorem \ref{thm:B-equation} by providing the necessary statements regarding the limit system in Section \ref{sec:B-proof}.
    }
    
	% \yadong{I think we should go with $ \delta \to 0 $ first. Otherwise, the strong formulation would depend on the high regularity and hence $ \delta $ (to get the logarithmic bounds).}

We start with the approximations.
\subsection{Two-layer Approximation}
	\label{sec:Two-layer-approximation}
	By the classical theory of eigenvalue problems for symmetric linear elliptic operators, one can carry out the finite space approximation. Let $ \{\bba_k\}_{k \in \bbn} $ be the eigenfunctions of the matrix-valued Laplace operator with homogeneous Neumann boundary conditions, namely,
	\begin{equation*}
		- \Delta \bba_l = \lambda_l \bba_l \tin \Omega, \quad 
		\ptial{\bn} \bba_l = {\mathbb{O}} \ton \partial \Omega,
	\end{equation*}
 %\todo{changed to zero matrix, see also (1.1)}
	with eigenvalues $ 0 < \lambda_1 \leq \lambda_2 \leq \cdots \leq \lambda_l \rightarrow \infty $, $ l \rightarrow \infty $. Moreover, {it holds}
	\begin{equation*}
		\bba_l \in W^{2,2}(\Omega; \bbr^{d \times d}_{\mathrm{sym}}) \cap C^\infty(\Omega; \bbr^{d \times d}_{\mathrm{sym}}).
	\end{equation*}
	Then $ \{\bba_l\}_{l \in \bbn} $ is an orthonormal basis in $ L^2(\Omega; \bbr^{d \times d}_{\mathrm{sym}}) $ and an orthogonal basis in $ W^{1,2}(\Omega; \bbr^{d \times d}_{\mathrm{sym}}) $ (see also \cite{BLS2017} for the compressible Oldroyd-B case). Define $ \mathbb{H}_l = \mathrm{span} \{\bba_1, \dots, \bba_l\} \subset L^2(\Omega; \bbr^{d \times d}_{\mathrm{sym}}) $ with $ l \in \bbn $ and denote by $ \mathcal{P}_l: W^{1,2}(\Omega; \bbr^{d \times d}_{\mathrm{sym}}) \rightarrow \bb{H}_l $ the orthogonal projection on $ \bb{H}_l $ with respect to the inner product in $ L^2(\Omega; \bbr_{\mathrm{sym}}^{d \times d}) $ such that
	\begin{equation*}
		\innerr{\bba}{\bbb} = \int_\Omega \bba : \bbb \dx, \quad \forall\, \bba, \bbb \in \bbr^{d \times d}.
	\end{equation*}
	Now we make the Galerkin ansatz to approximate the solution $ \bbb $ of \eqref{eqs:B-equation} as 
	\begin{equation*}
		\bbb_l = \sum_{k = 1}^{l} c_k^l(t) \bba_k(x), \quad 
		\bbb_l(0) = \mathcal{P}_l \bbb_0,
	\end{equation*}
	where $ \{c_k^l\} $ are scalar functions of time. Then we arrive at the following approximate system
	\begin{equation}
		\label{eqs:B_Galerkin}
		\begin{aligned}
			\innerr{\pt \bbb_l}{\bb{W}} & + \innerr{\sR_\eta \bv \cdot \nabla \bbb_l}{\bb{W}} - \innerr{\bbb_l \tran{\nabla} \sR_\eta \bv}{\bb{W}} \\
			& - \innerr{\nabla \sR_\eta \bv \bbb_l}{\bb{W}} + \innerr{\bbb_l - \bbi}{\bb{W}} 
%			= \inner{\frac{\kappa}{\mu_\eta(\xi)} \Delta \bbb_l}{\bb{W}},
			= - \innerr{\kappa \nabla \bbb_l}{\nabla \frac{\bb{W}}{\mu_\eta(\xi)}},
		\end{aligned}
	\end{equation}
	for all $ \bb{W} \in \bb{H}_l $, with the initial data $ \bbb_l(0) = \mathcal{P}_l \bbb_0 $. 
    {
	Classically, one can employ the usual ODE theory to verify the solvability of \eqref{eqs:B_Galerkin}. However, in our paper, this is not the case. In fact, we need to ensure that $ \bbb $ is positive definite such that the energy $ \tr(\bbb - \ln \bbb - \bbi) $ makes sense (especially the logarithmic term). }
        Motivated by \cite{BB2011}, we employ a regularization for the logarithmic function $ G(s) = \ln (s) $ to construct a family of symmetric positive definite approximations for $ \bbb $, namely, one defines a $ C^1 $ function for $ \delta > 0 $
	\begin{equation*}
		G_\delta(s) = \left\{
		\begin{alignedat}{2}
			& \frac{s}{\delta} + \ln \delta - 1, \quad & s < \delta, \\
			& \ln s, \quad & s \geq \delta,
		\end{alignedat}
		\right.
	\end{equation*}
	and a cut-off function $ \beta_\delta(s) = [G_\delta'(s)]^{-1} = \max\{s, \delta\} $, for all $ s \in \bbr $. This kind of regularization was also applied in, e.g., \cite{BLS2017} for a compressible Oldroyd-B system, or \cite{GKT2022} for a Cahn--Hilliard tumor growth model including viscoelasticity.
	
	Let us recall the following result from \cite[Lemma 2.1]{BB2011}. 
	\begin{lemma}
		For all $ \bba, \bbc \in \bbr_{\mathrm{sym}}^{d \times d} $, $d\in\{2,3\}$, and for any $ \delta \in (0,1) $, it holds
		\begin{subequations}
			\begin{align}
				\beta_\delta(\bba) G_\delta'(\bba) = G_\delta'(\bba) \beta_\delta(\bba) 
					& = \bbi, \\
				\tr\big(\beta_\delta(\bba) + \beta_\delta^{-1}(\bba) - 2\bbi\big) 
					& \geq 0, \\
				\tr\big(\bba - G_\delta(\bba) - \bbi\big) 
					& \geq 0, \\
				(\bba - \beta_\delta(\bba)) : (\bbi - G_\delta'(\bba)) 
					& \geq 0, \\
				(\bba - \bbc): G_\delta'(\bbc) 
					& \geq \tr\big(G_\delta(\bba) - G_\delta(\bbc)\big).
			\end{align}
		In addition, if $ \delta \in (0, \onehalf] $, it holds
		\begin{align}
			\label{eqs:tr(A-G(A))}
            {\tr(\bba - G_\delta(\bba))} 
            &{ \geq  \max\Big\{ \onehalf \abs{\bba} , \, \frac{1}{2 \delta} \abs{[\bba]_{-}} \Big\},
            }\\
			% \tr(\bba - G_\delta(\bba)) & \geq 
			% 	\left\{
			% 		\begin{aligned}
			% 			& \onehalf \abs{\bba}, \\
			% 			& \frac{1}{2 \delta} \abs{[\bba]_{-}},
			% 		\end{aligned}
			% 	\right. \\
			\bba : (\bbi - G_\delta'(\bba)) & \geq \onehalf \abs{\bba} - d, 
		\end{align}
		where $ [\cdot]_{-} $ denotes the negative part function defined by $ [s]_{-} \coloneqq \min\{s, 0\}, \forall s \in \bbr $.
		\end{subequations}
	\end{lemma}
	Now employing the regularization from above with respect to $ \delta > 0 $, {we introduce} the regularized Oldroyd-B equation with symmetric positive definite approximations.
    \begin{equation}
    \label{eqs:B_Galerkin_reg}
    \begin{aligned}
        \innerr{\pt \bbb_l^\delta}{\bb{W}} & 
        + \innerr{\sR_\eta \bv \cdot \nabla \beta_\delta(\bbb_l^\delta)}{\bb{W}} - \innerr{\beta_\delta(\bbb_l^\delta) \tran{\nabla} \sR_\eta \bv}{\bb{W}} \\
        & - \innerr{\nabla \sR_\eta \bv \beta_\delta(\bbb_l^\delta)}{\bb{W}} + \innerr{\bbb_l^\delta - \bbi}{\bb{W}} 
        = - \innerr{\kappa \nabla \bbb_l^\delta}{\nabla \frac{\bb{W}}{\mu_\eta(\xi)}}
    \end{aligned}
    \end{equation}	
    for all $ \bb{W} \in \bb{H}_l $, subjected to the initial values $ \bbb_l^\delta(0) = \delta\bbi + \mathcal{P}_l \sR_\delta \bbb_0 $, where we used a diagonal shift by $\delta\bbi$, since $\mathcal{P}_l \sR_\delta \bbb_0$ is not necessarily positive definite. Then we have $\delta\bbi + \bbb_0 \to \bbb_0$ in $L^2(\Omega; \bbr_\mathrm{sym}^{d \times d})$, as $\delta \to 0$, and $\delta\bbi + \bbb_0$ is positive definite a.e.~in $\Omega$. 
 %        Consequently, \eqref{eqs:B_Galerkin_reg} turns into the form of
	% \begin{equation}
	% 	%\bb{M}^l 
 %  { \ddt \bc^l } = \bb{L}^l(t) \bc^l + \mathbf{f}^l,
	% \end{equation}
	% where $ [\bc^l(t)]_j = c_j^l(t) $, $ [\bc^l(0)]_j = \innerrm{\delta \bbi+\mathcal{P}_l \bbb_0}{\bba_j}, j = 1,\dots,l $. The matrices $ \bb{M}^l $, $ \bb{L}^l $ and the vector $ \mathbf{f} $ are induced respectively with the entries for $ j,k = 1, \dots, l $, 
	% \begin{alignat*}{3}
	% 	[\bb{M}^l]_{jk} & = && \innerr{\bba_k}{\bba_j} { = \delta_{j,k}}, \\
	% 	[\bb{L}^l(t)]_{jk} & = && - \innerr{\sR_\eta \bv(t) \cdot \nabla \bba_k}{\bba_j}
	% 	+ \innerr{\bba_k \tran{\nabla} \sR_\eta \bv(t)}{\bba_j} \\
	% 	& &&
	% 	+ \innerr{\nabla \sR_\eta \bv(t) \bba_k}{\bba_j} - \innerr{\bba_k}{\bba_j} 
	% 	+ \kappa \lambda_k \innerr{\frac{1}{\mu_\eta(\xi)}\bba_k}{\bba_j}, \\
	% 	[\mathbf{f}^l]_j & = && \int_\Omega \tr(\bba_j) \dx.
	% \end{alignat*}
	% Now we are able to obtain the existence and uniqueness of $ \bc^l $ by verifying the conditions of classical linear ordinary differential equation theory, which thereafter implies that \eqref{eqs:B_GalerkinApproximation} admits a unique solution $ \bbb_l $ on the whole interval $ [0,T_l] $ for all $ l \in\bbn$. {Note that here it follows from the assumption \eqref{eqs:B-v-phi} that the coefficients are continuous with respect to time $t$.}
    Note that in the same way as in \cite{BLS2017,FN2017}, \eqref{eqs:B_Galerkin_reg} is a system of ordinary differential equations for $ \bbb_l^\delta $ with respect to the time coefficients, 
    % $c_k^l(t)$,
    % \yadong{need $c_k^l$?} \dennis{What do you mean?}
    % \yadong{here we didn't present the ODE system with $c_k^l$.}
    for which the Picard--Lindel{\"o}f theorem is applicable. 
    {In particular, the Lipschitz continuity regarding $ \bbb_l^\delta $ is easily verified since $ \beta_\delta(\cdot) $ is Lipschitz continuous and the other terms are linear in $ \bbb_l^\delta $. 
    {We also note the reformulation 
    \begin{align*}
        \innerr{\sR_\eta \bv \cdot \nabla \beta_\delta(\bbb_l^\delta)}{\bb{W}}
        &= 
        - \innerr{\sR_\eta \bv \cdot \nabla \bb{W}}{\beta_\delta(\bbb_l^\delta)}
        - \innerr{\Div \sR_\eta \bv} {\beta_\delta(\bbb_l^\delta) : \bb{W}}
        \\
        &+ \innerr{\sR_\eta \bv \cdot \bn}{\beta_\delta(\bbb_l^\delta) : \bb{W}}_{\partial\Omega},
    \end{align*}
    which is based on integration by parts, where $\innerr{\cdot}{\cdot}_{\partial\Omega}$ denotes the $L^2$ inner product on $\partial\Omega$.}
    The assumption \eqref{eqs:B-v-phi} ensures the continuity in time of the data.}
    This guarantees the existence of a unique solution on a local time interval {$[0,T_{l,\delta}]$}, where $T_{l,\delta} \in (0, T)$.

\subsection{Uniform bounds}
\label{sec:uniform-B}
%$ \abs{\beta_\delta(\bbb_l^\delta)} \leq 1 + \abs{\bbb_l^\delta} $, $ 0 < \delta \leq 1 $, that
{Recalling the scalar inequality $\abs{\beta_\delta(s)} \leq 1 + \abs{s}$, $\forall\, s\in\bbr, \delta\in(0,1)$, it follows from Young's inequality that
$\abs{\beta_\delta(\bbb_l^\delta)} \leq \sqrt{2d} + \sqrt{2} \abs{\bbb_l^\delta}$ for all $\delta\in(0,1)$. Therefore, 
setting $ \bb{W} = \bbb_l^\delta $ in \eqref{eqs:B_Galerkin_reg}, we have}
	\begin{align}
        \nonumber
		\onehalf & \ddt \norm{\bbb_l^\delta}_{L^2}^2
		+ \int_\Omega \big( \sR_\eta \bv \cdot \nabla\big) \beta_\delta(\bbb_l^\delta) : \bbb_l^\delta \dx
		+ \norm{\bbb_l^\delta}_{L^2}^2
		+ \frac{\kappa}{\underline{\mu}} \norm{\nabla \bbb_l^\delta}_{L^2}^2 \\
        \label{eqs:uniform-B-l-delta-ddt}
		& \leq \int_\Omega \tr \bbb_l^\delta \dx
		+ \kappa \int_\Omega \frac{\mu_\eta'(\xi)}{\mu_\eta^2(\xi)} (\nabla \sR_\eta \xi \cdot \nabla) \bbb_l^\delta : \bbb_l^\delta \dx
		+ 2 \int_\Omega \bbb_l^\delta \beta_\delta(\bbb_l^\delta) : \nabla \sR_\eta \bv \dx \\
        \nonumber
		& \leq \int_\Omega \abs{\tr \bbb_l^\delta} \dx
		+ \frac{\kappa \overline{\mu}'}{\underline{\mu}^2} \int_\Omega \abs{(\nabla \sR_\eta \xi \cdot \nabla) \bbb_l^\delta : \bbb_l^\delta} \dx
		+ 2 \int_\Omega (\sqrt{2d} + \sqrt{2} \abs{\bbb_l^\delta}) \abs{\bbb_l^\delta} \abs{\nabla \sR_\eta \bv} \dx.
	\end{align}
	In light of H\"older's and Young's inequalities, one obtains
    \begin{align*}
        \int_\Omega \abs{(\nabla \sR_\eta \xi \cdot \nabla) \bbb_l^\delta : \bbb_l^\delta} \dx
    	& \leq \varepsilon \norm{\nabla \bbb_l^\delta}_{L^2}^2 + C \norm{\nabla \sR_\eta \xi}_{L^\infty}^2 \norm{\bbb_l^\delta}_{L^2}^2 ,
         \\
    	2\int_\Omega (\sqrt{2d} + \sqrt{2} \abs{\bbb_l^\delta}) \abs{\bbb_l^\delta} \abs{\nabla \sR_\eta \bv} \dx
	        & \leq C \big(1 + \norm{\nabla \sR_\eta \bv}_{L^\infty}\big) \norm{\bbb_l^\delta}_{L^2}^2 + C \norm{\nabla \bv}_{L^2}^2,
    \end{align*}
	for some small $ \varepsilon > 0 $ that will be specified later.
    {Note that in \cite[Appendix C]{BLS2017} it is shown that $\beta_\delta(\mathbb{G}) \in W^{1,2}(\Omega;\bbr^{d\times d}_\mathrm{sym})$ with $\norm{\nabla\beta_\delta(\mathbb{G})}_{L^2} \leq \norm{\nabla\mathbb{G}}_{L^2}$ for $\mathbb{G}\in W^{1,2}(\Omega;\bbr^{d\times d}_\mathrm{sym})$. 
    Hence, we infer from H\"older's and Young's inequalities that
    \begin{align*}
        \int_\Omega \abs{\big(\sR_\eta \bv \cdot \nabla\big) \beta_\delta(\bbb_l^\delta) : \bbb_l^\delta} \dx
            & \leq \norm{\nabla \bbb_l^\delta}_{L^2} \norm{\sR_\eta \bv}_{L^\infty} \norm{\bbb_l^\delta}_{L^2}\\
            %\norm{\abs{\sR_\eta \bv} \abs{\bbb_l^\delta}}_{L^2} \\
            & 
	        \leq \varepsilon \norm{\nabla \bbb_l^\delta}_{L^2}^2 + C \norm{\sR_\eta \bv}_{L^\infty}^2 \norm{\bbb_l^\delta}_{L^2}^2,
    \end{align*}
    for some $\varepsilon\in(0,1)$ that will be specified later.} 
    In addition, by the Cauchy-Schwarz inequality and Young's inequality, we derive
	\begin{equation*}
		\int_\Omega \abs{\tr \bbb_l^\delta} \dx
        \leq \int_\Omega \sqrt{d} \abs{\bbb_l^\delta} \dx
        \leq \frac{d}{2} \abs{\Omega} + \frac12 \norm{\bbb_l^\delta}_{L^2}^2.
	\end{equation*} 
	Then integrating \eqref{eqs:uniform-B-l-delta-ddt} over $ (0,t) $, $ t \in (0,T_{l, \delta}) $ together with choosing $ \varepsilon \leq \frac{\kappa}{2 \underline{\mu}} $ yields
	\begin{align}
		& \onehalf \norm{\bbb_l^\delta(t)}_{L^2}^2
			+ \onehalf \int_0^t \norm{\bbb_l^\delta(\tau)}_{L^2}^2 \dtau
			+ \frac{\kappa}{2 \underline{\mu}} \int_0^t \norm{\nabla \bbb_l^\delta(\tau)}_{L^2}^2 \dtau 
			\nonumber \\
			& \leq 
		    % \label{eqs:apriori-B-reg-strong-integral} 
			\onehalf \norm{\bbb_l^\delta(0)}_{L^2}^2
            + \frac{d}{2} \abs{\Omega} t
			% + \int_0^t 2 \sqrt{2} \int_\Omega \tr(\bbb - G_\delta(\bbb))(\tau) \dx \dtau
			% + \frac{\kappa}{2 \underline{\mu}} \int_0^t \norm{\nabla \bbb_l^\delta(\tau)}_{L^2}^2 \dtau 
   %          \\
			% & \quad 
			% + \frac{1}{4} \int_0^t \norm{\bbb_l^\delta(\tau)}_{L^2}^2 \dtau
			+ C \int_0^t \norm{\nabla \bv(\tau)}_{L^2}^2 \dtau 
			\nonumber \\
			& \quad 
			+ C \int_0^t (1 + \norm{\nabla \sR_\eta \xi}_{L^\infty}^2 + \norm{\sR_\eta \bv}_{L^\infty}^2 + \norm{\nabla \sR_\eta \bv}_{L^\infty}) \norm{\bbb_l^\delta(\tau)}_{L^2}^2 \dtau,
			\nonumber
	\end{align}
        which, by using a Gronwall argument (cf.~\cite[Lemma 3.1]{GL2017}), gives rise to 
    \begin{equation}
      \label{eqs:uniform-B-Galerkin}
        \begin{aligned}
            \sup_{0 \leq \tau \leq t} & \norm{\bbb_l^\delta(\tau)}_{L^2}^2
        + \int_0^t \norm{\bbb_l^\delta(\tau)}_{L^2}^2 \dtau
        + \int_0^t \norm{\nabla \bbb_l^\delta(\tau)}_{L^2}^2 \dtau \\
        & {
        \leq C\left(
        \norm{\bbb_l^\delta(0)}_{L^2}^2
        + \abs{\Omega} t
        + \int_0^t \norm{\nabla \bv(\tau)}_{L^2}^2 \dtau\right)
        \left(
        1
        + C \zeta(t) e^{C\zeta(t)}\right) }
        \leq C,
        \end{aligned}
	\end{equation}
        for all $ t \in [0, T_{l,\delta}] $, where $C > 0$ is independent of $l\in\bbn$ and $\delta > 0$, but depends on $\kappa$, $\underline{\mu}$, $T_{l,\delta}$, $\normm{\bbb_0}_{L^2}$, $\bv, \xi$, $\eta$, and {$\zeta(t) \coloneqq \int_0^t (1 + \norm{\nabla \sR_\eta \xi}_{L^\infty}^2 + \norm{\sR_\eta \bv}_{L^\infty}^2 + \norm{\nabla \sR_\eta \bv}_{L^\infty}) \dtau$ which is finite due to \eqref{eqs:B-v-phi} and Lemma \ref{lem:mollification}}.
        % where $C > 0$ depends on $\delta$, $\eta$, $\kappa$, $\underline{\mu}$, $T_{l,\delta}$ and the finite norms\footnote{here also}, but is independent of $l \in \bbn$.

    {
    Moreover, by letting $\bbw = -\Delta \bbb_l^\delta$ and $\bbw = \pt \bbb_l^\delta$, respectively, in \eqref{eqs:B_Galerkin_reg}, it further follows from integration by parts and from H\"older's and Young's inequalities, that
    \begin{align}
        \nonumber
		\onehalf & \ddt \norm{\nabla \bbb_l^\delta}_{L^2}^2
		+ \norm{\nabla \bbb_l^\delta}_{L^2}^2
		+ \frac{\kappa}{\underline{\mu}} \norm{\Delta \bbb_l^\delta}_{L^2}^2 \\
        \nonumber
		& \leq 2 \abs{\int_\Omega \Delta \bbb_l^\delta \beta_\delta(\bbb_l^\delta) : \nabla \sR_\eta \bv \dx }
        + \abs{\int_\Omega \big( \sR_\eta \bv \cdot \nabla\big) \beta_\delta(\bbb_l^\delta) : \Delta \bbb_l^\delta \dx} \\
        \nonumber
		& \leq \frac{\kappa}{2 \underline{\mu}} \norm{\Delta \bbb_l^\delta}_{L^2}^2
        + C \norm{\nabla \sR_\eta \bv}_{L^\infty}^2 (1 + \norm{\bbb_l^\delta}_{L^2}^2)
        + C \norm{\sR_\eta \bv}_{L^\infty}^2 \norm{\nabla \bbb_l^\delta}_{L^2}^2,
	\end{align}
    and
    % \dennis{I think some terms are missing, and the inequality does not work in differential form, but it works in time-integral form, because
    % \begin{align*}
    %     &\kappa \int_\Omega \nabla\bbb : \nabla (\frac{1}{\mu_\eta(\xi)} \partial_t\bbb) \dx
    %     \\
    %     &=  \kappa \int_\Omega \frac{1}{\mu_\eta(\xi)}  \partial_t ( \frac12 \abs{\nabla\bbb}^2) \dx
    %     +
    %     \kappa \int_\Omega (\nabla \frac{1}{\mu_\eta(\xi)} \cdot \nabla\bbb) :  \partial_t\bbb \dx
    %     \\
    %     &= \ddt \kappa \int_\Omega \frac{1}{\mu_\eta(\xi)}   \frac12 \abs{\nabla\bbb}^2 \dx
    %     - \int_\Omega  \partial_t (\frac{1}{\mu_\eta(\xi)})  \frac12 \abs{\nabla\bbb}^2 \dx
    %     \\
    %     &\quad
    %     + \kappa \int_\Omega (\nabla \frac{1}{\mu_\eta(\xi)} \cdot \nabla\bbb) :  \partial_t\bbb \dx
    % \end{align*}
    % }
    % \yadong{Thank you! Already modified.}
    \begin{align}
        \nonumber
		& \ddt \Big(
            \frac{\kappa}{2 \mu_\eta(\xi)} \norm{\nabla \bbb_l^\delta}_{L^2}^2
            + \onehalf \norm{\bbb_l^\delta}_{L^2}^2
            + \int_\Omega \tr(\bbb_l^\delta) \dx
        \Big)
		+ \norm{\pt \bbb_l^\delta}_{L^2}^2 \\
        \nonumber
		& \leq 2 \abs{\int_\Omega \pt \bbb_l^\delta \beta_\delta(\bbb_l^\delta) : \nabla \sR_\eta \bv \dx}
        +\abs{ \int_\Omega \big( \sR_\eta \bv \cdot \nabla\big) \beta_\delta(\bbb_l^\delta) : \pt \bbb_l^\delta \dx} \\
        \nonumber
		& \quad + \kappa \abs{\int_\Omega (\nabla \sR_\eta \xi \cdot \nabla) \bbb_l^\delta : \pt \bbb_l^\delta \frac{\mu_\eta'(\xi)}{\mu_\eta^2(\xi)} \dx}
        + \abs{\kappa \int_\Omega \pt \sR_\eta \xi \abs{\nabla \bbb_l^\delta}^2 \frac{\mu_\eta'(\xi)}{2 \mu_\eta^2(\xi)} \dx} \\
        \nonumber
		& \leq \frac{1}{2} \norm{\pt \bbb_l^\delta}_{L^2}^2
        + C \norm{\nabla \sR_\eta \bv}_{L^\infty}^2 \norm{\bbb_l^\delta}_{L^2}^2 \\
        \nonumber
		& \qquad
        + C \big(\norm{\sR_\eta \bv}_{L^\infty}^2 + \norm{\nabla \sR_\eta \xi}_{L^\infty}^2
        + \norm{\pt \sR_\eta \xi}_{L^\infty} \big) \norm{\nabla \bbb_l^\delta}_{L^2}^2.
	\end{align}
    With the Gronwall's inequality, \eqref{eqs:uniform-B-Galerkin} and the mollification of the initial data, we have in summary that
    \begin{equation}
		\label{eqs:uniform-B-Galerkin-higher}
		\begin{aligned}
		  \sup_{0 \leq \tau \leq t}\norm{\nabla \bbb_l^\delta (\tau)}_{L^2}^2
        + \int_0^t \norm{\pt \bbb_l^\delta(\tau)}_{L^2}^2 \dtau
		+ \int_0^t \norm{\Delta \bbb_l^\delta(\tau)}_{L^2}^2 \dtau
		\leq C_\delta, 
  % \\
  %           \sup_{0 \leq \tau \leq t} 
  %           \norm{\nabla \bbb_l^\delta(\tau)}_{L^2}^2
  %           + \int_0^t \norm{\pt \bbb_l^\delta (\tau)}_{L^2}^2 \dtau
		% \leq C_\delta,
		\end{aligned}
    \end{equation}
    for all $ t \in [0, T_{l,\delta}] $, where the constant $C_\delta > 0$ is independent of $l \in \bbn$, but depends on $\delta > 0$ and $\kappa$, $\underline{\mu}$, $T_{l,\delta}$, $\normm{\bbb_0}_{L^2}$, $\bv, \xi$ and $\eta$.}

\subsection{Maximal Existence Time}\label{sec:B-max-existence-time}
{Here we have an existence time $ T_{l,\delta} $ for the approximate problem \eqref{eqs:B_Galerkin_reg}. In this section, we shall prove that $ T_{l,\delta} $ actually coincides with the final time $T$. By \eqref{eqs:uniform-B-Galerkin-higher} and the embedding $L^2(0,T_{l,\delta};W^{2,2}(\Omega)) \cap W^{1,2}(0,T_{l,\delta};L^2(\Omega)) \hookrightarrow C([0,T_{l,\delta}]; W^{1,2}(\Omega))$, we have
\begin{equation*}
    \norm{\bbb_l^\delta (t)}_{W^{1,2}}^2 \leq C,
\end{equation*}
for all $t \in [0,T_{l,\delta}]$, where $C>0$ does depend on $\delta$ but is independent of $l$. Hence, the existence time can go beyond $T_{l,\delta}$ with a continuity argument. Since the estimate \eqref{eqs:uniform-B-Galerkin} is uniform in $l,\delta$, this process can be repeated finitely many times. % as long as the existence time $T_{l,\delta} < T$, until the final time $T$ is approached, and therefore the maximal existence time is $T_{l,\delta} = T$. 
As long as $T_{l,\delta} < T$, the process continues until the final time $T$ is reached, leading to the conclusion that the maximum time interval is $T_{l,\delta} = T$.
}

\subsection{Passing to the Limit of the Galerkin Approximation}
\label{sec:limit_B_l}

Now, from the uniform bounds \eqref{eqs:uniform-B-Galerkin} and \eqref{eqs:uniform-B-Galerkin-higher}, it follows that the sequence $ \bbb_l $ has the following convergences up to a non-relabeled subsequence, as $ l \rightarrow \infty $,
\begin{alignat*}{4}
	\bbb_l^\delta & \rightarrow \bbb^\delta, \quad && \text{weakly-}* \quad && \text{in } L^\infty(0, T; W^{1,2}(\Omega; \bbr_{\mathrm{sym}}^{d \times d})); \\
	\bbb_l^\delta & \rightarrow \bbb^\delta, \quad && \text{weakly} \quad && \text{in } L^2(0, T; W^{2,2}(\Omega; \bbr_{\mathrm{sym}}^{d \times d})); \\
	\pt \bbb_l^\delta & \rightarrow \pt \bbb^\delta, \quad && \text{weakly} \quad && \text{in } L^2(0, T; L^2(\Omega; \bbr_{\mathrm{sym}}^{d \times d})).
\end{alignat*}
{It is shown in \cite[Appendix C]{BLS2017} that $\beta_\delta(\mathbb{B})$ is a Sobolev function, using that $\beta_\delta(\cdot) \colon \mathbb{R} \to \mathbb{R}$ is globally Lipschitz continuous with Lipschitz constant {equal} to $1$. Therefore, based on the weak(-$*$) convergence results for $\mathbb{B}_l^\delta$ and the Aubin-Lions lemma, one derives the limit passage in the weak formulation \eqref{eqs:B_Galerkin_reg} in a standard way.}
Then, with the higher-order bounds \eqref{eqs:uniform-B-Galerkin-higher},
we end up with the pointwise formulation
\begin{equation}
		\label{eqs:B-delta-strong-formulation}
		\pt \bbb^\delta + \sR_\eta \bv \cdot \nabla \beta_\delta(\bbb^\delta) - \beta_\delta(\bbb^\delta) \tran{\nabla} \sR_\eta \bv - \nabla \sR_\eta \bv \beta_\delta(\bbb^\delta) + (\bbb^\delta - \bbi) = \frac{\kappa}{\mu_\eta(\xi)} \Delta \bbb^\delta,
	\end{equation}
 subjected to $\bbb^\delta(0) = \delta\bbi+ \sR_\delta \bbb_0$.

{Moreover, from the uniform bounds \eqref{eqs:uniform-B-Galerkin} and the weak(-$*$) lower-semicontinuity of the norm in $L^p$ spaces and Fatou's lemma, we know
\begin{equation}
	\label{eqs:uniform-B-Galerkin-limsup}
    \begin{aligned}
        & \norm{\bbb^\delta(t)}_{L^2}^2
    	+ \int_0^t \norm{\bbb^\delta(\tau)}_{L^2}^2 \dtau
	   + \int_0^t \norm{\nabla \bbb^\delta(\tau)}_{L^2}^2 \dtau \\
        & \quad \leq \liminf_{l \to \infty} \left(\norm{\bbb_l^\delta(t)}_{L^2}^2
	   + \int_0^t \norm{\bbb_l^\delta(\tau)}_{L^2}^2 \dtau
	   + \int_0^t \norm{\nabla \bbb_l^\delta(\tau)}_{L^2}^2 \dtau 
        \right)
        \leq C
    \end{aligned}
\end{equation}
for a.e. $t \in (0,T)$, where $C > 0$ does not depend on $\delta>0$.
}

\subsection{Uniform Energy Estimate}
\label{sec:unifrom_B}
{
    As now we have a pointwise strong formulation \eqref{eqs:B-delta-strong-formulation}, by mimicking the arguments in \cite[Section 7.2]{BLS2017}, we derive by multiplying \eqref{eqs:B-delta-strong-formulation} with $ \frac{\mu_\eta(\xi)}{2} (\bbi - G_\delta'(\bbb^\delta)) $, where $G_\delta'(\bbb^\delta) = \beta_\delta^{-1}(\bbb^\delta)$, and integrating the resulting equation over $ \Omega $ and by parts, that
    \begin{align}
		\nonumber
		& \ddt \int_\Omega \frac{\mu_\eta(\xi)}{2} \tr(\bbb^\delta - G_\delta(\bbb^\delta)) \dx \\
		\nonumber
		& \quad 
		+ \int_\Omega \frac{\mu_\eta(\xi)}{2} 
        (\bbb^\delta - \bbi ) : (\bbi - G_\delta'(\bbb^\delta)) \dx % \tr( \bbb^\delta   %\beta_\delta(\bbb^\delta) % + \beta_\delta^{-1}(\bbb^\delta) - 2\bbi) \dx
        - \frac\kappa2 \int_\Omega \nabla \bbb^\delta : \nabla G_\delta'(\bbb^\delta) \dx
		% + \frac\kappa2 \int_\Omega \Delta \bbb^\delta : (\bbi-\beta_\delta^{-1}(\bbb^\delta)) \dx 
        \\
		\nonumber
		&  
		= \int_\Omega \frac{\mu_\eta'(\xi)}{2} \pt \sR_\eta \xi \, \tr(\bbb^\delta - G_\delta(\bbb^\delta)) \dx \\
        \nonumber
        & \quad
        -  \int_\Omega \frac{\mu_\eta(\xi)}{2} \sR_\eta \bv \cdot \nabla \Big(\tr(\beta_\delta (\bbb^\delta) - \ln \beta_\delta (\bbb^\delta))\Big) \dx
		+ \int_\Omega \mu_\eta(\xi) (\beta_\delta(\bbb^\delta) - \bbi) : \nabla \sR_\eta \bv \dx \\
        \label{eqs:B-reg-ddt}
        & \eqqcolon I_1 + I_2 + I_3. 
	\end{align}
    By direct computation,
    % $\normm{\tr(\beta_\delta (\bbb^\delta) - \ln \beta_\delta (\bbb^\delta))}_{L^1} \leq \normm{\tr(\bbb^\delta - G_\delta(\bbb^\delta))}_{L^1}$
    we have
    \begin{align*}
        \abs{I_1} & \leq \frac{\overline{\mu}'}{\underline{\mu}} \norm{\pt \sR_\eta \xi}_{L^\infty} \frac{\,\underline{\mu}\,}{2} \int_\Omega \tr(\bbb^\delta - G_\delta(\bbb^\delta)) \dx.
    \end{align*}
    Recalling \eqref{eqs:uniform-B-Galerkin-limsup}, it follows with H\"older's and Young's inequalities that
    \begin{align*}
        \abs{I_2} & \leq \frac{\overline{\mu}}{2} \norm{\sR_\eta \bv}_{L^2} \Big(\norm{\nabla \bbb^\delta}_{L^2} + \norm{\nabla \tr \ln \beta_\delta (\bbb^\delta)}_{L^2} \Big)\\
        & \leq C\norm{\sR_\eta \bv}_{L^2}^2 
        + \norm{\nabla \bbb^\delta}_{L^2}^2
        + \varepsilon \norm{\nabla \tr\ln \beta_\delta (\bbb^\delta)}_{L^2}^2, \\
        \abs{I_3} & \leq \overline{\mu} \norm{\nabla \sR_\eta \bv}_{L^2} \Big( 2 \norm{\bbb^\delta}_{L^2} + 2d \sqrt{\abs{\Omega}} + \sqrt{\abs{\Omega}} \Big)
        \leq C,
    \end{align*}
    where $C>0$ depends on the bounds in \eqref{eqs:uniform-B-Galerkin} and $\epsilon>0$ will be specified later.
    On noting (cf. \cite[Lemma 7.1]{BLS2017})
    \begin{equation*}
        - \int_\Omega \nabla \bbb^\delta : \nabla G_\delta'(\bbb^\delta) \dx
			\geq \frac{1}{d}\int_\Omega \abs{\nabla \tr \ln \beta_\delta(\bbb^\delta)}^2 \dx,
    \end{equation*}
    and by choosing $\varepsilon = \frac{\kappa}{4d}$, it follows from \eqref{eqs:B-reg-ddt} and from 
    \begin{align*}
    (\bbb^\delta - \bbi ) : (\bbi - G_\delta'(\bbb^\delta))
    &= \tr(\beta_\delta(\bbb^\delta)
		+ \beta_\delta^{-1}(\bbb^\delta) - 2\bbi)
    + (\bbb^\delta - \beta_\delta(\bbb^\delta)) : (\bbi-G_\delta'(\bbb^\delta))
    \\
    &\geq \tr(\beta_\delta(\bbb^\delta)
		+ \beta_\delta^{-1}(\bbb^\delta) - 2\bbi) \geq 0,
    \end{align*}
    that
    \begin{align}
		\nonumber
		& \ddt \int_\Omega \frac{\mu_\eta(\xi)}{2} \tr(\bbb^\delta - G_\delta(\bbb^\delta)) \dx \\
		\nonumber
		& \quad 
		+ \int_\Omega \frac{\mu_\eta(\xi)}{2} \tr(\beta_\delta(\bbb^\delta)
			+ \beta_\delta^{-1}(\bbb^\delta) - 2\bbi) \dx
		+ \frac{\kappa}{4d} \int_\Omega \abs{\nabla \tr \ln \beta_\delta(\bbb^\delta)}^2  \dx \\
		%		& \quad \leq \frac{\mu}{2} \big( \norm{\tr(\bbb^\delta(0) - \ln (\bbb^\delta(0)) - \bbi)}_{L^1} \big) \\
		\label{eqs:apriori-B-reg}
		&  
		\leq \frac{\overline{\mu}'}{\underline{\mu}} \norm{\pt \sR_\eta \xi}_{L^\infty} \frac{\,\underline{\mu}\,}{2} \int_\Omega \tr(\bbb^\delta - G_\delta(\bbb^\delta)) \dx + C + \norm{\nabla \bbb^\delta}_{L^2}^2. 
	\end{align}
	Integrating \eqref{eqs:apriori-B-reg} over $ (0,t) $, $ t \in (0,T) $, and using the bounds for $\mu_\eta(\cdot)$, we obtain
	\begin{align}
		& \frac{\,\underline{\mu}\,}{2} \int_\Omega \tr(\bbb^\delta - G_\delta(\bbb^\delta))(t) \dx 
		\nonumber \\
		& \quad
		+ \int_0^t \int_\Omega \frac{\mu_\eta(\xi)}{2} \tr(\beta_\delta(\bbb^\delta)
		+ \beta_\delta^{-1}(\bbb^\delta) - 2\bbi)(\tau) \dx \dtau
		+ \frac{\kappa}{4d} \int_0^t \norm{\nabla \tr \ln \beta_\delta(\bbb^\delta)(\tau)}_{L^2}^2 \dtau 
		\nonumber \\
		& 
		\leq C \left(\int_\Omega \tr(\bbb^\delta(0) - G_\delta(\bbb^\delta(0))) \dx + t
        + \int_0^t \norm{\nabla \bbb^\delta(\tau)}_{L^2}^2 \dtau\right)
		\nonumber \\
		& \quad 
		+ \int_0^t \frac{\overline{\mu}'}{\underline{\mu}} \norm{\pt \sR_\eta \xi(\tau)}_{L^\infty} \frac{\,\underline{\mu}\,}{2} \Big(\int_\Omega \tr(\bbb^\delta - G_\delta(\bbb^\delta))(\tau) \dx\Big) \dtau.
		\label{eqs:apriori-B-reg-integral}
	\end{align}
    Then, using a Gronwall argument (cf.~\cite[Lemma 3.1]{GL2017}), leads to
    \begin{align}
		\nonumber
		& \sup_{0 \leq \tau \leq t} \int_\Omega \tr(\bbb^\delta - G_\delta(\bbb^\delta))(\tau) \dx  
        + \int_0^t \norm{\nabla \tr \ln \beta_\delta(\bbb^\delta)(\tau)}_{L^2}^2 \dtau \\
		\nonumber
		& \qquad + \int_0^t \int_\Omega \frac{\mu_\eta(\xi)}{2} \tr(\beta_\delta(\bbb^\delta)
		+ \beta_\delta^{-1}(\bbb^\delta) - 2\bbi)(\tau) \dx \dtau \\
		\nonumber
		& 
		\leq C \left(\int_\Omega \tr(\bbb^\delta(0) - G_\delta(\bbb^\delta(0))) \dx + t
        + \int_0^t \norm{\nabla \bbb^\delta(\tau)}_{L^2}^2 \dtau\right) \\
		\label{eqs:apriori-B-reg-gronwall}
		& \qquad \times \Big(
        1 + C\norm{\pt \sR_\eta \xi}_{L_t^1 L_x^\infty}e^{C\norm{\pt \sR_\eta \xi}_{L_t^1 L_x^\infty}}\Big)
        \leq C,
	\end{align}
    for all $t \in [0,T]$, where $C>0$ depends on the bounds in \eqref{eqs:uniform-B-Galerkin-limsup}, $\normm{\tr \ln \bbb_0}_{L^1}$ and is independent of $\delta > 0$.}

\subsection{Passing to the Limit of the Energy Regularization}
\label{sec:limit_B_delta}
% {Now, from the uniform bounds \eqref{eqs:uniform-B-Galerkin} and the weak lower-semicontinuity of the norm in $L^p$ spaces and Fatou's lemma, we know
% \begin{equation}
% 	\label{eqs:uniform-B-Galerkin-limsup}
%     \begin{aligned}
%         & \norm{\bbb^\delta(t)}_{L^2}^2
%     	+ \int_0^t \norm{\bbb^\delta(\tau)}_{L^2}^2 \dtau
% 	   + \int_0^t \norm{\nabla \bbb^\delta(\tau)}_{L^2}^2 \dtau \\
%         & \quad \leq \liminf_{l \to \infty} \left(\norm{\bbb_l^\delta(t)}_{L^2}^2
% 	   + \int_0^t \norm{\bbb_l^\delta(\tau)}_{L^2}^2 \dtau
% 	   + \int_0^t \norm{\nabla \bbb_l^\delta(\tau)}_{L^2}^2 \dtau 
%         \right)
%         \leq C
%     \end{aligned}
% \end{equation}
% for a.e. $t \in (0,T)$, where $C > 0$ does not depend on $\delta>0$.
% }
From the weak formulation of \eqref{eqs:B-delta-strong-formulation} and the estimate \eqref{eqs:uniform-B-Galerkin-limsup}, one directly derives, up to a non-relabeled subsequence, as $ \delta \rightarrow 0 $, that
\begin{alignat*}{4}
	\bbb^\delta & \rightarrow \bbb, \quad && \text{weakly-}* \quad && \text{in } L^\infty(0, T; L^2(\Omega; \bbr_{\mathrm{sym}}^{d \times d})); \\
	\bbb^\delta & \rightarrow \bbb, \quad && \text{weakly} \quad && \text{in } L^2(0, T; W^{1,2}(\Omega; \bbr_{\mathrm{sym}}^{d \times d})); \\
	\pt \bbb^\delta & \rightarrow \pt \bbb, \quad && \text{weakly} \quad && \text{in } L^2(0, T; [W^{1,2}(\Omega; \bbr_{\mathrm{sym}}^{d \times d})]').
\end{alignat*}
In view of the Aubin--Lions lemma and the Sobolev embedding {$W^{1,2}(\Omega)\hookrightarrow L^q(\Omega)$ compactly for $q<\frac{2d}{d-2}$}, we have
\begin{alignat*}{4}
	\bbb^\delta & \rightarrow \bbb, \quad && \text{strongly} \quad && \text{in } {L^2(0, T; L^q(\Omega; \bbr_{\mathrm{sym}}^{d \times d})),}
\end{alignat*}
where the limit tensor $ \bbb $ is real symmetric as $ \bbb^\delta $ is real symmetric for all $ \delta > 0 $.
{As a consequence, one infers
\begin{alignat*}{4}
	\bbb^\delta & \rightarrow \bbb, \quad && \text{a.e.} \quad && \text{in } Q_T,
\end{alignat*}
which, using that $\beta_\delta(\cdot)$ is Lipschitz continuous with Lipschitz constant {equal} to $1$, implies that
\begin{alignat}{4}
    \label{eqs:convergence-betadelta-to-b-pointwise}
	\beta_\delta(\bbb^\delta) & \rightarrow [\bbb]_+, \quad && \text{a.e.} \quad && \text{in } Q_T,
\end{alignat}
where $[\cdot]_+ \colon s\mapsto \max\{s, 0\}$, $s\in\bbr$, is the positive part function. In \cite[Appendix C]{BLS2017} it is shown that $\beta_\delta(\mathbb{G}) \in W^{1,2}(\Omega;\bbr^{d\times d}_\mathrm{sym})$ with $\norm{\nabla\beta_\delta(\mathbb{G})}_{L^2} \leq \norm{\nabla\mathbb{G}}_{L^2}$ for $\mathbb{G}\in W^{1,2}(\Omega;\bbr^{d\times d}_\mathrm{sym})$. By this and the uniform in $\delta$ boundedness of $\{\bbb^\delta\}_{\delta>0}$ in $L^2(0,T;W^{1,q}(\Omega;\bbr^{d\times d}_\mathrm{sym}))$, $q\in[2,\frac{2d}{d-2})$, we find $\{\beta_\delta(\bbb^\delta)\}_{\delta > 0}$ is uniformly bounded in $L^2(0,T;W^{1,q}(\Omega;\bbr^{d\times d}_\mathrm{sym}))$.
Recalling Vitali's theorem, we further exploit
\begin{alignat}{4} 
    \label{eqs:convergence-betadelta-to-b}
	\beta_\delta(\bbb^\delta) & \rightarrow [\bbb]_+, \quad && \text{strongly} \quad && \text{in } {L^2(0, T; W^{1,q}(\Omega; \bbr_{\mathrm{sym}}^{d \times d}))}.
\end{alignat}
%where $ \bbb_l $ is real symmetric as $ \bbb_l^\delta $ is real symmetric for all $ l\in\bbn $ and $ \delta > 0 $.
Therefore, we can get \eqref{eqs:weak-B-formulation} by taking the limit $\delta\to0$ in the weak formulation of \eqref{eqs:B-delta-strong-formulation} with standard arguments that are based on the convergence results above for $\mathbb{B}^\delta$ and $\beta_\delta(\mathbb{B}^\delta)$.
}

\subsection{Positive Definiteness} %the Left Cauchy--Green Tensor}
\label{sec:B-positive-delta}
To make sure that the weak formulation \eqref{eqs:B_Galerkin} is recovered from \eqref{eqs:B_Galerkin_reg} in the limit $\delta \to 0$, 
% and similarly {for the inequality \eqref{eqs:weak-B-energyestimate}}, 
we mimic the arguments from \cite[Section 8.2]{BLS2017} and \cite[Lemma 4.1]{BLL2022} to prove the positive definiteness of the limit function $ \bbb $. 
{Since $\bbb^\delta$ is positive definite a.e.~in $\Omega$ for $\delta>0$, the pointwise limit $\bbb$ is positive semi-definite a.e.~in $\Omega$. In the following lemma, we use the uniform estimate \eqref{eqs:apriori-B-reg-gronwall} to show that the limit function $\bbb$ is in fact positive definite.}
%Here, we use the uniform estimate \eqref{eqs:apriori-B-reg-gronwall}. %, also see \cite[Theorem 6.2]{BB2011}. 
%\todo{added sentence as suggested}
\begin{lemma}
	\label{lem:positive_B}
	Let $ \bbb $ be the limit function of the sequence of positive definite solutions $ \bbb^\delta $. Then $ \bbb $ is positive definite a.e.~in $ Q_T $.
\end{lemma}
\begin{proof}
	Assume that $ \bbb $ is not positive definite a.e.~in $Q_T$. Then, there exists a subset $ D \subset Q_T $ with nonzero measure and a vector-valued function $ \bw \in L^\infty(Q_T; \bbr^d) $ satisfying $ \abs{\bw}= 1 $ a.e.~in $ D $ and $ \bw = 0 $ a.e.~in $ Q_T \setminus D $ such that
	\begin{equation}
		\label{eqs:contracdiction_Bw}
		[\bbb]_+ \bw = 0, \text{ a.e.~in } Q_T.
	\end{equation}
{On noting the Cauchy--Schwarz inequality, H\"older's inequality, \eqref{eqs:apriori-B-reg-gronwall} and the fact that $|\mathbf{w}| = 1$ a.e.~in $D$, we calculate 
\begin{align*}
\abs{D} &= \int_0^T \int_{\Omega} \abs{\bw}^2 \dx \, \d t
\\
&= \int_0^T \int_\Omega | \mathbf{w}^\top \beta_\delta^{-\frac12}(\mathbb{B}^\delta) \beta_\delta^{\frac12}(\mathbb{B}^\delta) \mathbf{w} | \dx \, \d t
\\
&\leq \Big( \int_0^T \int_\Omega | \mathbf{w}^\top \beta_\delta^{-1}(\mathbb{B}^\delta) \mathbf{w} | \dx \, \d t \Big)^{\frac12}
\Big( \int_0^T \int_\Omega | \mathbf{w}^\top \beta_\delta(\mathbb{B}^\delta) \mathbf{w} | \dx \, \d t \Big)^{\frac12}
\\
&\leq \Big( \int_0^T \int_\Omega | \beta_\delta^{-1}(\mathbb{B}^\delta)| \, |\mathbf{w}|^2 \dx \, \d t \Big)^{\frac12}
\Big( \int_0^T \int_\Omega | \mathbf{w}^\top \beta_\delta(\mathbb{B}^\delta) \mathbf{w} | \dx \, \d t \Big)^{\frac12}
\\
&\leq C \Big( \int_0^T \int_\Omega | \mathbf{w}^\top \beta_\delta(\mathbb{B}^\delta) \mathbf{w} | \dx \, \d t \Big)^{\frac12},
\end{align*}
}
where $ C $ does not depend on $ \delta $ due to the uniform estimate \eqref{eqs:apriori-B-reg-gronwall}. {Using \eqref{eqs:convergence-betadelta-to-b} and \eqref{eqs:contracdiction_Bw}, we obtain in the limit $\delta\to 0$, that $ \abs{D} = 0 $, which is a contradiction to the assumption of a nonzero measure of $ D $.} Hence, $ \bbb $ is positive definite a.e.~in $ Q_T $.
\end{proof}
On account of Lemma \ref{lem:positive_B} and the uniform bounds of $\tr \ln \beta_\delta(\bbb^\delta)$, we can extract a converging subsequence of $\{ \nabla \tr \ln \beta_\delta(\bbb^\delta)\}_{\delta>0}$ in the space $L^2(0,T;L^2(\Omega;\bbr^d))$ and a limit function $ \overline{\nabla \tr \ln \bbb} \in L^2(0,T;L^2(\Omega;\bbr^d))$ such that, as $\delta \to 0$, 
\begin{equation}
    \label{eqs:trlnBN-weak-convergence}
		{ \nabla \tr \ln \beta_\delta(\bbb^\delta) \rightarrow \overline{\nabla \tr \ln \bbb}, \quad \text{weakly} \quad \text{in } L^2(0, T; L^2(\Omega; \bbr^d)),}
\end{equation}
{ Next, we can identify $\overline{\nabla \tr \ln \bbb}$ with $\nabla \tr \ln \bbb$.
% \footnote{The argument is corrected now.\\Thank you very much! It's very nice!} 
Note that the energy estimates \eqref{eqs:apriori-B-reg-gronwall} and the definitions of $G_\delta(\cdot)$ and $\beta_\delta(\cdot)$ imply the uniform boundedness of $\tr \ln \beta_\delta(\bbb^\delta)$ in $L^\infty(0,T;L^1(\Omega))$. 
Hence, with the Poincaré--Wirtinger inequality, Lemma \ref{lem:positive_B}, the continuity of the logarithm, the pointwise convergence result \eqref{eqs:convergence-betadelta-to-b-pointwise} and the positive definiteness of $\bbb$ (which implies that $[\bbb]_+ = \bbb$), we obtain for a non-relabeled subsequence, as $\delta\to 0$,}
\begin{align*}
    \tr\ln\beta_\delta(\bbb^\delta) \to \tr\ln \bbb, \quad \text{weakly} \quad \text{in } L^2(0,T; L^2(\Omega)),
\end{align*}
{and thanks to \eqref{eqs:trlnBN-weak-convergence} also weakly in $L^2(0,T; W^{1,2}(\Omega))$.
}

\subsection{Completing the Proof of Theorem \ref{thm:B-equation}}
\label{sec:B-proof}
With all the arguments from above, we conclude the existence of weak solutions in Theorem \ref{thm:B-equation} by verifying Definition \ref{def:weak-B}. Noticing that we already have
\begin{equation*}
	\bbb \in L^\infty(0,T; L^2(\Omega; \bbr_{\mathrm{sym}}^{d \times d})) \cap W^{1,2}(0, T; [W^{1,2}(\Omega; \bbr_{\mathrm{sym}}^{d \times d})]')
\end{equation*}
and the Sobolev embeddings 
\begin{gather*}
	W^{1,2}(0,T; X) \hookrightarrow C([0,T]; X), \  \text{ where }X \text{ is a Banach space}, \\
	W^{1,2}(\Omega; \bbr_{\mathrm{sym}}^{d \times d}) \hookrightarrow L^2(\Omega; \bbr_{\mathrm{sym}}^{d \times d}) \hookrightarrow [W^{1,2}(\Omega; \bbr_{\mathrm{sym}}^{d \times d})]' \text{ densely },
\end{gather*}
one infers from Lemma \ref{lem:C_w} that 
\begin{equation*}
	\bbb \in C_w([0,T); L^2(\Omega; \bbr_{\mathrm{sym}}^{d \times d})).
\end{equation*}
% Moreover, it follows from \eqref{eqs:apriori-B-reg} (see also \cite[(3.14)]{BLS2017}) that
% \begin{align*}
%     \int_0^T \int_\Omega \pt \tr G_\delta(\bbb^\delta) \dxdt
%     + \frac{\kappa}{d} \int_0^T \int_\Omega \abs{\nabla \tr\ln\beta_\delta(\bbb^\delta)}^2 \dxdt 
%     \leq C.
% \end{align*}
% By letting $\delta \to 0$, one gets
% \begin{align*}
%     \int_0^T \int_\Omega \pt \tr \ln \bbb \dxdt
%     + \frac{\kappa}{d} \int_0^T \int_\Omega \abs{\nabla \tr\ln\bbb}^2 \dxdt 
%     \leq C,
% \end{align*}
% which implies \footnote{Current idea.}
% \begin{equation*}
%     \tr \ln \bbb \in C([0,T); L^1(\Omega)).
% \end{equation*}
Combining all the convergence results from above, we obtain the weak formulation \eqref{eqs:weak-B-formulation}. The {inequality} \eqref{eqs:weak-B-energyestimate} follows immediately from the uniform estimates \eqref{eqs:uniform-B-Galerkin-limsup} and \eqref{eqs:apriori-B-reg-gronwall} using the weak-($*$) lower semicontinuity of norms and Fatou's lemma. This proves the existence.

{In the final step, we provide the main arguments for the uniqueness and continuous dependence of weak solutions to \eqref{eqs:B-equation} on the data $(\bv, \xi, \bbb_0)$. In fact, one only needs to prove the continuous dependence, and, as a consequence, we then have the uniqueness.
%\todo{removed `in fact'} 
The continuous dependence easily follows since the equation is linear in $\bbb$. Hence, we only need to control the differences in regard to $\bv, \xi$ and coefficients, which are quite standard. 
Therefore, we omit the details and present the final inequality. Let $ \bbb_i $, $ i = 1,2 $ be two weak solutions of \eqref{eqs:B-equation} in the sense of Definition \ref{def:weak-B} with data $ (\bv_i, \xi_i, \bbb_0^i) $, respectively, satisfying \eqref{eqs:B-v-phi} and let $ \overline{\bbb} = \bbb_1 - \bbb_2 $ and $ (\overline{\bv}, \overline{\xi}) = (\bv_1 - \bv_2, \xi_1 - \xi_2) $.  Then, it holds
	\begin{align}
		\label{eqs:B-uniqueness-estimate}
		& \norm{\overline{\bbb}(t)}_{L^2}^2
			+ \int_0^t \norm{\overline{\bbb}(\tau)}_{L^2}^2 \dtau
			+ \int_0^t \norm{\nabla \overline{\bbb}}_{L^2}^2 \dtau \\
   \nonumber
			& \leq C(\eta) \Big( \norm{\overline{\bv}}_{L_t^2 L_x^2}^2 + \norm{\nabla \overline{\bv}}_{L_t^2 L_x^2}^2
			+ \norm{\overline{\xi}}_{L_t^\infty L_x^2}^2 + \norm{\nabla \overline{\xi}}_{L_t^2 L_x^2}^2
			+ \norm{\overline{\bbb}_0}_{L^2}^2 \Big)(1 + Cg(t)e^{Cg(t)}),
	\end{align}
	where $ C(\eta,T) $ depends on $ \normm{\nabla \bbb_2}_{L_t^2 L_x^2} $, $T$, $ \eta $, $ \kappa $ and the upper-lower bounds of the coefficients, the function $ g(t) \coloneqq \int_0^t (\normm{\nabla \sR_\eta \bv_1(\tau)}_{L^\infty} + \normm{\nabla \sR_\eta \xi_1(\tau)}_{L^\infty}^2 + \normm{\bbb_2(\tau)}_{L^2}^2+ \normm{\nabla \bbb_2(\tau)}_{L^2}^2) \dtau $ is finite for $ t \in (0,T) $ thanks to the existence of weak solutions and mollifiers.}
	
This completes the proof of Theorem \ref{thm:B-equation}. \qed

\begin{remark}
    For suitably smooth solutions, the positive definiteness of the {tensor} $ \bbb $ can be preserved for all $ t > 0 $ if $ \bbb $ is positive definite initially. This has been observed for the 2D Oldroyd-B equation with stress diffusion \cite[Proposition 1]{CK2012} and for other viscoelastic fluid systems in, e.g., \cite[Lemma 2.1]{HL2007}, \cite[Remark 3.4]{LMNR2017}. 
    We remark here that one could obtain a similar result as in \cite[Proposition 1]{CK2012} for $d \in \{2,3\}$, by means of our argument above, provided with a stronger initial data $\bbb_0 \in W^{1,2}(\Omega; \bbb_{\mathrm{sym}}^{d \times d})$ that is positive definite a.e.~in $\Omega$ and $\tr \ln \bbb_0 \in L^1(\Omega)$. 
    % The argument would not change too much in the Galerkin approximation due to higher regularity estimates, while in the entropy regularization, the initial data does not need to be mollified.
    % \footnote{Maybe we remove this sentence to avoid confusion.\\Agree.}
    % Since we have mollified smooth initial data and velocity, in principle one could apply these results directly to get nonnegative definiteness. In this manuscript, we still carry out the above argument, in case of the later proof of the main theorem, which does not have a ''linear`` $\bbb$-equation and smooth data, see Sections \ref{sec:proof-reg} and \ref{sec:proof-weak-limiting}.
\end{remark}

\section{Regularized System}
\label{sec:regularized}
As discussed before, we are going to introduce a regularization to \eqref{eqs:Model}, so that one can easily obtain good uniform estimates and pass to the limit in the regularization parameter with enough compactness. In view of the regularization operator $ \sR_\eta $ defined in \eqref{eqs:mollification} for $ \eta > 0 $, we define $ \mu_\eta(\phi) \coloneqq \mu(\sR_\eta \phi) $. The aim of this section is therefore to obtain a weak solution to the following regularized system: 
\begin{subequations}
	\label{eqs:Model_reg}
	\begin{alignat}{3}
		\label{eqs:fluid_reg-momentum}
		\begin{split}
			\pt (\rho(\phi) \bu) + \bu \cdot \nabla(\rho(\phi) \bu) & - \rho'(\phi) \Div \big( \bu \otimes m(\phi) \nabla q \big) + \nabla \pi \\
			- \Div \big( \bbs_\eta(\nabla \bu, \bbb, \phi) \big) & = q \nabla \phi + \sR_\eta \Big[\frac{\mu_\eta(\phi)}{2} \nabla \tr(\bbb - \ln \bbb - \bbi)\Big]
		\end{split} && \tin Q_T, \\
		\Div \bu & = 0 && \tin Q_T, \\
		\label{eqs:fluid_reg-B}
		\pt \bbb + \sR_\eta \bu \cdot \nabla \bbb + (\bbb - \bbi) & = \bbb \tran{\nabla} \sR_\eta \bu + \nabla \sR_\eta \bu \bbb + \frac{\kappa}{\mu_\eta(\phi)} \Delta \bbb && \tin Q_T, \\
		\label{eqs:CH-phi_reg}
		\pt \phi + \bu \cdot \nabla \phi & = \Div (m(\phi) \nabla q) && \tin Q_T, \\
		\label{eqs:CH-q_reg}
		q - W'(\phi) + \Delta \phi
		- \eta \ptial{t} \phi
		& = \sR_\eta \Big[\frac{\mu_\eta'(\phi)}{2} \tr(\bbb - \ln \bbb - \bbi)\Big] && \tin Q_T, \\
		\bu = \mathbf{0},\ \ptial{\bn} \bbb & = {\mathbb{O}} && \ton S_T, \\ 
		\ptial{\bn} \phi = \ptial{\bn} q & = 0  && \ton S_T, \\
		(\bu, \bbb, \phi)(0) & = (\bu_0, \bbb_0, \phi_0) && \tin \Omega,
	\end{alignat}
\end{subequations}
%\todo{changed to zero matrix, see also (1.1)}
where
\begin{equation*}
	\bbs_\eta(\nabla \bu, \bbb, \phi) = \nu(\phi)(\nabla \bu + \tran{\nabla} \bu) + \sR_\eta \big[\mu_\eta(\phi)(\bbb - \bbi)\big].
\end{equation*}

\subsection{Formal a Priori Estimates and Weak Solutions}
\label{sec:a-priori-reg}
{
In this section, we start with the formal energy estimates uniform in the regularization parameter $ \eta > 0 $.
Note that this energy is not necessarily finite if $ \bbb $ is not positive definite, which is due to the logarithmic term. In Section \ref{sec:Two-layer-approximation} we overcome it by means of an energy regularization to the logarithmic function. %Here we do everything formally, which will be justified later. 
{Now we discuss formal \textit{a priori} estimates} for \eqref{eqs:Model_reg}. Testing \eqref{eqs:fluid_reg-momentum} with $ \bu $, we have
\begin{equation}
	\label{eqs:formal-u_reg}
	\begin{aligned}
		\frac{\d}{\d t} & \int_\Omega \frac{\rho(\phi)}{2} \abs{\bu}^2 \dx
		+ \int_\Omega \Big( \frac{\nu(\phi)}{2} \abs{\nabla \bu + \tran{\nabla} \bu}^2
			+ \sR_\eta \big[\mu_\eta(\phi) \big( \bbb - \bbi \big)\big] : \nabla \bu \Big) \dx
%			+ \underbrace{\int_\Omega \bu \cdot \nabla \Big( \frac{\rho(\phi)}{2} \abs{\bu}^2 \Big)}_{= 0 \text{ due to } \Div \bu = 0,\ \bu_{|\partial \Omega} = 0}
		\\
		&  
		+ \int_\Omega \Big( \underbrace{(\pt + \bu \cdot \nabla) \rho(\phi) - \Div \big( m(\phi) \nabla q \big) \rho'(\phi)}_{= 0 \text{ by } \eqref{eqs:CH-phi_reg}} \Big) \frac{\abs{\bu}^2}{2} \dx \\
		& = \int_\Omega q \nabla \phi \cdot \bu \dx
        + \int_\Omega \sR_\eta \Big[\frac{\mu_\eta(\phi)}{2} \nabla \tr(\bbb - \ln \bbb - \bbi)\Big] \cdot \bu \dx,
		% + \underbrace{\int_\Omega \sR_\eta \Big[\frac{\mu_\eta(\phi)}{2} \nabla \tr(\bbb - \ln \bbb - \bbi)\Big] \cdot \bu \dx}_{= - \int_\Omega \frac{\mu_\eta'(\phi)}{2} \sR_\eta \bu \cdot \nabla \sR_\eta \phi \tr(\bbb - \ln \bbb - \bbi) \dx},
	\end{aligned}
\end{equation} 
where we used the identity for a first-order differential operator $ \partial \cdot $ that
\begin{align*}
	{ \partial (\rho \bu) \cdot \bu 
	= \partial \rho \abs{\bu}^2 + \rho \partial \bu \cdot \bu 
    %\\
	% & = \partial \rho \abs{\bu}^2 + \rho \partial \Big(\frac{\abs{\bu}^2}{2}\Big) 
	% = \partial \rho \abs{\bu}^2 + \partial \Big(\rho \frac{\abs{\bu}^2}{2}\Big) - \partial \rho \frac{\abs{\bu}^2}{2} 
	= \partial \Big(\rho \frac{\abs{\bu}^2}{2}\Big) + \partial \rho \frac{\abs{\bu}^2}{2}. }
\end{align*}
Multiplying \eqref{eqs:fluid_reg-B} by $ \frac{\mu_\eta(\phi)}{2} (\bbi - \inv{\bbb}) $ and integrating over $ \Omega $, one obtains
\begin{align}
	\ddt & \int_\Omega \frac{\mu_\eta(\phi)}{2} \tr(\bbb - \ln \bbb - \bbi) \dx 
	\nonumber \\
	& \qquad + \int_\Omega \frac{\mu_\eta(\phi)}{2} \tr(\bbb
	+ \inv{\bbb} - 2\bbi) \dx 
	+ \frac{\kappa}{2d} \int_\Omega \abs{\nabla \tr \ln \bbb}^2 \dx 
	\label{eqs:formal-B_reg} \\
	& \quad
	% \leq \int_\Omega \frac{\mu_\eta'(\phi)}{2} (\pt + \sR_\eta \bu \cdot \nabla) \sR_\eta \phi \tr(\bbb - \ln \bbb - \bbi) \dx
	% + \int_\Omega \mu_\eta(\phi) (\bbb - \bbi) : \nabla \sR_\eta \bu \dx, 
	\leq \int_\Omega \frac{\mu_\eta'(\phi)}{2} \pt \sR_\eta \phi \, \tr(\bbb - \ln \bbb - \bbi) \dx
    - \int_\Omega \frac12 \mu_\eta(\phi) \sR_\eta \bu \cdot \nabla \tr(\bbb - \ln \bbb - \bbi) \dx 
    \nonumber
    \\ 
    &\quad
	+ \int_\Omega \mu_\eta(\phi) (\bbb - \bbi) : \nabla \sR_\eta \bu \dx, 
	\nonumber
\end{align}
where we employed \eqref{eqs:partial-B-entropy}, \eqref{eqs:symmetric-product} and 
\begin{equation*}
	- \int_\Omega \nabla \bbb : \nabla \inv{\bbb} \dx \geq \frac1d \int_\Omega \abs{\nabla \tr \ln \bbb}^2 \dx,
\end{equation*}
which is referred to \cite[Lemma 3.1]{BLS2017} with $ d = 2,3 $. Testing \eqref{eqs:CH-phi_reg} with $ q $, \eqref{eqs:CH-q_reg} with $ - \pt \phi $, integrating and adding both equations and integrating by parts over $\Omega$ yield
\begin{equation}
	\label{eqs:formal-CH_reg}
	\begin{aligned}
		\ddt \int_\Omega \Big(\frac{1}{2} \abs{\nabla \phi}^2 & + W(\phi)\Big) \dx
		+ \int_\Omega \bu \cdot \nabla \phi q \dx
		+ \eta \int_\Omega \abs{\ptial{t} \phi}^2 \dx
		\\
		& \qquad   
		+ \int_\Omega \frac{\mu_\eta'(\phi)}{2} \pt \sR_\eta \phi \tr(\bbb - \ln \bbb - \bbi) \dx
		+ \int_\Omega m(\phi) \abs{\nabla q}^2 \dx = 0.
	\end{aligned}
\end{equation}
%Thanks to the regularization parameter $ \eta  $ we introduced, 
Summing \eqref{eqs:formal-u_reg}, \eqref{eqs:formal-B_reg} and \eqref{eqs:formal-CH_reg} together gives rise to
\begin{equation}
	\label{eqs:dt-formal_reg}
	\begin{aligned}
		\ddt \cE_\eta(t) 
    		+ \int_\Omega \Big( \frac{\nu(\phi)}{2} & \abs{\nabla \bu + \tran{\nabla} \bu}^2 
    		+ m(\phi) \abs{\nabla q}^2 
    		+ \eta \abs{\ptial{t} \phi}^2
		\Big) \dx \\
		&  
		+ \int_\Omega \frac{\mu_\eta(\phi)}{2} \tr(\bbb
		+ \inv{\bbb} - 2\bbi) \dx
		+ \frac{\kappa}{2d} \int_\Omega \abs{\nabla \tr \ln \bbb}^2 \dx
		\leq 0,
	\end{aligned}
\end{equation}
with the help of \eqref{eqs:mollification-commute}, where 
%rather complicated terms regarding variable shear muduli $ \mu(\phi) $ 
the mixed terms are canceled after the summation, {and the total energy $ \cE_\eta(t) = \cE_\eta(\bu,\phi,\bbb)(t) $ is defined by
\begin{equation}
	\label{eqs:energy-reg}
	\cE_\eta(\bu,\phi,\bbb) = \int_\Omega \frac{\rho(\phi)}{2} \abs{\bu}^2 \dx
	+ \int_\Omega \frac{\mu_\eta(\phi)}{2} \tr(\bbb - \ln \bbb - \bbi) \dx + \int_\Omega \frac{1}{2} \abs{\nabla \phi}^2 + W(\phi) \dx.
\end{equation}}
Then, integrating \eqref{eqs:dt-formal_reg} over $ t \in (0, \tau) $, $ \tau \in (0,T) $ and using the upper and lower bounds of $ \nu, \mu, m $, one finally obtains the \textit{a priori} estimate
\begin{equation}
    \label{eqs:formalEstimate_reg}
    \begin{aligned}
        \cE_\eta(\tau) & + \int_0^\tau \norm{\nabla \bu(t)}_{L^2}^2 \dt
	    + \int_0^\tau \norm{\nabla q(t)}_{L^2}^2 \dt 
	    + \eta \int_0^\tau \norm{\ptial{t} \phi(t)}_{L^2}^2 \dt
	    \\
	    & 
	    + \int_0^\tau \norm{\tr(\bbb + \inv{\bbb} - 2\bbi)(t)}_{L^1} \dt
	    + \int_0^\tau \norm{\nabla \tr \ln \bbb(t)}_{L^2}^2 \dt
	    \leq C \cE_\eta^0,
    \end{aligned}
\end{equation}
for a.e.~$ \tau \in (0,T) $ {and $ \cE_\eta^0 \coloneqq \cE_\eta(\bu_0,\phi_0,\bbb_0) $}. %\todo{adapt notation energy}
Here, the positive constant $ C $ depends on the upper and lower bounds of $ \nu, \mu, m $, and is independent of $ \eta > 0 $.
}

Let us continue with the definition of weak solutions to the system \eqref{eqs:Model_reg}. 
\begin{definition}
	\label{def:model_reg}
	Let Assumption \ref{ass:weak} hold with $\Omega\in\bbr^d$, $ d \in \{2,3\} $, $ T > 0 $, and $ (\bu_0, \bbb_0, \phi_0) \in L_\sigma^2(\Omega) \times L^2(\Omega; \bbr_{\mathrm{sym}}^{d \times d}) \times W^{1,2}(\Omega) $ with $ \bbb_0 $ positive definite a.e.~in $\Omega$, $\tr\ln\bbb_0 \in L^1(\Omega)$ and $ \abs{\phi_0} \leq 1 $ a.e.~in $ \Omega $.
	We call the quadruple $ (\bu, \phi, q, \bbb) $ a \textit{finite energy} weak solution to \eqref{eqs:Model_reg} with initial data $ (\bu_0, \bbb_0, \phi_0) $, provided that
	\begin{enumerate}
		\item the quadruple $ (\bu, \phi, q, \bbb) $ satisfies
		\begin{gather*}
			\bu \in C_w([0,T]; L_\sigma^2(\Omega)) 
			\cap L^2(0,T; W_0^{1,2}(\Omega; \bbr^d)); \\
%			\cap W^{1,\frac{4}{3}}(0,T; [W_0^{1,2}(\Omega; \bbr^d)]'); \\
			\phi \in C_w([0,T]; W^{1,2}(\Omega)) 
			\cap L^2(0,T; W^{2,2}(\Omega)) \text{ with } \phi\in(-1,1) \text{ a.e.~in } Q_T; \\
			W'(\phi) \in L^2(0,T; L^2(\Omega)), \quad
%			\cap W^{1,2}(0,T; [W_0^{1,2}(\Omega)]'); \\
			q \in L^2(0,T; W^{1,2}(\Omega)); \\
			\bbb \text{ is symmetric positive definite a.e.~in } Q_T; \\
			\bbb \in C_w([0,T]; L^2(\Omega; \bbr_{\mathrm{sym}}^{d \times d})) 
			\cap L^2(0,T; W^{1,2}(\Omega; \bbr_{\mathrm{sym}}^{d \times d}));\\
			\tr \ln \bbb \in L^\infty(0,T; L^1(\Omega)) 
			\cap L^2(0,T; W^{1,2}(\Omega)); 
			\end{gather*}
   % \footnote{Add th $\tr \ln$ regularity}
		\item for all $ t \in (0,T) $ and all $ \bw \in C^\infty([0,T]; C_0^\infty(\Omega; \bbr^d)) $ with $ \Div \bw = 0 $, we have
		\begin{align}
			\int_0^t \int_\Omega & \Big( \rho(\phi) \bu \cdot \pt \bw 
			+ (\rho(\phi) \bu \otimes \bu) : \nabla \bw 
			- \rho'(\phi) (\bu \otimes m(\phi) \nabla q) : \nabla \bw\Big) \dx \d\tau 
        \nonumber\\
			& \quad - \int_0^t \int_\Omega \Big( 
			\nu(\phi) (\nabla \bu + \tran{\nabla} \bu) : \nabla \bw 
			+ \sR_\eta \big[\mu_\eta(\phi)(\bbb - \bbi)\big] : \nabla \bw
			\Big) \dx \d\tau
			\label{eqs:weak-reg-u-formulation} \\
			& = - \int_0^t \int_\Omega q \nabla \phi \cdot \bw \dx \d\tau
			- \int_0^t \int_\Omega \sR_\eta \Big[\frac{\mu_\eta(\phi)}{2} \nabla \tr(\bbb - \ln \bbb - \bbi)\Big] \cdot \bw \dx \d\tau \nonumber \\
			& \quad + \int_\Omega \rho(\phi(\cdot, t)) \bu(\cdot, t) \cdot \bw(\cdot, t) \dx
			- \int_\Omega \rho(\phi_0) \bu_0 \cdot \bw(\cdot, 0) \dx; \nonumber
		\end{align}
		\item for all $ t \in (0,T) $ and all $ \xi \in C^\infty([0,T]; C^1(\overline{\Omega})) $, we have
		\begin{align}
			\int_0^t \int_\Omega & \phi \big(  \pt \xi 
			+ \bu \cdot \nabla \xi \big) \dx \d\tau 
				- \int_0^t \int_\Omega m(\phi) \nabla q \cdot \nabla \xi \dx \d\tau 
			\label{eqs:weak-reg-phi-formulation} \\
			& = \int_\Omega \phi(\cdot, t) \xi(\cdot, t) \dx
			- \int_\Omega \phi_0 \xi(\cdot, 0) \dx; \nonumber
		\end{align}
		\item for a.e.~$ (x, t) \in Q_T $, we have
		\begin{equation}
		    \label{eqs:weak-reg-q-formulation}
			q = W'(\phi) - \Delta \phi
			+ \eta \ptial{t} \phi
			+ \sR_\eta \Big[\frac{\mu_\eta'(\phi)}{2} \tr(\bbb - \ln \bbb - \bbi)\Big];
		\end{equation}
		\item for all $ t \in (0,T) $ and all $ \bbc \in C^\infty(\overline{Q_T}; \bbr_{\mathrm{sym}}^{d \times d}) $, we have
		\begin{align}
			\int_0^t \int_\Omega & \Big( \bbb : \pt \bbc 
				{ - (\sR_\eta \bu \cdot \nabla) \bbb :  \bbc} \Big) \dx \d\tau \nonumber\\
			& \quad + \int_0^t \int_\Omega \Big(\big(\nabla \sR_\eta \bu \bbb + \bbb \tran{\nabla} \sR_\eta \bu \big) : \bbc 
				- \kappa \nabla \bbb : \nabla \frac{\bbc}{\mu_\eta(\phi)} \Big) \dx \d\tau 
			\label{eqs:weak-reg-B-formulation} \\
			& = \int_0^t \int_\Omega (\bbb : \bbc - \tr \bbc) \dx \d\tau
				+ \int_\Omega \bbb(\cdot, t) : \bbc(\cdot, t) \dx
			- \int_\Omega \bbb_0 : \bbc(\cdot, 0) \dx; \nonumber
		\end{align}
		\item for a.e.~$ t \in (0,T) $, the following energy estimate holds 
		\begin{align}
			& \cE_\eta(t)
				+ \int_0^t \norm{\sqrt{\nu(\phi)}(\nabla \bu + \tran{\nabla} \bu)(\tau)}_{L^2}^2 \dtau 
				\nonumber \\
			& \qquad
				+ \int_0^t \norm{\sqrt{m(\phi)}\nabla q(\tau)}_{L^2}^2 \dtau
				+ \eta \int_0^t \norm{\ptial{t} \phi(\tau)}_{L^2}^2 \dtau
				% + \int_0^t \Big( \norm{\bbb(\tau)}_{L^2}^2
				% + \kappa \norm{\nabla \bbb(\tau)}_{L^2}^2 \Big) \dtau 
				\label{eqs:energy-dissipation-reg-inequality} \\
			& \qquad + \int_0^t \Bigg(\norm{\frac{\mu_\eta(\phi)}{2}\tr(\bbb
				+ \inv{\bbb} - 2\bbi)(\tau)}_{L^1} 
				+ \frac{\kappa}{d} \norm{\nabla \tr \ln \bbb(\tau)}_{L^2}^2\Bigg) \dtau 
				\leq {\cE_\eta^0},
				\nonumber
		\end{align}
		where {the energy of the regularized system is defined by \eqref{eqs:energy-reg} and $ \cE_\eta^0 \leq C(\overline{\mu}, \underline{\mu}) \cE_0 $ uniformly regarding $ \eta $}.
	\end{enumerate}
\end{definition}
%\todo{Adapt notation energy. Do we also have a constant here, i.e. $C E_{\eta,0}$?}
Our main result in this section will be as follows.
\begin{theorem}
	\label{thm:main-reg}
	Let Assumption \ref{ass:weak} hold with $\Omega\in\bbr^d$, $ d \in \{2,3\} $, $ (\bu_0, \bbb_0, \phi_0) \in L_\sigma^2(\Omega) \times L^2(\Omega; \bbr_{\mathrm{sym}}^{d \times d}) \times W^{1,2}(\Omega) $ with $\tr \ln \bbb_0 \in L^1(\Omega) $, $ \bbb_0 $ positive definite and $ \abs{\phi_0} \leq 1 $ a.e.~in $\Omega$ and $ \fint_\Omega \phi_0 \dx \in (-1, 1) $. Then there exists a finite energy weak solution $ (\bu, \phi, q, \bbb) $ of \eqref{eqs:Model_reg} in the sense of Definition \ref{def:model_reg}.
\end{theorem}
\begin{remark}
    \label{rem:dimension-discussion} The theorem is {finally} proved in Section \ref{sec:proof-reg}.
    Thanks to the regularization $\sR_\eta$ we introduced, the existence theorem can be established both in two and three dimensions.
    %, without employing the Gagliardo--Nirenberg inequality for the estimate of $ \bbb $. 
    This cannot be reached when passing to the limit as $ \eta \rightarrow 0 $ in Section \ref{sec:proof-of-weak-solution}, where we need a stronger estimate of $\bbb$ (uniform in $\eta$) in Section \ref{sec:a-priori-reg-strong}, which is restricted to two dimensions.
\end{remark}

    In the rest of this section, we only proceed with the case $ d = 3 $, while the other case $ d = 2 $ is even simpler with better Sobolev embeddings. For this, we refer to the final proof of Theorem \ref{thm:main} in Section \ref{sec:proof-weak-limiting}.
% {State the main strategy to prove Theorem 4.2 again as in Section 1.}
% {In the following, we will prove Theorem \ref{thm:main-reg} via a hybrid time discretization approximate scheme, in view of the solvability of AGG model in \cite{ADG2013} and the regularized Oldroyd-B equation from Section \ref{sec:B}. More precisely, by means of the time discretization technique from \cite{ADG2013}, we exploit a piecewise continuous data from $\bbb$ with time integral averaging, cf. Section \ref{sec:time-discretization}. Then, by the Leray--Schauder principal, we establish the solvability of time-discrete solutions in Section \ref{sec:existence-hybrid}. Later, for such time-discrete solutions, we construct approximate solutions to \eqref{eqs:Model_reg} in Section \ref{sec:construction-approx}. Finally we establish corresponding uniform estimates in Section \ref{sec:proof-reg} by a discrete energy inequality and we pass to the limit.}
{
To prove Theorem \ref{thm:main-reg}, we employ a hybrid time discretization scheme, drawing on the solvability of the AGG model as presented in \cite{ADG2013} and the regularized Oldroyd-B equation from Section \ref{sec:B}.
We begin with the time discretization technique from \cite{ADG2013}, which involves using piecewise continuous data for $\bbb$ with time integral averaging (see Section \ref{sec:time-discretization}).
Then, applying the Leray--Schauder principle, we establish the solvability of the time-discrete solutions in Section \ref{sec:existence-hybrid}. For these time-discrete solutions, we construct approximate solutions to \eqref{eqs:Model_reg} in Section \ref{sec:construction-approx}. Finally, we derive the corresponding uniform estimates in Section \ref{sec:proof-reg} using a discrete energy inequality. We then pass to the limit, thus completing the proof of Theorem \ref{thm:main-reg}. }

In the presence of the singular, non-convex potential $W(\phi)$, we first rewrite the chemical potential. 
The idea is to use the subdifferential of a convex potential, which was also employed in \cite{AW2007} for the Cahn--Hilliard equation and in \cite{ADG2013} for the Cahn--Hilliard--Navier--Stokes equation. First we define an extended potential $ \widetilde{W} $ through
\begin{equation*}
    \widetilde{W} : \bbr \rightarrow \bbr, \quad
    \widetilde{W}(s) = \left\{
        \begin{aligned}
            & W(s) &&\ \text{if } s \in [-1, 1], \\
            & + \infty &&\ \text{else},
        \end{aligned}
    \right.
\end{equation*}
where $ W $ is the potential fulfilling Assumption \ref{ass:free-energy}. Note that $ \widetilde{W} $ is not necessarily convex by this assumption. Hence, we define $ \widetilde{W}_0(r) \coloneqq \widetilde{W}(r) + \frac{\omega}{2} r^2 $, which satisfies $ \widetilde{W}_0 \in C([-1,1]) \cap C^2((-1,1)) $ and is convex due to the Assumption \ref{ass:free-energy}. In particular, we have $ \widetilde{W}'(r) = \widetilde{W}_0'(r) - \omega r $. Now \eqref{eqs:CH-q_reg} is equivalent to
\begin{equation*}
    q + \omega \phi = \widetilde{W}_0'(\phi) - \Delta \phi
		+ \eta \ptial{t} \phi
        + \sR_\eta \Big[\frac{\mu_\eta'(\phi)}{2} \tr(\bbb - \ln \bbb - \bbi)\Big], \quad \text{a.e.~in } Q_T.
\end{equation*}
Following \cite{AW2007,ADG2013}, we define a modified energy $ \widetilde{E}: L^2(\Omega) \rightarrow \bbr \cup \{+\infty\} $ by
\begin{equation*}
    \widetilde{E}(\phi) = \left\{
        \begin{aligned}
            & \onehalf \int_\Omega \abs{\nabla \phi}^2 + \int_\Omega \widetilde{W}_0(\phi) &&\ \text{for } \phi \in \cD(\widetilde{E}), \\
            & + \infty &&\ \text{else},
        \end{aligned}
    \right.
\end{equation*}
with the domain of definition
\begin{equation*}
    \cD(\widetilde{E}) = \left\{
        \phi \in W^{1,2}(\Omega) : - 1 \leq \phi \leq 1 \text{ a.e.}
    \right\}.
\end{equation*}
Then, as in \cite{ADG2013}, the domain of the definition of the subgradient $ \partial \widetilde{E} $ is
\begin{equation*}
    \cD(\partial \widetilde{E}) = \left\{
        \phi \in W^{2,2}(\Omega) : 
            \widetilde{W}_0'(\phi) \in L^2(\Omega), \ 
            \widetilde{W}_0''(\phi) \abs{\nabla \phi}^2 \in L^1(\Omega),\ 
            \ptial{\bn} \phi |_{\partial \Omega} = 0
    \right\}.
\end{equation*}
With these definitions, one has
\begin{equation*}
    \partial \widetilde{E}(\phi) = - \Delta \phi + \widetilde{W}_0'(\phi) \quad \text{for } \phi \in \cD(\partial \widetilde{E}).
\end{equation*}
Note that $ \partial \widetilde{E} $ is maximal monotone by the convexity of $ \widetilde{W}_0 $ and the lower semicontinuity.
Moreover, it holds that
\begin{equation}
    \label{eqs:subdifferential-estimate}
    \norm{\phi}_{W^{2,2}}^2 + \normm{\widetilde{W}_0'(\phi)}_{L^2}^2
    + \int_\Omega \widetilde{W}_0''(\phi(x)) \abs{\nabla \phi(x)}^2 \dx
    \leq C \Big( \normm{\partial \widetilde{E}(\phi)}_{L^2}^2 + \norm{\phi}_{L^2}^2 + 1 \Big).
\end{equation}

% \begin{remark}
% Note that in \cite{ADG2013} a more complicated subgradient $\partial \widetilde{E}(\phi) = - \Delta A(\phi) + \widetilde{W}_0'(A(\phi))$ was considered to resolve the problem caused by the phase-dependent free energy $\int_\Omega \big( \frac{a(\phi)}{2} \abs{\nabla \phi}^2 + W(\phi) \big) \dx$, where $a(\phi)$ is some positive function and $A(\phi)$ is related to $a(\phi)$, see \cite{ADG2013}.
% However, in this paper, the simpler case of $ a(\phi) = 1 $ is studied. 
% % 	due to the elastic relaxation we have in the equation of chemical potential an additional regularization term $ \sR_\eta \big[\frac{\mu_\eta'(\phi)}{2} \tr(\bbb - \ln \bbb - \bbi)\big] $, which does not converge almost everywhere in $Q_T$, as $ \eta \rightarrow 0 $. Therefore, 
% We still take the advantage of $\partial \widetilde{E}(\phi) = - \Delta \phi + \widetilde{W}_0'(\phi)$, which is a maximal monotone operator, to pass to the limit in the final proof, see Section \ref{sec:proof-reg}.
%  Another possible strategy is to approximate the singular potential with a sequence of regular potentials, for which we also refer to, e.g., \cite{GGW2019} for the case with dynamic boundary conditions and $ a(\phi) = 1 $, where the authors take the idea of a convex decomposition of a singular potential and a regular approximation. 
% \end{remark}

\subsection{Hybrid Implicit Time Discretization}
\label{sec:time-discretization}
Inspired by the well-posedness result of the Oldroyd-B equation \eqref{eqs:B-equation} and the implicit time discretization argument in \cite{ADG2013} to solve the AGG model with unmatched densities and singular potential, we propose a \textit{hybrid implicit time discretization} for the whole regularized system \eqref{eqs:Model_reg}. Namely, for the AGG part, we employ a time discretization similarly to \cite{ADG2013}, while the Oldroyd-B part is solved continuously in time with the help of Theorem \ref{thm:B-equation} on each discrete time interval.

More precisely, let $T>0$, $ h = \frac{T}{N} $ for $ N \in \bbn $, $ \bbb_0 \in L^2(\Omega; \bbr_{\mathrm{sym}}^{d \times d}) $ positive definite a.e.~in $\Omega$ with $ \tr \ln \bbb_0 \in L^1(\Omega) $, and, for $k\in\{0,...,N-1\}$, let $ \bu_k \in L^2_\sigma(\Omega) $, $ \phi_k \in W^{1,2}(\Omega) $ with $ W'(\phi_k) \in L^2(\Omega) $, and $ \rho_k = \onehalf (\rho_1 + \rho_2) + \onehalf (\rho_2 - \rho_1) \phi_k $ be given. We construct $ (\bu, \phi, q, \widetilde{\bbb}) = (\bu_{k + 1}, \phi_{k + 1}, q_{k + 1}, \widetilde{\bbb}_{k + 1}) $ as a solution of the following nonlinear system.

	\begin{subequations}
		\label{eqs:regularized-discret}
		Let $ I_{k+1} \coloneqq (t_k,t_{k+1}) $ with $ t_k = kh $, $ k \in \{0,...,N-1\}$. Find $ (\bu, \phi, q, \widetilde{\bbb}) $ with $ \bu \in W_{0,\sigma}^{1,2}(\Omega) $, $ \phi \in \cD(\partial \widetilde{E}) $, $ \mu \in W_n^{2,2} = \{u \in W^{2,2}(\Omega): \partial_\bn u |_{\partial \Omega} = 0 \text{ on } \partial \Omega\} $, and $ \widetilde{\bbb} \in L^\infty(I_{k+1};L^2(\Omega;\bbr_{\mathrm{sym}}^{d \times d})) \cap L^2(I_{k+1};W^{1,2}(\Omega;\bbr_{\mathrm{sym}}^{d \times d})) $ with $ \widetilde\bbb $ positive definite a.e.~in $ \Omega $ and $ \tr \ln \widetilde\bbb \in L^1(I_{k+1};L^1(\Omega)) $, such that $ \overline{\bbb}_{k+1} = \frac{1}{h} \int_{I_{k+1}} \widetilde{\bbb}(t) \dt $, where $ \widetilde{\bbb} $ is a weak solution of \eqref{eqs:B-equation} in the time interval $ \overline{I}_{k+1} $ with initial data $ \overline{\bbb}_k $ and data $ (\bv, \xi)(t) = (\bu, \phi_k) $, $ t \in I_{k+1} $, where $\overline\bbb_0 = \bbb_0$, that is,
		\begin{equation}
			\label{eqs:B-equation-discrete}
			\begin{alignedat}{3}
				\pt \widetilde{\bbb} + \sR_\eta \bu \cdot \nabla \widetilde{\bbb} - \widetilde{\bbb} \tran{\nabla} \sR_\eta \bu - \nabla \sR_\eta \bu \widetilde{\bbb} + (\widetilde{\bbb} - \bbi) & = \frac{\kappa}{\mu_\eta(\phi_k)} \Delta \widetilde{\bbb},  &&  \tin \Omega \times I_{k+1}, \\
				\ptial{\bn} \widetilde{\bbb} & = {\mathbb{O}}, && \ton \partial \Omega \times I_{k+1}, \\
				\widetilde{\bbb}(t_k) & = \overline{\bbb}_k, && \tin \Omega,
			\end{alignedat}
		\end{equation}
        in the weak sense. %\todo{changed to zero matrix, see also (1.1)}
		Note that we have continuous dependence of $ \widetilde{\bbb} $ on $ (\bu, \phi_k, \overline\bbb_k) $ by Theorem \ref{thm:B-equation}. Moreover, we consider the following discrete problem with time-averaged terms with respect to $ \widetilde{\bbb} $:
		\begin{equation}
			\label{eqs:AGG-discret-u}
			\begin{aligned}
				& \inner{\frac{\rho \bu - \rho_k \bu_k}{h}}{\bw}
				+ \inner{\Div (\rho_k \bu \otimes \bu)}{\bw} 
				+ \inner{\Div (\bu \otimes \bJ)}{\bw} \\
				& \qquad + \inner{\nu(\phi_k)(\nabla \bu + \tran{\nabla} \bu)}{\nabla \bw} 
				- \frac1h \int_{I_{k+1}} \inner{\Div \sR_\eta \big[\mu_\eta(\phi_k)(\widetilde\bbb - \bbi)\big]}{\bw} \dt \\
				& \qquad = \inner{q \nabla \phi_k}{\bw}
				+ \frac1h \int_{I_{k+1}} \inner{\sR_\eta \Big[\frac{\mu_\eta(\phi_k)}{2}   \nabla \tr (\widetilde\bbb - \ln\widetilde\bbb - \bbi) \Big]}{\bw} \dt
			\end{aligned}
		\end{equation}
		for {a.e.~$t\in I_{k+1}$} and all $ \bw \in C_0^\infty(\overline{\Omega}) $ with $ \Div \bw = 0 $, where
		\begin{equation*}
			\bJ = \bJ_{k + 1} = - \rho'(\phi_k) m(\phi_k) \nabla q_{k + 1}
			= - \rho'(\phi_k) m(\phi_k) \nabla q.
		\end{equation*}
		In addition, for a.e.~$ x \in \Omega $ and {$t\in I_{k+1}$},,
		\begin{gather}
			\label{eqs:AGG-discret-phi}
			\frac{\phi - \phi_k}{h} + \bu \cdot \nabla \phi_k = \Div(m(\phi_k) \nabla q), \\
			\label{eqs:AGG-discret-q}
			\begin{aligned}
				q + \omega \frac{\phi + \phi_k}{2} & = \widetilde{W}_0'(\phi) - \Delta \phi
				+ \eta \frac{\phi - \phi_k}{h} \\
				& \quad
				+ \sR_\eta \Big[\onehalf \frac{\mu_\eta(\phi) - \mu_\eta(\phi_k)}{\sR_\eta (\phi - \phi_k)} \frac{1}{h} \int_{I_{k+1}} \tr(\widetilde{\bbb} - \ln \widetilde{\bbb} - \bbi)(t) \dt\Big].
			\end{aligned}
		\end{gather}
	\end{subequations}

\begin{remark}
    % \footnote{Explain the reason for the terms regarding$\widetilde{\bbb}$.}
In the above system, we distinguish between time-discrete functions and functions that are continuous in time. That is the reason why we write a tilde on the top of the {tensor} $\widetilde\bbb$. 
In order to make sure the time-continuous problem \eqref{eqs:B-equation-discrete} and the time-discrete subsystem \eqref{eqs:AGG-discret-u}--\eqref{eqs:AGG-discret-q} are compatible with each other, we make use of time averages for the terms containing $\widetilde\bbb$ in the time-discrete problem. In fact, it is necessary to take advantage of all the information of $\widetilde\bbb$ in a fixed time interval so that the terms regarding $\widetilde\bbb$ are well-defined concerning the regularity. 

Another motivation is to get the uniform energy estimate for the whole hybrid discrete problem with respect to $h$, for which we approximate $ \mu_\eta'(\phi) $ due to the variable shear modulus by the difference quotient $\frac{\mu_\eta(\phi) - \mu_\eta(\phi_k)}{\sR_\eta (\phi - \phi_k)}$, and proceed with the energy estimate \eqref{eqs:B-identity-discrete} for $\widetilde\bbb$ in the interval $I_{k+1}$. With all these ingredients, we establish a uniform estimate %with energy inequality for the discrete AGG part \eqref{eqs:AGG-energy-discrete-B} similar to \cite{ADG2013}, 
by canceling the mixed terms corresponding to $\widetilde\bbb$. %thanks to the time-averaged approximation for these terms.

% In the time-discrete equations \eqref{eqs:AGG-discret-u}--\eqref{eqs:AGG-discret-q}, we use two different time-approximations of $\widetilde\bbb$. 
% Namely, the time-averaged terms in \eqref{eqs:AGG-discret-u} are used to match with the corresponding terms from \eqref{eqs:B-discrete-tilde} which will also be time-averaged in the testing procedure. 
% In contrast, we will treat the time-averaged term in \eqref{eqs:AGG-discret-q} directly. 
%$\tr(\bbb - \ln \bbb - \bbi)= \tr(\widetilde\bbb(t_{k+1}) - \ln \widetilde\bbb(t_{k+1}) - \bbi)$ in \eqref{eqs:AGG-discret-q} is used for a discrete version of the time-derivative of the energy.
    % Here wewrite\footnote{Did I write the uncomplete sentence by mistake?OK\\No, I just want to formulate them together}  This is of great importance for us to derive the uniform discrete estimate, see  and \eqref{eqs:B-identity-discrete} in next subsection.
\end{remark}

\begin{remark}
The approximation of $\mu_\eta'(\phi)$ with the difference quotient $\frac{\mu_\eta(\phi) - \mu_\eta(\phi_k)}{\sR_\eta (\phi - \phi_k)}$ is well-defined for $\phi \neq \phi_k$, as $\mu \in C^1(\bbr)$ with bounded derivative in view of Assumption \ref{ass:lower-upper-bounds}. Indeed, by the mean value theorem, it holds
    \begin{equation*}
        \frac{\mu_\eta(\phi) - \mu_\eta(\phi_k)}{\sR_\eta (\phi - \phi_k)}
        = \mu_\eta'(\xi) \in [\underline{\mu}' , \overline{\mu}'],
    \end{equation*}
    {for some $\xi$ between $\phi$ and $\phi_k$}. Moreover, as $\phi_k \to \phi$ a.e.~in $Q_T$, it holds
    \begin{equation*}
        \frac{\mu_\eta(\phi) - \mu_\eta(\phi_k)}{\sR_\eta (\phi - \phi_k)}
        \to \mu_\eta'(\phi).
    \end{equation*}
\end{remark}
% Strategy:
% \begin{itemize}
% \item test $\tilde \bbb$ equation with $\mu_\eta(\phi_k) ( \bbi - \tilde\bbb^{-1})$
% \item then take time average, i.e., $\frac1h \int_0^h ... \dt$ in the resulting equation.
% \item in $q$-equation: take $\sR_\eta \Big[\onehalf \frac{\mu_\eta(\phi) - \mu_\eta(\phi_k)}{\sR_\eta (\phi - \phi_k)} \tr(\bbb - \ln \bbb - \bbi)\Big]$
% \item in $\bu$-equation: take $\frac1h \int_0^h  ... \dt$ for the whole equation. 
% \item For simplicity of analysis: take all unnecessary constants to 1
% %\item We have good control on $\tilde\bbb$ after these uniform a priori estimates, as we are continuous in time. The energy inequality also holds for $\tilde\bbb$ and $\bbb=\frac1h \int_0^h  \tilde\bbb \dt$. 
% \item First uniform estimates to get knowledge for the sequence $(\bu_k, \phi_k, q_k)$ (with remark, using discrete Gronwall). The estimate has nothing to do with approximation of $\bbb$. Then, we do uniform estimate for continuous $\tilde\bbb$ on the whole time interval.\footnote{We do not need this additional step. We can now do both at once. I adapted the system from above (see system (4.12b-d) so that this will already lead to the uniform estimates, also for $\tilde\bbb$!}
% \item approximate solution: $\bu^N, \phi^N, q^N$ as already written, and for $\bbb$: define $\bbb^N$ as continuous $\tilde\bbb$ 
% % \item approximate solutions: we can also do $\bbb^N$ with $\bbb$ instead of $\tilde\bbb$. It does not matter.
% \item integration by parts in the term $\nabla \phi$ to disappear (in uniform a priori estimates)
% \end{itemize}

	\begin{remark}
	    By simply integrating \eqref{eqs:AGG-discret-phi} over $\Omega$ and applying integration by parts over $ \Omega $ together with $ \Div \bu = 0 $ and the Neumann boundary condition $ \ptial{\bn} q |_{\partial \Omega} = 0 $, one infers that
	    \begin{equation*}
	        \int_\Omega \phi(t) \dx = \int_\Omega \phi_k \dx = \int_\Omega \phi_0 \dx
	    \end{equation*}
	    {for a.e.~$t\in I_{k+1}$ and any $k \in \{0, \dots, N-1\}$}, which implies that the mass is conserved for the time-discrete problem.
	\end{remark}
	In a similar fashion as in \cite[Lemma 4.2]{ADG2013}, we have the following lemma, which will be frequently used later. For the reader's convenience, we point out the main differences compared to \cite[Lemma 4.2]{ADG2013} arising from the additional terms.
	\begin{lemma}
	    \label{lem:convex-potential-L1-q}
	    Let $ \phi \in \cD (\partial \widetilde{E}) $ and $ q \in W^{1,2}(\Omega) $ be solving \eqref{eqs:AGG-discret-q} with given $ \phi_k \in W^{2,2}(\Omega) $ satisfying $ \abs{\phi_k} \leq 1 $ and
	    \begin{equation*}
	        \sqrt{\eta} \frac{\phi - \phi_k}{h} \in L^2(\Omega), \quad 
	        \frac{1}{\abs{\Omega}} \int_\Omega \phi \dx
	        = \frac{1}{\abs{\Omega}} \int_\Omega \phi_k \dx
	        \in (-1, 1).
	    \end{equation*}
	    Moreover, let $ \widetilde{\bbb} \in L^2(I_{k+1}; L^2(\Omega; \bbr_{\mathrm{sym}}^{d \times d}))  $ be positive definite with $ \tr \ln \widetilde{\bbb} \in L^2(I_{k+1}; L^2(\Omega)) $. Then, there exists a constant $ C > 0 $ depending on $ \eta $ and $\int_\Omega \phi_k \dx $ such that
	    \begin{align*}
	        \norm{\widetilde{W}_0'(\phi)}_{L^1(\Omega)}
	        + \abs{\int_\Omega q \dx}
	        \leq C \Big( 
	            & \norm{\nabla q}_{L^2}^2
	            + \norm{\nabla \phi}_{L^2}^2 
	            + \norm{\nabla \phi_k}_{L^2}^2 \\
	            & + \eta \norm{\frac{\phi - \phi_k}{h}}_{L^2}^2 + \frac1h \norm{\tr(\widetilde{\bbb} - \ln \widetilde{\bbb} - \bbi)}_{L^2(I_{k+1}; L^2(\Omega))}^2
	            + 1
	        \Big).
	    \end{align*}
	\end{lemma}
	\begin{proof}
	    %The proof is adapted from \cite[Lemma 4.2]{ADG2013}, where a similar result was derived. 
     First of all, testing \eqref{eqs:AGG-discret-q} with $ \phi - \overline{\phi} $, where $ \overline{\phi} \coloneqq \frac{1}{\Omega} \int_\Omega \phi \dx $ is the mean value of $ \phi $ over $ \Omega $, leads to
	    \begin{align*}
	        & \int_\Omega q (\phi - \overline{\phi}) \dx 
	        + \int_\Omega \omega \frac{\phi + \phi_k}{2} (\phi - \overline{\phi}) \dx \\
			& = \int_\Omega \widetilde{W}_0'(\phi) (\phi - \overline{\phi}) \dx 
			+ \int_\Omega \nabla \phi \cdot \nabla (\phi - \overline{\phi}) \dx 
			+ \int_\Omega \eta \frac{\phi - \phi_k}{h} (\phi - \overline{\phi}) \dx \\
			& \quad 
			+ \int_\Omega \frac{\mu_\eta(\phi) - \mu_\eta(\phi_k)}{2\sR_\eta(\phi - \phi_k)} \Big( \frac{1}{h} \int_{I_{k+1}} \tr(\widetilde{\bbb} - \ln \widetilde{\bbb} - \bbi)(t) \dt\Big) \sR_\eta (\phi - \overline{\phi}) \dx.
	    \end{align*}
	    By means of H\"older's and Young's inequalities, and the Poincar\'e--Wirtinger inequality, one has the following estimates:
	    \begin{align*}
	        & \int_\Omega q (\phi - \overline{\phi}) \dx 
	        = \int_\Omega (q - \overline{q}) \phi \dx
	        \leq \norm{q - \overline{q}}_{L^2} \norm{\phi}_{L^2}
	        \leq \norm{\nabla q}_{L^2}^2 + C (\norm{\nabla \phi}_{L^2} + 1), \\
	        & \int_\Omega \omega \frac{\phi + \phi_k}{2} (\phi - \overline{\phi}) \dx
	        \leq C ( \norm{\phi}_{L^2} + \norm{\phi_k}_{L^2} ) \norm{\nabla \phi}_{L^2}
	        \leq C (\norm{\nabla \phi}_{L^2}^2 + \norm{\nabla \phi_k}_{L^2}^2 + 1), \\
	        & \int_\Omega \nabla \phi \cdot \nabla (\phi - \overline{\phi}) \dx 
	        \leq C \norm{\nabla \phi}_{L^2}^2, \\
			& \int_\Omega \eta \frac{\phi - \phi_k}{h} (\phi - \overline{\phi}) \dx 
			\leq \eta \norm{\frac{\phi - \phi_k}{h}}_{L^2} \norm{\phi - \overline{\phi}}_{L^2}
			\leq \eta \norm{\frac{\phi - \phi_k}{h}}_{L^2}^2 
			+ \eta \norm{\nabla \phi}_{L^2}^2, \\
			& \int_\Omega \frac{\mu_\eta(\phi) - \mu_\eta(\phi_k)}{2\sR_\eta(\phi - \phi_k)} \Big( \frac{1}{h} \int_{I_{k+1}} \tr(\widetilde{\bbb} - \ln \widetilde{\bbb} - \bbi)(t) \dt \Big) \sR_\eta (\phi - \overline{\phi}) \dx \\
			& \qquad \qquad \qquad
			\leq \frac{\overline{\mu}'}{4} \frac1h  \norm{\tr(\widetilde{\bbb} - \ln \widetilde{\bbb} - \bbi)}_{L^2(I_{k+1}; L^2(\Omega))}^2 + \frac{\overline{\mu}'}{4} \norm{\nabla \phi}_{L^2}^2,
	    \end{align*}
	    where $ C $ depends on $ \int_\Omega \phi_k \dx $. Moreover, it follows from the assumption $ \lim_{\phi \rightarrow \pm 1} \widetilde{W}_0'(\phi) \rightarrow \pm \infty $ and $ \widetilde{W}_0' \in C([-1 + \frac{\varepsilon}{2}, 1 - \frac{\varepsilon}{2}]) $ due to \ref{ass:free-energy}, together with the fact that $ \overline{\phi} \in (-1 + \varepsilon, 1 - \varepsilon) $ for some $ \varepsilon > 0 $, that
	    \begin{equation*}
	        \widetilde{W}_0'(\phi) (\phi - \overline{\phi}) 
	        \geq 
	        C \abs{\widetilde{W}_0'(\phi)} - C_1,
	    \end{equation*}
	    by discussing the range of $ \phi $ in three cases: $ [-1, - 1 + \frac{\varepsilon}{2}] $, $ [-1 + \frac{\varepsilon}{2}, 1 - \frac{\varepsilon}{2}] $, $ [1 - \frac{\varepsilon}{2}, 1] $.
	    As a consequence, 
	    \begin{equation*}
	        \int_\Omega \widetilde{W}_0'(\phi) (\phi - \overline{\phi}) \dx
	        \geq 
	        C \int_\Omega \abs{\widetilde{W}_0'(\phi)} \dx - C_2.
	    \end{equation*}
	    Collecting all the estimates from above yields
	    \begin{align*}
	        \norm{\widetilde{W}_0'(\phi)}_{L^1(\Omega)}
	        \leq C \Big( 
	            \norm{\nabla q}_{L^2}^2
	            & + \norm{\nabla \phi}_{L^2}^2 
	            + \norm{\nabla \phi_k}_{L^2}^2 \\
	            & + \eta \norm{\frac{\phi - \phi_k}{h}}_{L^2}^2 
	            + \frac1h  \norm{\tr(\widetilde{\bbb} - \ln \widetilde{\bbb} - \bbi)}_{L^2(I_{k+1}; L^2(\Omega))}^2
	            + 1
	        \Big),
	    \end{align*}
	    which thereafter implies that
	    \begin{align*}
	        \abs{\int_\Omega q \dx}
	        \leq C \Big( 
	            \norm{\nabla q}_{L^2}^2
	            & + \norm{\nabla \phi}_{L^2}^2 
	            + \norm{\nabla \phi_k}_{L^2}^2 \\
	            & + \eta \norm{\frac{\phi - \phi_k}{h}}_{L^2}^2 
	            + \frac1h  \norm{\tr(\widetilde{\bbb} - \ln \widetilde{\bbb} - \bbi)}_{L^2(I_{k+1}; L^2(\Omega))}^2
	            + 1
	        \Big),
	    \end{align*}
	    by integrating \eqref{eqs:AGG-discret-q} over $ \Omega $, for which we used $ \int_\Omega \Delta \phi \dx = \int_{\partial \Omega} \ptial{\bn} \phi \,\d \mathcal{H}^{d-1} = 0 $.
	\end{proof}

\begin{remark}
The condition $\tr \ln \widetilde{\bbb} \in L^2(I_{k+1}; L^2(\Omega))$ in Lemma \ref{lem:convex-potential-L1-q} follows directly from Theorem \ref{thm:B-equation} and the Poincar\'e--Wirtinger inequality.
\end{remark}
 
\subsection{Existence for the Hybrid Discrete System}
\label{sec:existence-hybrid}
Now we present and prove the existence of solutions to the hybrid discrete system \eqref{eqs:regularized-discret}.
\begin{proposition}
    \label{prop:existence-hybrid}
	Let $ \bbb_0 \in L^2(\Omega; \bbr_{\mathrm{sym}}^{d \times d}) $ be positive definite a.e.~in $\Omega$ with $ \tr \ln \bbb_0 \in L^1(\Omega) $, and, for $k\in\{0,...,N-1\}$, let $ \bu_k \in L_\sigma^2(\Omega) $, $ \phi_k \in W^{2,2}(\Omega) $, and $ \rho_k = \onehalf (\rho_1 + \rho_2) + \onehalf (\rho_2 - \rho_1) \phi_k $ be given. Then there are some $ (\bu, \phi, q, \widetilde\bbb) \in W_{0,\sigma}^{1,2}(\Omega) \times \cD(\partial \widetilde{E}) \times W_n^{2,2} \times L^\infty(I_{k+1};L^2(\Omega;\bbr_{\mathrm{sym}}^{d \times d})) \cap L^2(I_{k+1};W^{1,2}(\Omega;\bbr_{\mathrm{sym}}^{d \times d})) $ with $ \widetilde\bbb $ positive definite a.e.~in $ \Omega $ and $ \tr \ln \widetilde\bbb \in L^1(\Omega \times I_{k+1}) $ solving \eqref{eqs:regularized-discret} and satisfying the discrete energy inequality
% 	\footnote{check function space and energy inequality for $\tilde\bbb$.}
\begin{align}
	& \cE_{AGG}(\bu, \phi) + \frac{1}{h} \int_{I_{k+1}} \cE_{B}(\widetilde\bbb, \phi)(t) \dt
	+ \int_\Omega \rho_k \frac{\abs{\bu - \bu_k}^2}{2} \dx
	+ \int_\Omega \frac{\abs{\nabla \phi - \nabla \phi_k}^2}{2} \dx 
	\nonumber\\
	& \qquad
        + h \int_\Omega \frac{\nu(\phi_k)}{2} \abs{\nabla \bu + \tran{\nabla} \bu}^2 \dx 
	    + h \int_\Omega m(\phi_k) \abs{\nabla q}^2 \dx
		+ \eta h \int_\Omega \abs{\frac{\phi - \phi_k}{h}}^2 \dx
        \nonumber\\
	&\qquad 
        + \frac{\kappa}{d} \int_{I_{k+1}} \int_\Omega \abs{\nabla \tr \ln \widetilde\bbb}^2 \dxdt
        + \int_{I_{k+1}} \int_\Omega \frac{\mu_\eta(\phi_k)}{2} \tr(\widetilde\bbb
        + \widetilde\bbb^{-1} - 2\bbi) \dx \dt
%			+ \frac{\eta h}{2} \int_\Omega \abs{\frac{\phi - \phi_k}{h}}^2 \dx 
	\nonumber\\
	& \quad \leq \cE_{AGG}(\bu_k, \phi_k)
	+ \frac{1}{h} \int_{I_k} \cE_{B}(\widetilde{\bbb}_k, \phi_k)(t) \dt,
%			+ C (h + 1) \big( 
%    			\cE_{B}(\bbb_k) 
%    			+ \norm{\bbb_k}_{L^2}^2
%    			+ 1
%			\big)^2,
	\label{eqs:discrete-energy-Notime}
	\end{align}
	where 
	\begin{equation*}
		\cE_{AGG}(\bu, \phi) \coloneqq \int_\Omega \frac{\rho(\phi)}{2} \abs{\bu}^2 \dx
		+ \int_\Omega \Big(\frac{1}{2} \abs{\nabla \phi}^2 + W(\phi)\Big) \dx, 
	\end{equation*}
	and
	\begin{equation*}
	    \cE_{B}(\bbb, \phi) \coloneqq \int_\Omega \frac{\mu_\eta(\phi)}{2} \tr(\bbb - \ln \bbb - \bbi) \dx.
	\end{equation*}
\end{proposition}
The main strategy of the proof is as follows. First, we derive a uniform energy estimate which holds true for any solution $(\bu, \phi, q, \widetilde\bbb)$ of \eqref{eqs:regularized-discret}. 
    Then, we justify the existence of at least one solution of \eqref{eqs:regularized-discret} with a fixed point argument. More precisely, we first construct a continuous mapping $(\bv, \phi_k, \overline\bbb_k) \mapsto \mathcal{T}_{k+1}(\bv, \phi_k, \overline\bbb_k)$, see \eqref{eqs:mapping-B_k+1}, where $\mathcal{T}_{k+1}(\bv, \phi_k, \overline\bbb_k)$ denotes the unique solution of the Oldroyd-B equation \eqref{eqs:B-equation} on the time interval $I_{k+1}$, which depends continuously on a given function triplet $(\bv, \phi_k, \overline\bbb_k)$, see Theorem \ref{thm:B-equation} and \eqref{eqs:B-uniqueness-estimate}.
    Subsequently, we rewrite \eqref{eqs:regularized-discret} as a fixed point problem for $(\bu, \phi, q)$ with $\widetilde\bbb$ implicitly given by $ \widetilde\bbb = \mathcal{T}_{k+1}(\bu, \phi_k, \overline\bbb_k)$, which is then solved with the Leray--Schauder principle.
\begin{proof}[{Proof of Proposition \ref{prop:existence-hybrid}}]
    {\textit{\uline{Energy inequality.}}} 
    % For given $ (\bu_k, \phi_k, \bbb_k) $, we have shown existence of a solution tuple $(\widetilde\bbb_{k+1}, \bu_{k+1}, \phi_{k+1}, q_{k+1})$. If we add both of them and sum over all $ k = 0,..., N $, where $N=\frac1h$, then the right-hand side would blow up in the limit $N\to \infty$. The reason for that is that we estimated the mixed terms separately which was necessary to decouple the existence proofs. Now, we show good a priori estimates.
    % For simplicity, we neglect the index $k+1$.
    With a standard testing procedure as in \cite{ADG2013}, one may expect a similar estimate, except for the terms regarding $ \widetilde{\bbb} $. Namely, testing \eqref{eqs:AGG-discret-u} with $ \bu $, \eqref{eqs:AGG-discret-phi} with $ q $ and \eqref{eqs:AGG-discret-q} with $\frac{\phi-\phi_k}{h}$, and multiplying everything by $ h $, one has
    \begin{align}
        &\cE_{AGG}(\bu_, \phi) 
        + \int_\Omega \rho_k \frac{\abs{\bu - \bu_k}^2}{2} \dx
        + \int_\Omega \frac{\abs{\nabla \phi - \nabla \phi_k}^2}{2} \dx 
        \nonumber\\
        & \quad + h \int_\Omega \frac{\nu(\phi_k)}{2} \abs{\nabla \bu + \tran{\nabla} \bu}^2 \dx 
        + h \int_\Omega m(\phi_k) \abs{\nabla q}^2 \dx
        + \eta h \int_\Omega \abs{\frac{\phi - \phi_k}{h}}^2 \dx 
        \nonumber\\
        &\quad +  \int_\Omega \frac{\mu_\eta(\phi) - \mu_\eta(\phi_k)}{2} \Big(\frac{1}{h} \int_{I_{k+1}} \tr(\widetilde{\bbb} - \ln \widetilde{\bbb} - \bbi)(t) \dt\Big) \dx 
        \nonumber\\
        &\leq 
        \cE_{AGG}(\bu_{k}, \phi_{k}) 
        + \int_{I_{k+1}} \int_\Omega \frac{\mu_\eta(\phi_k)}{2} \sR_\eta \bu \cdot \nabla  \tr(\widetilde{\bbb} - \ln \widetilde{\bbb} - \bbi) \dx \dt
        \nonumber\\ 
        &\quad
        - \int_{I_{k+1}} \int_\Omega \mu_\eta(\phi_k) (\widetilde{\bbb} - \bbi) : \nabla \sR_\eta \bu \dx \dt. 
        \label{eqs:AGG-energy-discrete-B}
    \end{align}
    The key point here is to cancel out or control these extra terms associated with $ \widetilde{\bbb} $ uniformly with respect to $ h $, so that the energy will not blow up after summing over $ k = 0,...,N-1 $ and sending $ N \rightarrow \infty$.
    Here, we make sufficient use of the energy structure of $ \widetilde{\bbb} $. Formally multiplying \eqref{eqs:B-equation-discrete} with $ \frac{\mu_\eta(\phi_k)}{2} ( \bbi - \widetilde\bbb^{-1})$ and integrating over $ \Omega \times (t_k, t) $ for a.e.~$ t \in I_{k+1} $ yields
    \begin{align*}
    	& \int_\Omega \frac{\mu_\eta(\phi_k)}{2} \tr(\widetilde\bbb - \ln \widetilde\bbb - \bbi)(t) \dx 
    	+ \int_{t_{k}}^t \int_\Omega \frac{\mu_\eta(\phi_k)}{2} \tr(\widetilde\bbb
    	+ \widetilde\bbb^{-1} - 2\bbi)(s) \dx \d s
    	\nonumber \\
    	&\quad
    	+ \frac{\kappa}{d} \int_{t_k}^t \int_\Omega \abs{\nabla \tr \ln \widetilde\bbb}^2  (s) \dx \d s
    	\nonumber\\
    	&\leq 
    	\int_\Omega \frac{\mu_\eta(\phi_k)}{2} \tr(\overline{\bbb}_k - \ln \overline{\bbb}_k - \bbi) \dx 
    	- \int_{t_k}^t \int_\Omega \frac{\mu_\eta(\phi_k)}{2} \sR_\eta \bu \cdot 
    	\nabla  \tr(\widetilde\bbb - \ln \widetilde\bbb - \bbi) \dx \d s
    	\nonumber \\
    	&\quad
    	+ \int_{t_k}^t \int_\Omega \mu_\eta(\phi_k) (\widetilde\bbb - \bbi) : \nabla \sR_\eta \bu \dx \d s. 
    \end{align*}
    We note that this inequality can be recovered from the limit passage of \eqref{eqs:B-reg-ddt} in Section \ref{sec:uniform-B}.
    Now, integrating over $t\in(t_k, t_{k+1})=I_{k+1}$ and multiplying both sides by $\frac1h$ lead to
    \begin{align}
    	& \frac1h \int_{I_{k+1}}\int_\Omega \frac{\mu_\eta(\phi_k)}{2} \tr(\widetilde\bbb - \ln \widetilde\bbb - \bbi) \dxdt
    	+ \int_{I_{k+1}} \int_\Omega \frac{\mu_\eta(\phi_k)}{2} \tr(\widetilde\bbb
    	+ \widetilde\bbb^{-1} - 2\bbi) \dxdt
    	\nonumber \\
    	&\quad
    	+ \frac{\kappa}{d} \int_{I_{k+1}} \int_\Omega \abs{\nabla \tr \ln \widetilde\bbb}^2  \dxdt
    	\nonumber\\
    	&\leq 
    	\int_\Omega \frac{\mu_\eta(\phi_k)}{2} \tr(\overline{\bbb}_k - \ln \overline{\bbb}_k - \bbi) \dx
     - \int_{I_{k+1}} \int_\Omega \frac{\mu_\eta(\phi_k)}{2} \sR_\eta \bu \cdot 
    	\nabla  \tr(\widetilde\bbb - \ln \widetilde\bbb - \bbi) \dxdt
    	\nonumber \\
    	&\quad
    	+ \int_{I_{k+1}} \int_\Omega \mu_\eta(\phi_k) (\widetilde\bbb - \bbi) : \nabla \sR_\eta \bu \dxdt,
    	\label{eqs:B-identity-discrete}
    \end{align}
    which is with the help of the fundamental theorem of calculus. 
    % \footnote{I add this sentence to make it clear (at least for me). If it is clear from your side, please just delete it.\\ It is fine.}
	% \begin{equation}
	%     \label{eqs:entropy-B-sup}
	% 	\int_\Omega \frac{\mu_\eta(\phi_k)}{2} \frac{1}{h} \int_{I_{k+1}} \tr(\widetilde{\bbb} - \ln \widetilde{\bbb} - \bbi)(t) \dt \dx 
	% 	\leq \sup_{t \in I_{k+1}} \int_\Omega \frac{\mu_\eta(\phi_k)}{2} \tr(\widetilde\bbb - \ln \widetilde\bbb - \bbi)(t) \dx.
	% \end{equation}
	Using Jensen's inequality, we estimate
	\begin{equation}
	    \label{eqs:entropy-B-averaged}
		\int_\Omega \frac{\mu_\eta(\phi_k)}{2} \tr(\overline{\bbb}_k - \ln \overline{\bbb}_k - \bbi) \dx
		\leq \frac{1}{h} \int_{I_k} \int_\Omega \frac{\mu_\eta(\phi_k)}{2} \tr(\widetilde{\bbb}_k - \ln \widetilde{\bbb}_k - \bbi)(t) \dxdt.
	\end{equation}
%    \begin{align*}
%        & \int_\Omega \frac{\mu_\eta(\phi_k)}{2} \ddt \tr(\widetilde\bbb - \ln \widetilde\bbb - \bbi) \dx 
%        + \int_\Omega \frac{\mu_\eta(\phi_k)}{2} \tr(\widetilde\bbb
%        + \widetilde\bbb^{-1} - 2\bbi) \dx
%        + \frac{\kappa}{d} \int_\Omega \abs{\nabla \tr \ln \widetilde\bbb}^2  \dx 
%        \\
%        &\leq - \int_\Omega \frac{\mu_\eta(\phi_k)}{2} (\sR_\eta \bu \cdot \nabla)  \tr(\widetilde\bbb - \ln \widetilde\bbb - \bbi) \dx
%        + \int_\Omega \mu_\eta(\phi_k) (\widetilde\bbb - \bbi) : \nabla \sR_\eta \bu \dx. 
%    \end{align*}
%    Then we do a trick to take the time integral over $ I_{k+1} $, i.e.,
%    \begin{align}
%    & \int_\Omega \frac{\mu_\eta(\phi_k)}{2} \tr(\widetilde\bbb - \ln \widetilde\bbb - \bbi)(t_{k+1}) \dx 
%        + \int_{I_{k+1}} \int_\Omega \frac{\mu_\eta(\phi_k)}{2} \tr(\widetilde\bbb
%        + \widetilde\bbb^{-1} - 2\bbi) \dxdt
%        \nonumber \\
%        &\quad
%        + \frac{\kappa}{d} \int_{I_{k+1}} \int_\Omega \abs{\nabla \tr \ln \widetilde\bbb}^2  \dxdt
%        \nonumber\\
%        &\leq 
%        \int_\Omega \frac{\mu_\eta(\phi_k)}{2} \tr(\bbb_k - \ln \bbb_k - \bbi) \dx 
%        - \int_{I_{k+1}} \int_\Omega \frac{\mu_\eta(\phi_k)}{2} \sR_\eta \bu \cdot 
%     \nabla  \tr(\widetilde\bbb - \ln \widetilde\bbb - \bbi) \dxdt
%        \nonumber\\
%        &\quad
%        + \int_{I_{k+1}} \int_\Omega \mu_\eta(\phi_k) (\widetilde\bbb - \bbi) : \nabla \sR_\eta \bu \dxdt. 
%      	\label{eqs:B-identity-discrete}
%    \end{align}
% \footnote{Explain the reason for these two estimates.\\Is it okay now? Yes, thank you!}

    Adding \eqref{eqs:AGG-energy-discrete-B} and \eqref{eqs:B-identity-discrete} and noting \eqref{eqs:entropy-B-averaged} gives rise to the desired uniform estimate, where the mixed terms on the right-hand sides of \eqref{eqs:AGG-energy-discrete-B}--\eqref{eqs:B-identity-discrete} cancel out, i.e.,
    \begin{align*}
        &\cE_{AGG}(\bu, \phi) 
        + \frac{1}{h} \int_{I_{k+1}} \int_\Omega \frac{\mu_\eta(\phi)}{2} \tr(\widetilde{\bbb} - \ln \widetilde{\bbb} - \bbi)(t) \dxdt
        \\
        &\quad 
        + \int_{I_{k+1}} \int_\Omega \frac{\mu_\eta(\phi_k)}{2} \tr(\widetilde\bbb
        + \widetilde\bbb^{-1} - 2\bbi) \dx\dt
        \\
        &\quad
        + \int_\Omega \rho_k \frac{\abs{\bu - \bu_k}^2}{2} \dx
        + \int_\Omega \frac{\abs{\nabla \phi - \nabla \phi_k}^2}{2} \dx 
        + \frac{\kappa}{d} \int_{I_{k+1}} \int_\Omega \abs{\nabla \tr \ln \widetilde\bbb}^2  \dxdt
        \nonumber\\
        & \quad + h \int_\Omega \frac{\nu(\phi_k)}{2} \abs{\nabla \bu + \tran{\nabla} \bu}^2 \dx 
        + h \int_\Omega m(\phi_k) \abs{\nabla q}^2 \dx
        + \eta h \int_\Omega \abs{\frac{\phi - \phi_k}{h}}^2 \dx
        \\
        &\leq 
        \cE_{AGG}(\bu_{k}, \phi_{k}) 
        + \frac{1}{h} \int_{I_k} \int_\Omega \frac{\mu_\eta(\phi_k)}{2} \tr(\widetilde{\bbb}_k - \ln \widetilde{\bbb}_k - \bbi)(t) \dxdt.
    \end{align*}

    {\textit{\uline{Construction of a fixed-point mapping.}}}
    In order to prove the existence of solutions, we proceed with the Leray--Schauder principle (see, e.g., \cite[Lemma 3.1.1, Chapter II]{Sohr2001}), which was carried out in \cite{ADG2013} as well. Now we define
    % \footnote{$Z$ convex space? \\
    % @D: Yes, one can easily check that by taking two tensors from $Z$ and verifing their linear conbination is still in $Z$, with the convexity of $ \tr(\bbb - \ln \bbb) $ ($L^2$ norm is obviously convex).\\
    % @Y: This space would also be `convex' with only $\tr\ln\bbb \in L^1(\Omega\times I_{k+1})$. More precisely, $\tr\ln\bbb \in L^1(\Omega\times I_{k+1})$ implies $-\tr\ln\bbb \in L^1(\Omega\times I_{k+1})$, and $\bbb\mapsto -\tr\ln(\bbb)$ is strictly convex.}
    %\footnote{the third component of the product space $X$ should be $W^{2,2}_n(\Omega)$ instead of $W^{1,2}_n(\Omega)$, right? I corrected it.\\Yes! Thank you! that is a typo.}
    \begin{gather*}
        X = W_{0,\sigma}^{1,2}(\Omega) \times \cD(\partial \widetilde{E}) \times W_n^{2,2}(\Omega), \quad
        Y = [W_{0,\sigma}^{1,2}(\Omega)]' \times L^2(\Omega) \times L^2(\Omega), \\
        Z_{k+1} = \left\{ 
            \begin{aligned}
                \bbb \in L^2(I_{k+1}; W^{1,2}(\Omega; \bbr_{\mathrm{sym}}^{d \times d})):\  
                & \bbb \text{ is positive definite a.e.~in } \Omega, \\
                & \tr \ln \bbb \in L^1(\Omega \times I_{k+1})
            \end{aligned}
        \right\}.
    \end{gather*}
    By means of Theorem \ref{thm:B-equation}, we construct a mapping $ \cT_{k+1}: X \rightarrow Z_{k+1} $ such that 
    \begin{equation}
        \label{eqs:mapping-B_k+1}
        \widetilde{\bbb} = \cT_{k+1} (\bu, \phi_k, \overline{\bbb}_k) \text{ solves \eqref{eqs:B-equation-discrete} with } \overline{\bbb}_{k} = \frac{1}{h} \int_{I_{k}} \widetilde{\bbb}_k(t) \dt,
    \end{equation}
    which are both positive definite a.e.~in $ \Omega $, and such that the inequality \eqref{eqs:B-identity-discrete} is fulfilled.

    Next, we are going to deal with \eqref{eqs:AGG-discret-u}--\eqref{eqs:AGG-discret-q} with the help of \eqref{eqs:mapping-B_k+1}. 
    For $ \bw = (\bu, \phi, q) \in X $, we define $ \cL_k: X \rightarrow Y $ as
    \begin{equation*}
        \cL_k(\bw) = \left(
            \begin{gathered}
                L_k(\bu) \\
                - \Div(m(\phi_k) \nabla q) + \int_\Omega q \dx \\
                \partial \widetilde{E}(\phi) + \phi 
            \end{gathered}
        \right),
    \end{equation*}
    where
    \begin{equation*}
        \inner{L_k(\bu)}{\bv}_{[W_{0,\sigma}^{1,2}]' \times W_{0,\sigma}^{1,2}} = \int_\Omega \nu(\phi_k) (\nabla \bu + \tran{\nabla} \bu): \nabla \bv \dx \quad \text{for } \bv \in W_{0,\sigma}^{1,2}(\Omega),
    \end{equation*}
    and the second and third lines are understood in the pointwise sense.
    Moreover, for $ \bw = (\bu, \phi, q) \in X $ we introduce $ \cF_k: X \rightarrow Y $ as
    \begin{equation*}
        \cF_k(\bw) = \left(
            \begin{gathered}
                F_k(\bw) \\
                - \frac{\phi - \phi_k}{h} - \bu \cdot \nabla \phi_k + \int_\Omega q \dx \\
                - \eta \frac{\phi - \phi_k}{h}
                + \phi + q + \omega \frac{\phi + \phi_k}{2} - G(\bw, \widetilde{\bbb})
            \end{gathered}
        \right),
    \end{equation*}
    where
    \begin{align*}
    	G(\bw, \widetilde{\bbb})
    	= &\, \sR_\eta \Big[\onehalf \frac{\mu_\eta(\phi) - \mu_\eta(\phi_k)}{\sR_\eta (\phi - \phi_k)} \frac{1}{h} \int_{I_{k+1}} \tr(\widetilde{\bbb} - \ln \widetilde{\bbb} - \bbi)(t) \dt \Big], \\
        F_k(\bw) = 
        & - \frac{\rho \bu - \rho_k \bu_k}{h}
        - \Div(\rho_k \bu \otimes \bu) 
        + q \nabla \phi_k \\
        & - \Big( \Div \bJ - \frac{\rho - \rho_k}{h} - \bu \cdot \nabla \rho_k \Big) \frac{\bu}{2}
        - (\bJ \cdot \nabla) \bu \\
        & - \Div \big(\sR_\eta \big[\mu_\eta(\phi_k)(\overline{\bbb}_{k+1} - \bbi)\big]\big) + \frac{1}{h} \int_{I_{k+1}} \sR_\eta \Big[\frac{\mu_\eta(\phi_k)}{2} \nabla \tr(\widetilde{\bbb} - \ln \widetilde{\bbb} - \bbi)\Big] \dt,
    \end{align*}
    and $\overline{\bbb}_{k+1}$ and $\widetilde{\bbb} $ are constructed in \eqref{eqs:mapping-B_k+1}. Consequently, we see that $ (\bu, \phi, q, \widetilde\bbb) $ is a weak solution of the system \eqref{eqs:regularized-discret} if and only if $ \bw = (\bu, \phi, q) \in X $ satisfies
    \begin{equation}
    	\label{eqs:LF}
        \cL_k(\bw) - \cF_k(\bw) = 0.
    \end{equation}
    {Note that by \eqref{eqs:LF}, if one can show the invertibility of $\cL_k$ and continuity of $\cF_k$, we obtain a fixed-point formulation $\bw = \cL_k^{-1} \circ \cF_k (\bw)$.}

    Similarly to \cite{ADG2013}, one can check that $ \cL_k: X \rightarrow Y $ is invertible with the inverse $ \inv{\cL}_k: Y \rightarrow X $. Note that $ X $ is not a Banach space since $ \cD(\partial \widetilde{E}) $ consists of inequality constraints. To get a continuous and even compact operator, we introduce for $ 0 < s < \frac{1}{4} $ the following Banach spaces
    \begin{equation*}
    	\widetilde{X} \coloneqq W_{0,\sigma}^{1,2}(\Omega) \times W^{2-s,2}(\Omega) \times W_n^{2,2}(\Omega), \quad 
    	\widetilde{Y} \coloneqq L^{\frac{3}{2}}(\Omega; \bbr^2) \times W^{1,\frac{3}{2}}(\Omega) \times W^{1,2}(\Omega).
    \end{equation*}
Then we obtain the continuity of $ \inv{\cL}_k : Y \rightarrow \widetilde{X} $ from standard theory and with the above note concerning {the continuity of the term $\partial \widetilde{E}(\phi) + \phi $ within the definition of $ \cL_k $}. In view of the Rellich--Kondrachov theorem in three dimensions, one knows that $ \widetilde{Y} \subset \subset Y $ compactly, which implies that the restriction $ \inv{\cL}_k : \widetilde{Y} \rightarrow \widetilde{X} $ is a compact operator. Moreover, we infer that $ \cT_{k+1} : \widetilde{X} \rightarrow Z $ is well-defined due to Theorem \ref{thm:B-equation} and the restriction $ \cF_k : \widetilde{X} \rightarrow \widetilde{Y} $ is continuous and maps bounded sets into bounded sets, which can be verified by the same argument as in \cite{ADG2013}. Here we leave out the details except for the terms regarding $\widetilde\bbb $. 
Thanks to the regularization operator $ \sR_\eta $, we have
	\begin{align*}
		& \norm{\Div \big(\sR_\eta \big[\mu_\eta(\phi_k)(\overline{\bbb}_{k+1} - \bbi)\big]\big)}_{L^\frac{3}{2}} \\
			& \qquad \qquad \leq C \frac{1}{h} \int_{I_{k+1}} (\normm{\widetilde{\bbb}}_{L^2}^2 
		+ \normm{\nabla \widetilde{\bbb}}_{L^2}^2) \dt + \norm{\phi_k}_{L^2}^2 + \norm{\nabla \phi_k}_{L^2}^2 + C
		\leq C_k(1 + \frac{1}{h}), \\
		& \frac{1}{h} \int_{I_{k+1}} \norm{\sR_\eta \Big[\frac{\mu_\eta(\phi_k)}{2} \nabla \tr(\widetilde{\bbb} - \ln \widetilde{\bbb} - \bbi)\Big]}_{L^\frac{3}{2}} \dt \\
			& \qquad \qquad \leq C \frac{1}{h} \int_{I_{k+1}} \norm{\nabla \tr(\widetilde{\bbb} - \ln \widetilde{\bbb})}_{L^2}^2 \dt + C
		\leq C_k(1 + \frac{1}{h}), \\
		& \norm{G(\bw, \widetilde{\bbb})}_{W^{1,2}} 
			\leq C  \frac{1}{h} \int_{I_{k+1}} \norm{\tr(\widetilde{\bbb} - \ln \widetilde{\bbb} - \bbi)}_{W^{1,2}} \dt
		\leq C_k(1 + \frac{1}{h}),
	\end{align*}
    where $ C_k = C_k(\normm{\bu}_{W^{1,2}}) > 0 $. Hence, $\cF_k$ maps bounded sets into bounded sets.

Next we record the continuity of $\cF_k$, that is, $\cF_k(\bw^\ell) \to \cF_k(\bw)$ in $\widetilde{Y}$ if $\bw^\ell \to \bw$ in $\widetilde{X}$ as $\ell \to \infty$. Note that most of the terms of $\cF_k$ are similar to \cite{ADG2013}, and the main differences are the terms corresponding to $ \widetilde{\bbb} $. In order to prove the continuity, we take an arbitrary sequence $ \{\bw^\ell\}_{\ell \in \bbn} \subset \widetilde{X} $ such that $\bw^\ell \to \bw$ in $\widetilde X$, as $ \ell \to \infty $, which implies $\bu^\ell\to\bu$ in $W_{0,\sigma}^{1,2}$, and then we investigate the continuity of $\cT_{k+1}$ in terms of $\bu^\ell\in W_{0,\sigma}^{1,2}$ for fixed $\phi_k$ and $\bbb_k$. It follows from Theorem \ref{thm:B-equation} and \eqref{eqs:B-uniqueness-estimate} that $\bbb^\ell \to \bbb$ in $L^2(\Omega \times I_{k+1})$ and $\nabla \bbb^\ell \to \nabla \bbb$ in $L^2(\Omega \times I_{k+1})$. By applying a similar argument as in Section \ref{sec:B-positive-delta}, one is able to show the positive definiteness of $\bbb$ and $\tr \ln \bbb^\ell \to \tr \ln \bbb$ a.e.~ in $\Omega \times I_{k+1}$ (up to a non-relabeled subsequence). Then, with the uniform boundedness of $\tr\ln\bbb^\ell$ in $L^2(\Omega \times I_{k+1})$, one concludes that $\tr\ln\bbb^\ell \to \tr\ln\bbb$ strongly in $L^1(\Omega \times I_{k+1})$ in view of Vitali's convergence theorem. Hence, $\tr(\bbb^\ell-\ln\bbb^\ell-\bbi) \to \tr(\bbb-\ln\bbb-\bbi)$ in $L^1(\Omega \times I_{k+1})$, which finally implies the continuity of $ \cF_k $.

    {\textit{\uline{Application of a fixed-point theorem.}}}
    The final step is to employ the Leray--Schauder principle on $ \widetilde{Y} $, for which we denote $ \cK_k \coloneqq \cF_k \circ \inv{\cL}_k : \widetilde{Y} \rightarrow \widetilde{Y} $, rewrite \eqref{eqs:LF} as
	\begin{equation*}
		\mathbf{f} - \cK_k(\mathbf{f}) = 0 \quad 
		\text{for } \mathbf{f} = \cL_k(\bw),
	\end{equation*}
	and find a fixed point of $ \cK_k $. Note that $ \cK_k $ is a compact operator because $ \inv{\cL}_k $ is compact and $ \cF_k $ is continuous. {Now we are in a position to apply the Leray--Schauder principle (see, e.g., \cite[Lemma 3.1.1, Chapter II]{Sohr2001}) to find a fixed point. Namely, we will show that} %\todo{change formulation}
	\begin{equation}
		\label{eqs:Leray-Schauder-f}
		\exists R > 0 \text{ such that } \norm{\mathbf{f}}_{\widetilde{Y}} \leq R, \text{ for } \mathbf{f} \in \widetilde{Y} \text{ fulfilling } \mathbf{f} = \lambda \cK_k(\mathbf{f}) \text{ for some } 0 \leq \lambda \leq 1.
	\end{equation}
 Here, $ \lambda $ is a constant used for the proof and is unrelated to the relaxation time of the model \eqref{eqs:Model}. By the definition $ \bw = \inv{\cL}_k(\mathbf{f}) $, one may see that $ \mathbf{f} = \lambda \cK_k(\mathbf{f}) $ in \eqref{eqs:Leray-Schauder-f} is equivalent to 
	\begin{equation}
		\label{eqs:Leray-Schauder-w}
		\cL_k(\bw) - \lambda \cF_k(\bw) = 0 \quad \text{for some } 0 \leq \lambda \leq 1.
	\end{equation}
	The rest of this subsection is devoted to the estimate of $ \bw $ in $ \widetilde{X} $ satisfying \eqref{eqs:Leray-Schauder-w}, which thereby gives the estimate of $ \mathbf{f} $ in $ \widetilde{Y} $ by the fact that $ \cF_k : \widetilde{X} \rightarrow \widetilde{Y} $ is bounded. The subsequent argument is rather lengthy but is very similar to \cite{ADG2013}. Thus, we point out the main differences associated with $ \widetilde\bbb $ in the following. For the reader's convenience, we give the equivalent weak formulation corresponding to \eqref{eqs:Leray-Schauder-w} as follows.
	\begin{subequations}
		\label{eqs:regularized-discret-lambda}
		\begin{equation}
			\label{eqs:AGG-discret-u-lambda}
			\begin{aligned}
				& \lambda \inner{\frac{\rho \bu - \rho_k \bu_k}{h}}{\bv}
				+ \lambda \inner{\Div (\rho_k \bu \otimes \bu)}{\bv} 
				+ \lambda \inner{\Div (\bu \otimes \bJ)}{\bv} \\
				& \qquad + \inner{\nu(\phi_k)(\nabla \bu + \tran{\nabla} \bu)}{\nabla \bv} 
				- \lambda \inner{\Div \sR_\eta \big[\mu_\eta(\phi_k)(\overline{\bbb}_{k+1} - \bbi)\big]}{\bv} \\
				& \qquad = \lambda \inner{q \nabla \phi_k}{\bv}
				+ \lambda \frac{1}{h} \int_{I_{k+1}} \inner{\sR_\eta \Big[\frac{\mu_\eta(\phi_k)}{2} \nabla \tr(\widetilde{\bbb} - \ln \widetilde{\bbb} - \bbi)\Big]}{\bv} \dt,
			\end{aligned}
		\end{equation}
		for all $ \bv \in W_{0,\sigma}^{1,2}(\Omega) $,
  %\footnote{changed the test function from $\bw$ to $\bv$, as $\bw$ is also used for Leray--Schauder. \\Thank you! I overlooked this point.} 
  and for a.e.~$ x \in \Omega $,
		\begin{gather}
			\label{eqs:AGG-discret-phi-lambda}
			\lambda \frac{\phi - \phi_k}{h} + \lambda \bu \cdot \nabla \phi_k = \Div(m(\phi_k) \nabla q), \\
			\label{eqs:AGG-discret-q-lambda}
			\begin{aligned}
				\lambda q + \lambda \omega \frac{\phi + \phi_k}{2} + \lambda \phi & = \partial \widetilde{E}(\phi)
				+ \phi
				+ \lambda \eta \frac{\phi - \phi_k}{h}
				+ \lambda G(\bw, \widetilde{\bbb}).
			\end{aligned}
		\end{gather}
	\end{subequations}
	By the analogous test procedure as for the energy estimate (also in \cite{ADG2013}), that is, testing \eqref{eqs:AGG-discret-u-lambda} with $ \bu $, \eqref{eqs:AGG-discret-phi-lambda} with $ q $ and \eqref{eqs:AGG-discret-q-lambda} with $ \frac{\phi - \phi_k}{h} $, we derive similar estimates. The main difference here compared to \cite{ADG2013} is that we have extra terms associated with $ \widetilde\bbb $, which in light of \eqref{eqs:B-identity-discrete} can be canceled, as
	\begin{align*}
		& \quad \lambda \int_{I_{k+1}} \int_\Omega \mu_\eta(\phi_k) (\widetilde{\bbb} - \bbi) : \nabla \sR_\eta \bu \dxdt
			- \lambda \int_{I_{k+1}} \int_\Omega \frac{\mu_\eta(\phi_k)}{2} \sR_\eta \bu \cdot \nabla \tr(\widetilde{\bbb} - \ln \widetilde{\bbb} - \bbi) \dxdt \\
		    & \quad + \lambda \int_{I_{k+1}} \int_\Omega \frac{\sR_\eta (\phi - \phi_k)}{2 h} \frac{\mu_\eta(\phi) - \mu_\eta(\phi_k)}{\sR_\eta (\phi - \phi_k)} \tr(\widetilde{\bbb} - \ln \widetilde{\bbb} - \bbi)(t) \dxdt \\
		% & \geq \lambda \int_{I_{k+1}} \int_\Omega \mu_\eta(\phi_k) (\widetilde{\bbb} - \bbi) : \nabla \sR_\eta \bu \dxdt
		% - \lambda \int_{I_{k+1}} \int_\Omega \mu_\eta(\phi_k) \sR_\eta \bu \cdot \nabla \tr(\widetilde{\bbb} - \ln \widetilde{\bbb} - \bbi) \dxdt \\
		% & \quad + \lambda \frac{1}{h} \int_{I_{k+1}} \int_\Omega \frac{\mu_\eta(\phi)}{2} \tr(\widetilde{\bbb} - \ln \widetilde{\bbb} - \bbi)(t) \dxdt
		% 	-  \sup_{t \in I_{k+1}} \lambda \int_\Omega \frac{\mu_\eta(\phi_k)}{2} \tr(\widetilde{\bbb} - \ln \widetilde{\bbb} - \bbi)(t) \dx \\
		& \geq \lambda \frac{1}{h} \int_{I_{k+1}} \int_\Omega \frac{\mu_\eta(\phi)}{2} \tr(\widetilde{\bbb} - \ln \widetilde{\bbb} - \bbi)(t) \dxdt
			- \lambda \int_\Omega \frac{\mu_\eta(\phi_k)}{2} \tr(\overline{\bbb}_k - \ln \overline{\bbb}_k - \bbi) \dx \\
			& \quad + \lambda \int_{I_{k+1}} \int_\Omega \frac{\mu_\eta(\phi_k)}{2} \tr(\widetilde{\bbb} + \widetilde{\bbb}^{-1} - 2\bbi) \dxdt
			+ \lambda \frac{\kappa}{2} \int_{I_{k+1}} \int_\Omega \abs{\nabla \tr \ln \widetilde\bbb}^2  \dxdt.
	\end{align*}
	Then we reach the estimate
	\begin{equation*}
		\norm{\bw}_{\widetilde{X}}
		+ \norm{\partial \widetilde{E}(\phi)}_{L^2}
		\leq C_k,
	\end{equation*}
	for fixed $ h $, and hence
	\begin{equation*}
		\norm{\mathbf{f}}_{\widetilde{Y}}
		= \norm{\lambda \cF_k(\bw)}_{\widetilde{Y}}
		\leq C_k(\norm{\bw}_{\widetilde{X}} + 1)
		\leq C_k,
	\end{equation*}
	which finishes the proof of Proposition \ref{prop:existence-hybrid} by the Leray--Schauder principle.
\end{proof}

\subsection{Construction of Approximate Solutions}
\label{sec:construction-approx}
Let $T>0$ and $ N \in \bbn $ be given and, for $k\in\{1,...,N\}$, let $ (\bu_k, \phi_k, q_k, \bbb_k) $ be chosen successively as a solution to \eqref{eqs:regularized-discret} with $ h = \frac{T}{N} $ and $ (\bu_0, \phi_0^N, \bbb_0) $ as the initial data. 
% Here $ \widetilde{\bbb}_k(t) $ is the unique solution of \eqref{eqs:B-equation} on $ (I_{k+1}) $ with $ (\bv, \xi) $ substituted by $ (\bu_{k - 1}, \phi_{k - 1}) $ and initial data $ \widetilde{\bbb}_k(t)|_{t = 0} = \bbb_{k - 1} $. 
Here, the regularized initial value $ \phi_0^N \in W^{2,2}(\Omega) $ is constructed as in \cite{ADG2013} and satisfies $ \phi_0^N \rightarrow \phi_0 $ in $ W^{1,2}(\Omega) $, as $ N \rightarrow \infty $. 

Now we define $ f^N(t) $ on $ [-h, T) $ through
\begin{equation*}
	f^N(t) = f_k \quad \text{for }  t \in [t_{k-1}, t_k),
\end{equation*}
where $ k \in \{0,...,N\}$ and $ f \in \{\bu, \phi, q, \bbb\} $. In particular, it holds that
\begin{equation*}
	f^N((k - 1)h) = f_k, \quad 
	f^N(kh) = f_{k + 1}, \quad
	f^N(t) = f_{k + 1} \text{ with } t \in [t_k, t_{k+1}), k\in\{0,...,N-1\}.
\end{equation*}
Moreover, for $ t \in [t_{k-1}, t_k) $, $k\in\{1,...,N\}$, we define $ \widetilde{\bbb}^N(t) \coloneqq \widetilde{\bbb}_k(t) $, $ \overline{\bbb}^N(t) \coloneqq \frac{1}{h} \int_{I_k} \widetilde{\bbb}_k(t) \dt $ and
\begin{gather*}
	\rho^N \coloneqq \rho(\phi^N), \quad 
	f_h \coloneqq f(t - h), \\
	(\Delta_h^+ f)(t) \coloneqq f(t + h) - f(t), \quad 
	\ptial{t,h}^+ f(t) \coloneqq \frac{1}{h} (\Delta_h^+ f)(t), \\
	(\Delta_h^- f)(t) \coloneqq f(t) - f(t - h), \quad 
	\ptial{t,h}^- f(t) \coloneqq \frac{1}{h} (\Delta_h^- f)(t).
\end{gather*}
By definition, it follows that
% \footnote{specify $k$\\@D: is it like this?\\ Yesss.}
\begin{equation}
    \label{eqs:identity-integral-summation}
    \int_0^\tau f^N(t) \dt = h \sum_{k=0}^{\tau/h} f_{k+1}, \quad
    \int_0^\tau \widetilde{\bbb}^N(t) \dt
    = \sum_{k=0}^{\tau/h} \int_{I_{k+1}} \widetilde{\bbb}_{k+1}(t) \dt,
\end{equation}
{for $\tau=nh$, $n\in \{0,...,N-1\}$.} %\todo{Changed. Also check this everywhere else}
%for $ \tau \in h\cdot\{0,...,N-1\} = \{0, h, 2h, ..., (N-1)h\}$. %\footnote{Should it be $\{1,...,N\}$?, also for the next. Okay, we can do it like that.\\ Then we have a contribution of $f_{N+1}$..\\But here is the range of $\tau$..Not a general range for $f^N$.//Wait..You are right. Yes... You are right. Sorry about it.}

\begin{subequations}
    \label{eqs:model-reg-discrete-formulation}
Then for arbitrary $ \bw \in C^\infty([0,T]; C_0^\infty(\Omega; \bbr^d)) $ with $ \Div \bw = 0 $ we shall take $ \widetilde{\bw} \coloneqq \int_{kh}^{(k + 1)h} \bw \dt $ satisfying $ \Div \widetilde{\bw} = 0 $ as the test function in \eqref{eqs:AGG-discret-u} and sum over $ k \in \{0,...,N-1\} $ to get 
% \footnote{I think this notation/formulation is not completely correct. Maybe we must define the piecewise constant (in time) functions $\overline{\bbb}^N (t) \coloneqq \frac{1}{h} \int_{kh}^{(k+1)h} \widetilde{\bbb}_{k+1}(\tau) \dtau$, and some $G^N(t) \coloneqq \frac{1}{h} \int_{kh}^{(k+1)h} \nabla \tr(\widetilde{\bbb}_{k+1} - \ln \widetilde{\bbb}_{k+1} - \bbi)(\tau) \dtau $ for all $t\in( kh, (k+1)h)$ and $k\in\bbn$. \\
% @D: Thank you! I checked this problem. Maybe we can discuss after lunch.}
\begin{align}
	& - \int_0^\tau \int_\Omega \Big( \rho^N \bu^N \cdot \ptial{t,h}^+ \bw 
	- (\rho_h^N \bu^N \otimes \bu^N) : \nabla \bw 
	\Big) \dxdt 
	\nonumber \\
	& \qquad + \int_0^\tau \int_\Omega \Big( 
	\nu(\phi_h^N) (\nabla \bu^N + \tran{\nabla} \bu^N) : \nabla \bw 
	- (\bu^N \otimes \bJ^N) : \nabla \bw 
	\Big) \dxdt 
	\nonumber \\
	& \qquad + \int_0^\tau \int_\Omega \sR_\eta \big[\mu_\eta(\phi_h^N)(\overline{\bbb}^N - \bbi)\big] : \nabla \bw \dxdt
	\label{eqs:AGG-discrete-u-formulation} \\
	& = \int_0^\tau \int_\Omega q^N \nabla \phi_h^N \cdot \bw \dxdt
% 	\nonumber \\
% 	& \qquad 
	+ \int_0^\tau \int_\Omega \sR_\eta \Big[\frac{\mu_\eta(\phi_h^N)}{2} \nabla G^N\Big] \cdot \bw \dxdt \nonumber 
	\\
	& \qquad - \int_\Omega \rho(\phi^N(\cdot, \tau-h)) \bu^N(\cdot, \tau-h) \cdot \overline{\bw}_\tau \dx
	+ \int_\Omega \rho(\phi_0^N) \bu_0 \cdot \overline{\bw}_0 \dx, \nonumber
\end{align}
for all $ \tau \in h\cdot\{0,...,N-1\}$, where
\begin{equation*}
    \overline{\bw}_s 
    \coloneqq \frac{1}{h} \int_s^{s + h} \bw(t) \dt \quad \text{ for } s \in [0,\tau],
\end{equation*}
and
\begin{equation*}
    G^N(t) \coloneqq \frac{1}{h} \int_{I_{k+1}} \tr(\widetilde{\bbb}_{k+1} - \ln \widetilde{\bbb}_{k+1} - \bbi)(s) \d s \quad \text{ for }  t \in [t_k, t_{k+1}).
\end{equation*}
% \footnote{We do not need the case $\tau=0$.\\
% Yes, thank you!}
% \begin{align}
% 	& - \int_0^T \int_\Omega \Big( \rho^N \bu^N \cdot \ptial{t,h}^+ \bw 
% 	- (\rho_h^N \bu^N \otimes \bu^N) : \nabla \bw 
% 	\Big) \dxdt 
% 	\nonumber \\
% 	& \qquad + \int_0^T \int_\Omega \Big( 
% 	\nu(\phi_h^N) (\nabla \bu^N + \tran{\nabla} \bu^N) : \nabla \bw 
% 	- (\bu^N \otimes \bJ^N) : \nabla \bw 
% 	\Big) \dxdt 
% 	\nonumber \\
% 	& \qquad + \int_0^T \int_\Omega \sR_\eta \big[\mu_\eta(\phi_h^N)(\frac{1}{h} \int_{I_{k+1}} \widetilde{\bbb}^N(\tau) \dtau - \bbi)\big] : \nabla \bw \dxdt
% 	\label{eqs:AGG-discrete-u-formulation} \\
% 	& = \int_0^T \int_\Omega q^N \nabla \phi_h^N \cdot \bw \dxdt
% 	\nonumber \\
% 	& \qquad + \int_0^T \int_\Omega \sR_\eta \Big[\frac{\mu_\eta(\phi_h^N)}{2} \frac{1}{h} \int_{I_{k+1}} \nabla \tr(\widetilde{\bbb}^N - \ln \widetilde{\bbb}^N - \bbi)(\tau) \dtau\Big] \cdot \bw \dxdt \nonumber 
% %	\\
% %	& \quad + \int_\Omega \rho(\phi(\cdot, t)) \bu(\cdot, t) \cdot \bw(\cdot, t) \dx
% %	- \int_\Omega \rho(\phi_0) \bu_0 \cdot \bw(\cdot, 0) \dx; \nonumber
% \end{align}
% for all $ \bw \in C_0^\infty(\overline{Q_T}; \bbr^d) $ with $ \Div \bw = 0 $. 
Here the identity
% \footnote{maybe, there is a small mistake in this identity}
\begin{align*}
	\int_0^\tau \int_\Omega & \ptial{t,h}^- (\rho^N \bu^N) \cdot \bw \dxdt
	 + \int_0^\tau \int_\Omega \rho^N \bu^N \cdot \ptial{t,h}^+ \bw \dxdt \\
	    &\quad
        = \int_\Omega \rho(\phi^N(\cdot, \tau-h)) \bu^N(\cdot, \tau-h) \cdot \overline{\bw}_\tau \dx
         - \int_\Omega \rho(\phi_0^N) \bu_0 \cdot \overline{\bw}_0 \dx
\end{align*}
is employed for all $\tau \in h\cdot\{0,...,N-1\}$ and $ \bw \in C^\infty([0,T]; C_0^\infty(\Omega; \bbr^d)) $. Analogously, it follows from \eqref{eqs:AGG-discret-phi} and \eqref{eqs:AGG-discret-q} that
% \begin{equation}
% 	\label{eqs:AGG-discrete-phi-formulation} 
% 	- \int_0^T \int_\Omega \Big( \phi^N \ptial{t,h}^+ \xi 
% 	+ \phi_h^N \bu^N \cdot \nabla \xi \Big) \dxdt
% 	= - \int_0^T \int_\Omega m(\phi_h^N) \nabla q^N \cdot \nabla \xi \dxdt
% \end{equation}
\begin{align}
	\label{eqs:AGG-discrete-phi-formulation} 
	- \int_0^\tau \int_\Omega & \Big( \phi^N \ptial{t,h}^+ \xi 
	    + \phi_h^N \bu^N \cdot \nabla \xi \Big) \dxdt \\
	    & = - \int_0^\tau \int_\Omega m(\phi_h^N) \nabla q^N \cdot \nabla \xi \dxdt
	    - \int_\Omega \phi^N(\cdot, \tau-h) \overline{\xi}_\tau \dx
	    + \int_\Omega \phi_0^N \overline{\xi}_0 \dx
	    \nonumber
\end{align}
for all $ \tau \in h\cdot\{0,...,N-1\} $ and $ \xi \in C^\infty([0,T]; C^1(\overline{\Omega})) $, where
\begin{equation*}
    \overline{\xi}_s 
    \coloneqq \frac{1}{h} \int_s^{s + h} \xi(t) \dt \quad \text{ for } s \in [0,\tau],
\end{equation*}
In addition,
\begin{equation}
    \begin{aligned}
        \partial \widetilde{E}(\phi^N) = \widetilde{W}_0'(\phi^N) - \Delta \phi^N 
        & = q^N + \frac{\omega}{2}(\phi^N + \phi_h^N)
            - \eta \partial_{t,h}^{-} \phi^N
            \\
			& \quad
			- \onehalf \sR_\eta \Big[\frac{\mu_\eta(\phi^N) - \mu_\eta(\phi_h^N)}{\sR_\eta (\phi^N - \phi_h^N)} G^N \Big]
    \end{aligned}
	\label{eqs:AGG-discrete-q-formulation}
\end{equation}
for a.e.~$ x\in\Omega $, and all $ \tau \in h\cdot\{0,...,N-1\} $.
Moreover, testing \eqref{eqs:B-equation-discrete} with any $ \bbc \in C^\infty(\overline{Q_T}; \bbr_{\mathrm{sym}}^{d \times d}) $ and integrating over $\Omega \times I_{k+1}$ with $k\in\{0,...,N-1\}$, applying integration by parts (for the time derivative) and summing over all $k\in\{0,...,m\}$ with $m\in\{0,...,N-1\}$ and $\tau=hm$, one knows
\begin{align}
	\int_0^\tau & \int_\Omega \Big( \widetilde{\bbb}^N : \pt \bbc 
		{- (\sR_\eta \bu^N \cdot\nabla) \widetilde{\bbb}^N : \bbc} \Big) \dxdt \nonumber\\
	& \quad + \int_0^\tau \int_\Omega 2 \bbc \widetilde{\bbb}^N : \nabla \sR_\eta \bu^N
		- \kappa \nabla \widetilde{\bbb}^N : \nabla \frac{\bbc}{\mu_\eta(\phi_h^N)} \Big) \dxdt 
	\label{eqs:B-discrete-formulation} \\
	& = \int_0^\tau \int_\Omega \big( \widetilde{\bbb}^N : \bbc - \tr \bbc \big) \dxdt 
	    + \int_\Omega \widetilde{\bbb}^N(\cdot, \tau) : \bbc(\cdot, \tau) \dx
	    - \int_\Omega \bbb_0 : \bbc(\cdot, 0).
	\nonumber
\end{align}
% \begin{align}
% 	\int_0^T & \int_\Omega \Big( \bbb^N : \partial_{t,h}^+ \bbc 
% 		+ (\sR_\eta \bu^N \otimes \frac{1}{h} \int_{I_{k+1}} \widetilde{\bbb}^N(\tau) \dtau) : \nabla \bbc \Big) \dxdt \nonumber\\
% 	& \quad + \int_0^T \int_\Omega 2 \bbc \frac{1}{h} \int_{I_{k+1}} \widetilde{\bbb}^N(\tau) \dtau : \nabla \sR_\eta \bu^N
% 		- \kappa \frac{1}{h} \int_{I_{k+1}} \nabla \widetilde{\bbb}^N(\tau) \dtau : \nabla \frac{\bbc}{\mu_\eta(\phi)} \Big) \dxdt 
% 	\nonumber \\
% 	& = \int_0^T \int_\Omega \frac{1}{h} \int_{I_{k+1}} \widetilde{\bbb}^N(\tau) \dtau : \bbc \dxdt - \int_0^T \int_\Omega \tr \bbc \dxdt.
% % 		+ \int_\Omega \bbb(\cdot, T) : \bbc(\cdot, T) \dx
% % 			- \int_\Omega \bbb_0 : \bbc(\cdot, 0) \dx; 
% 	\label{eqs:B-discrete-formulation}
% \end{align}
\end{subequations}

Now let $ \cE^N(t) $ be the piecewise linear interpolant of
\begin{equation*}
	 \cE_{tot}(\bu_k, \phi_k, \widetilde{\bbb}_k) 
	 \coloneqq \cE_{AGG}(\bu_k, \phi_k) + \frac{1}{h} \int_{I_k} \cE_{B}(\widetilde{\bbb}_k,\phi_k)(s) \,\d s
\end{equation*}
at $ t_k = kh $ given by
\begin{equation*}
	\cE^N(t) 
	= \frac{(k + 1)h - t}{h} \cE_{tot}(\bu_k, \phi_k, \widetilde{\bbb}_k)
		+ \frac{t - kh}{h} \cE_{tot}(\bu_{k + 1}, \phi_{k + 1}, \widetilde{\bbb}_{k + 1})
\end{equation*}
for $ t \in [kh, (k + 1)h) $, satisfying $ \cE^N(0) = \cE_{tot}(\bu_0, \phi_0^N, \bbb_0) \coloneqq \cE_{AGG}(\bu_0, \phi_0^N) + \cE_{B}(\bbb_0,\phi_0^N) $.
By the discrete energy inequality \eqref{eqs:discrete-energy-Notime}, we have
\begin{equation}
	\label{eqs:discrete-energy-interpolation}
	- \ddt \cE^N(t)
	= \frac{\cE_{tot}(\bu_k, \phi_k, \widetilde{\bbb}_k) - \cE_{tot}(\bu_{k + 1}, \phi_{k + 1}, \widetilde{\bbb}_{k + 1})}{h} \geq \cD^N(t),
\end{equation}
where the piecewise constant dissipation $ \cD^N(t) $ is given by
\begin{align*}
	\cD^N(t) & \coloneqq
	\int_\Omega \frac{\nu(\phi_k)}{2} \abs{\nabla \bu_{k + 1} + \tran{\nabla} \bu_{k + 1}}^2 \dx 
	    + \int_\Omega m(\phi_k) \abs{\nabla q_{k + 1}}^2 \dx
	\\
	& \quad + \frac{1}{h} \int_\Omega \rho_k \frac{\abs{\bu_{k + 1} - \bu_k}^2}{2} \dx
        + \frac{1}{h} \int_\Omega \frac{\abs{\nabla \phi_{k + 1} - \nabla \phi_k}^2}{2} \dx 
	    + \eta \int_\Omega \abs{\frac{\phi_{k + 1} - \phi_k}{h}}^2 \dx
	\\
	& \quad 
	    + \frac{\kappa}{d} \frac{1}{h} \int_{I_{k+1}} \int_\Omega \abs{\nabla \tr \ln \widetilde\bbb_{k + 1}}^2 \dx \d s
	    + \frac{1}{h} \int_{I_{k+1}} \int_\Omega \frac{\mu_\eta(\phi_k)}{2} \tr(\widetilde\bbb_{k + 1}
	+ \widetilde\bbb_{k + 1}^{-1} - 2\bbi) \dx \d s
\end{align*}
for $ t \in I_{k+1} $, $ k \in \{0,...,N-1\} $.

\subsection{Existence of Weak Solutions for the Regularized System}
\label{sec:proof-reg} 
% \footnote{Follow the compactness arguments as in [Abels--Depner--Garcke, JFMF2013], similarly to \cite[Section 3.2]{GGW2019}.}
%\footnote{Should we repeat the compactness argument as in [Abels--Depner--Garcke, JFMF2013]? We will anyway do this in the final proof (Section 5.2).
%\\
%@Y: Good point. Maybe it is enough to say: "This can be done with standard arguments which we will carry out in the final proof in Section 5.2".\\
%@D: Maybe we can follow the similar argument as in \cite[Section 3.2]{GGW2019}.
%\\
%@Y: Actually i think this is even necessary to give more arguments for compactness..}
Now we are ready to prove Theorem \ref{thm:main-reg} by compactness arguments and limit passages. 
% \footnote{Only show the 3D case here. Do we need to point out 2D? Even it is easier... Or we comment that 2D case will be done similarly in the final proof of main theorem. \\@Y: Comment on 2D case should be enough.}

{\textit{\uline{Uniform bounds and convergences.}}}
We obtain the energy inequality for the approximate solution $ (\bu^N, \phi^N, q^N, \widetilde{\bbb}^N) $
by integrating \eqref{eqs:discrete-energy-interpolation} over $ I_{k+1} $ and summing over $ k = 0,..., m $, where $m\in\{0,...,N-1\}$ and $\tau=hm$, together with \eqref{eqs:identity-integral-summation},
%\footnote{The differences $\bu^N-\bu_h^N$ (same with $\nabla\phi$) with the $\frac1h$ factor are important. This helps us to show that $\bu^N,\bu_h^N$ have the same limit function as $N\to\infty$. 
%	\\
%	From the definitions of $\tilde \bbb^N$ and $\bbb^N$, we already know that they will have the same limit function.\\
%	@D: Yes, thank you!} 
\begin{equation}
	\label{eqs:model-discrete-energy-dissipation}
	\begin{aligned}
		& \cE_{AGG}(\phi^N(\tau), \bu^N(\tau))
		+ \frac{1}{h} \int_{\tau - h}^{\tau} \cE_{B}(\widetilde{\bbb}^N(t), \phi^N(t)) \dt
		\\
		&\qquad 
		+ \frac{1}{2 h} \int_0^\tau \int_\Omega \Big( \rho_h^N \abs{\bu^N - \bu_h^N}^2  + \abs{\nabla\phi^N - \nabla\phi_h^N}^2 \Big) \dxdt
		\\
		& \qquad+ \int_0^\tau \int_\Omega \Big( \frac{\nu(\phi_h^N)}{2} \abs{\nabla \bu^N + \tran{\nabla} \bu^N}^2 
		+ m(\phi_h^N) \abs{\nabla q^N}^2 
		+ \eta \abs{\ptial{t,h}^- \phi^N}^2 \Big) \dxdt
		\\
		& \qquad
		+ \int_0^\tau \int_\Omega \Big( \frac{\kappa}{d} \abs{\nabla \tr \ln \widetilde{\bbb}^N}^2 
		+ \frac{\mu_\eta(\phi_h^N)}{2} \tr(\widetilde\bbb^N
		+ (\widetilde\bbb^N)^{-1} - 2\bbi) 
		\Big) \dxdt 
		\\
		& \quad \leq {\cE_{AGG}(\phi^N_0, \bu_0) 
		+ \cE_{B}(\bbb_0, \phi^N_0)}
	\end{aligned}
\end{equation}
which induces the boundedness of certain norms. %\todo{Write it with initial data $\phi_0$ and so on}
However, even with these uniform bounds, the limit passing is still not possible for $ N \rightarrow \infty $. The reason is that, for $ \widetilde{\bbb}^N $, there is no compactness available at the moment (only {$ \normm{\tr \, \bbb^N}_{L^1} $}). %\todo{Added space.} 
To overcome this problem, we recall the estimate \eqref{eqs:weak-B-energyestimate} derived in Section \ref{sec:B}, namely,
\begin{equation}
	\label{eqs:B-discrete-tilde}
	\begin{aligned}
		& \norm{\tr(\widetilde{\bbb}^N - \ln \widetilde{\bbb}^N)(\tau)}_{L^1}
		+ \norm{\widetilde{\bbb}^N(\tau)}_{L^2}^2 
		\\
		& \quad + \int_0^\tau \Big(\norm{\widetilde{\bbb}^N(t)}_{L^2}^2
		+ \kappa \norm{\nabla \widetilde{\bbb}^N(t)}_{L^2}^2\Big) \dt 
		\\
		& \quad + \int_0^\tau \Big(\norm{\tr(\widetilde{\bbb}^N
			+ \inv{(\widetilde{\bbb}^N)} - 2\bbi)(t)}_{L^1} 
		+ \kappa \norm{\nabla \tr \ln \widetilde{\bbb}^N(t)}_{L^2}^2\Big) \dt 
		\\
		& \leq C\Big({\cE_{B}(\bbb_0, \phi^N_0)}, \norm{\bbb_0}_{L^2}\Big)
	\end{aligned}
\end{equation}
for all $ \tau \in (0,T) $, where $ C > 0 $ depends on $ \eta > 0 $ and certain norms of $ (\bu^N, \phi^N) $, which are bounded uniformly regarding $ N $ due to \eqref{eqs:model-discrete-energy-dissipation}. The right-hand side of \eqref{eqs:B-discrete-tilde} is also uniformly bounded. Combining \eqref{eqs:model-discrete-energy-dissipation} and \eqref{eqs:B-discrete-tilde}, we obtain the following uniform bounds on $ N $ (resp. $h$): 
%\footnote{Same bounds hold true for $\bbb^N$?\\
%$ \int_0^\tau \normm{\bbb^N(t)}_{L^2}^2 \dt = h \sum_k \normm{\bbb_{k+1}}_{L^2}^2 = h \sum_k \normm{\widetilde\bbb_{k+1}(h(k+1))}_{L^2}^2 \leq h \int_0^\tau \normm{\widetilde\bbb^N(t)}_{L^2}^2 \dt $?
%\\ Use $\tilde\bbb^N \in C_w([0,T];L^2(\Omega;\bbr^{d\times d}_{\mathrm{sym}}))$, which implies the boundedness in this space.}
% \\
% Don't we have $\norm{\widetilde\bbb^N(t) - \bbb^N(t)}_{[W^{1,2}(\Omega; \bbr_{\mathrm{sym}}^{d \times d})]'} 
% \rightarrow 0$ ? This should control the time error, i.e. $\int_0^\tau \normm{\bbb^N(t)}_{L^2}^2 \dt \leq 
% \text{``difference''} + \int_0^\tau \normm{\widetilde\bbb^N(t)}_{L^2}^2 \dt \leq o(h) + \int_0^\tau \normm{\widetilde\bbb^N(t)}_{L^2}^2 \dt$, where $o(h)\to 0$ as $h\to 0$ ?\\
% By using that $L^2$ is dense in $W^{-1,2}$?
% \\
% I am not 100percent sure...\\
% This is what I am a bit confused... Looks like I missed something from Analysis I/II/III...\\
% Let me think about this. Let's think/discuss in the coffee break.:)}
\begin{alignat*}{3}
	\bu^N & \quad \text{ is bounded in } \quad && L^2(0,T; W_0^{1,2}(\Omega; \bbr^d)) \text{ and } L^\infty(0,T; L_\sigma^2(\Omega)), \\
	\nabla q^N & \quad \text{ is bounded in } \quad && L^2(0,T; L^2(\Omega; \bbr^d)), \\
	\phi^N & \quad \text{ is bounded in } \quad && L^\infty(0,T; W^{1,2}(\Omega)), \\
	\widetilde{\bbb}^N & \quad \text{ is bounded in } \quad && L^2(0,T; W^{1,2}(\Omega; \bbr_{\mathrm{sym}}^{d \times d})) \text{ and } L^\infty(0,T; L^2(\Omega; \bbr_{\mathrm{sym}}^{d \times d})), \\
 \tr \ln \widetilde{\bbb}^N & \quad \text{ is bounded in } \quad && L^\infty(0,T; L^1(\Omega)), \\
	\nabla \tr \ln \widetilde{\bbb}^N & \quad \text{ is bounded in } \quad && L^2(0,T; L^2(\Omega; \bbr^d)), \\
	\sqrt{\eta} \ptial{t,h}^- \phi^N & \quad \text{ is bounded in } \quad && L^2(0,T; L^2(\Omega)),
\end{alignat*}
and by 
\eqref{eqs:subdifferential-estimate} and 
Lemma \ref{lem:convex-potential-L1-q},
\begin{gather*}
    \begin{alignedat}{3}
	    \phi^N & \quad \text{ is bounded in } \quad && L^2(0,T; W^{2,2}(\Omega)), \\
	    \widetilde{W}_0'(\phi^N) & \quad \text{ is bounded in } \quad && L^2(0,T; L^2(\Omega)), 
	   % \widetilde{W}_0'(\phi^N) & \quad \text{ is bounded in } \quad && L^\infty(0,T; L^1(\Omega)),
    \end{alignedat} \\
	\int_0^T \abs{\int_\Omega q^N \dx} \dt \leq  M(T),
\end{gather*}
for a certain monotone function $ M: \bbr^+ \rightarrow \bbr^+ $.
% \footnote{I think we don't need that $M$ is monotone w.r.t. $T\in(0,\infty)$, as we only look at finite time intervals (not as in \cite{ADG2013}).\\ @D: Yes. But here the monotone is not assumed, but due to estimate, we have $ M(T) = C (T^\delta + 1) $ for some $\delta > 0$.} 
Up to a subsequence (not to be relabeled), one concludes the following convergences:
% \footnote{weak convergence of $\nabla q^N$ in $L^2(L^2)$ is covered with the weak convergence of $q^N$ in $L^2(W^{1,2})$? \\ @D: yes, that's true. Here the point is that the $L^2(L^2)$-convergence is not obtained for free, instead due to the $L^1$ and gradient. I comment the gradient convergence and make a small remark later.}
\begin{alignat*}{4}
	\bu^N & \rightarrow \bu, \quad && \text{weakly} \quad && \text{in } L^2(0,T; W_0^{1,2}(\Omega; \bbr^d)), \\
	\bu^N & \rightarrow \bu, \quad && \text{weakly-}* \quad && \text{in } L^\infty(0,T; L_\sigma^2(\Omega)) \cong [L^1(0,T; L_\sigma^2(\Omega))]', \\
	\phi^N & \rightarrow \phi, \quad && \text{weakly} \quad && \text{in } L^2(0,T; W^{2,2}(\Omega)), \\
	\phi^N & \rightarrow \phi, \quad && \text{weakly-}* \quad && \text{in } L^\infty(0,T; W^{1,2}(\Omega)) \cong { [L^1(0,T; (W^{1,2}(\Omega))')]', } \\
	\sqrt{\eta} \ptial{t,h}^- \phi^N & \rightarrow \sqrt{\eta} \pt \phi, \quad && \text{weakly} \quad && \text{in } L^2(0,T; L^2(\Omega)), \\
	q^N & \rightarrow q, \quad && \text{weakly} \quad && \text{in } L^2(0,T; W^{1,2}(\Omega)), \\
% 	\nabla q^N & \rightarrow \nabla q, \quad && \text{weakly} \quad && \text{in } L^2(0,T; L^2(\Omega; \bbr^d)), \\
	\widetilde{\bbb}^N & \rightarrow \bbb, \quad && \text{weakly} \quad && \text{in } L^2(0,T; W^{1,2}(\Omega; \bbr_{\mathrm{sym}}^{d \times d})), \\
	\widetilde{\bbb}^N & \rightarrow \bbb, \quad && \text{weakly-}* \quad && \text{in } L^\infty(0,T; L^2(\Omega; \bbr_{\mathrm{sym}}^{d \times d})) \cong [L^1(0,T; L^2(\Omega; \bbr_{\mathrm{sym}}^{d \times d}))]', \\
	\nabla \tr \ln \widetilde{\bbb}^N & \rightarrow \overline{\nabla \tr \ln \bbb}, \quad && \text{weakly} \quad && \text{in } L^2(0,T; L^2(\Omega; \bbr^d)).
\end{alignat*}
Here the $ L^2(\Omega) $-convergence of $ q^N $ is derived by the Poincar\'e--Wirtinger inequality together with the integrability from above.

{\textit{\uline{Compactness for $\widetilde{\bbb}^N$.}}}
Recall the definition of $\widetilde{\bbb}^N(t) = \widetilde{\bbb}_{k+1}(t)$ for $ t \in [t_{k}, t_{k+1}) $, $k\in\{0,...,N-1\}$, with $\widetilde{\bbb}_{k+1}(t)$ defined by \eqref{eqs:B-equation-discrete}.
Thanks to the boundedness of $ \bu^N, \phi^N $, the mollifier $ \sR_\eta $ and the weak formulation \eqref{eqs:B-discrete-formulation} restricted to each time interval $I_{k+1}$, $k\in\{0,...,N-1\}$, we get 
% \footnote{I think there are some $h$ missing somewhere...\\I don't get why...\\
% I think it must be $\mathbf{h} \sum_{k=0}^{N-1} \int_{I_{k+1}} \norm{\pt \widetilde{\bbb}_{k+1}(t)}_{[W^{1,2}]'}^2 \dt \leq C$. I am checking this.\\
% This only works for piecewise constant functions? as in (4.24)?\\ No everything is fine. We get here a quadratic 
% order, i.e. $h^2$, because left-hand side is also squared. :)\\
% So nothing needed to be changes here?\\ Correct. I just give one more line for better understanding.\\Great! Thank you! I will have a meeting at 4pm and I will be back after that :). See you.} 
\begin{alignat}{3}
\label{eqs:dtB_k+1}  
    &\sum_{k=0}^{N-1} \int_{I_{k+1}} \norm{\pt \widetilde{\bbb}_{k+1}(t)}_{[W^{1,2}]'}^2 \dt
    \\
    \nonumber  
    &\leq C(\eta) \Big( \norm{\bu^N}_{L^2(0,T; W^{1,2})}^2 \norm{\widetilde\bbb^N}_{L^\infty(0,T; L^2)}^2
    + \norm{\widetilde\bbb^N}_{L^2(0,T; W^{1,2})}^2 
    \big( 1 + \norm{\nabla\phi_h^N}_{L^\infty(0,T; L^2)}^2 \big) + 1 \Big),
        %\leq C(\eta) \Big( \norm{\bu^N}_Y^2  + \norm{\widetilde\bbb^N}_Y^2 + \norm{\nabla \phi^N}_Y^2  \Big),
\end{alignat}
where $C(\eta)>0$ depends on $\eta$ but not on $N\in\bbn$.
%with $Y = L^\infty(0,T;L^2(\Omega)) \cap L^2(0,T; W^{1,2}(\Omega))$, where $C(\eta)>0$ depends on $\eta$ but not on $N\in\bbn$.
%which implies that $\widetilde{\bbb}_{k+1} \in C(\overline{I}_{k+1}; [W^{1,2}(\Omega; \bbr_{\mathrm{sym}}^{d \times d})]')$.
This allows us to employ a time translation compactness argument for $\widetilde\bbb^N$, that is,
\begin{align*}
    &\int_0^{T-h} \norm{\widetilde\bbb^N(t+h) - \widetilde\bbb^N(t)}_{[W^{1,2}]'}^2 \dt
    \\
    & \quad = \sum_{k=0}^{N-2} \int_{I_{k+1}} \norm{\widetilde\bbb_{k+2} (t+h) - \widetilde\bbb_{k+1} (t)}_{[W^{1,2}]'}^2 \dt
    \\
    & \quad \leq \sum_{k=0}^{N-2} \int_{I_{k+1}} 2 \Big( \norm{\widetilde\bbb_{k+2} (t+h) - \overline\bbb_{k+1} (t)}_{[W^{1,2}]'}^2
    + \norm{\overline\bbb_{k+1} (t) - \widetilde\bbb_{k+1} (t)}_{[W^{1,2}]'}^2 \Big) \dt
    \\
    & \quad = 2 \sum_{k=1}^{N-1} \int_{I_{k+1}}  \norm{\widetilde\bbb_{k+1} (t) - \widetilde\bbb_{k+1} (t_{k})}_{[W^{1,2}]'}^2 \dt
    + 2 \sum_{k=0}^{N-2} \int_{I_{k+1 }}\norm{\overline\bbb_{k+1} (t) - \widetilde\bbb_{k+1} (t)}_{[W^{1,2}]'}^2  \dt,
\end{align*}
where $\overline{\bbb}_{k+1}(t) = \frac{1}{h} \int_{I_{k+1}} \widetilde\bbb_{k+1}(s)\,\d s$, as defined in Section \ref{sec:time-discretization}. 
For the first term, using the fundamental theorem of calculus, Jensen's inequality and \eqref{eqs:dtB_k+1}, we calculate 
\begin{align*}
    \sum_{k=1}^{N-1} \int_{I_{k+1}}  \norm{\widetilde\bbb_{k+1} (t) - \widetilde\bbb_{k+1} (t_{k})}_{[W^{1,2}]'}^2 \dt 
    % &\leq \sum_{k=1}^{N-1} \int_{I_{k+1}}  \norm{\widetilde\bbb_{k+1} (t_{k+1}) - \widetilde\bbb_{k+1} (t_{k})}_{[W^{1,2}]'}^2 \dt 
    \leq h^2 \sum_{k=1}^{N-1} \int_{I_{k+1}}  \norm{\partial_t \widetilde\bbb_{k+1} (t)}_{[W^{1,2}]'}^2 \dt 
    \leq Ch^2,
\end{align*}
where $C$ does not depend on $N\in\bbn$. For the second term, we also use Jensen's inequality, the fundamental theorem of calculus and \eqref{eqs:dtB_k+1} to get
\begin{align*}
    &\sum_{k=0}^{N-2} \int_{I_{k+1 }}\norm{\overline\bbb_{k+1} (t) - \widetilde\bbb_{k+1} (t)}_{[W^{1,2}]'}^2  \dt
    \\
    & \quad \leq 
    \frac1h \sum_{k=0}^{N-2} \int_{I_{k+1 }} \int_{I_{k+1 }}\norm{\widetilde\bbb_{k+1} (s) - \widetilde\bbb_{k+1} (t)}_{[W^{1,2}]'}^2  \,\d s \dt 
    \\
    & \quad \leq C h^2 \sum_{k=0}^{N-2} \int_{I_{k+1}}  \norm{\partial_t \widetilde\bbb_{k+1} (t)}_{[W^{1,2}]'}^2 \dt
    \\
    & \quad \leq C h^2,
\end{align*}
where $C$ is independent of $N\in\bbn$.
Then, one concludes that
\begin{align*}
    \int_0^{T-h} \norm{\widetilde\bbb^N(t+h) - \widetilde\bbb^N(t)}_{[W^{1,2}]'}^2 \dt
    \to 0,
\end{align*}
as $N\to \infty$ ($h\to 0$, respectively).
By virtue of Lemma \ref{lem:compactness_translation}, one obtains the strong convergences
\begin{alignat*}{4}
	\widetilde{\bbb}^N & \rightarrow \bbb, \quad && \text{strongly} \quad && \text{in } L^2(0,T; L^p(\Omega; \bbr_{\mathrm{sym}}^{d \times d})), \quad 2 \leq p < { \frac{2d}{d-2}},
	\\
	\widetilde{\bbb}^N & \rightarrow \bbb, \quad && \text{a.e.} \quad && \text{in } Q_T.
\end{alignat*}
Proceeding in a similar fashion as in Section \ref{sec:B-positive-delta} together with the Poincar\'e--Wirtinger inequality leads to
\begin{align*}
	& \bbb \text{ is positive definite a.e.~in } Q_T, 
	 \\
	& \tr \ln \widetilde{\bbb}^N \rightarrow \tr \ln \bbb, \quad \text{weakly} \quad \text{in } L^2(0,T; W^{1,2}(\Omega)),
\end{align*}
It follows from the continuity of $ \tr \ln (\cdot) $ that
\begin{alignat}{4}
	\label{eqs:trlnBN-a.e.-convergence}
	\tr \ln \widetilde{\bbb}^N & \rightarrow \tr \ln \bbb, \quad && \text{a.e.} \quad && \text{in } Q_T.
\end{alignat}
%On noting the embedding $ W^{1,2}(0,T; [W^{1,2}(\Omega; \bbr_{\mathrm{sym}}^{d \times d})]') \hookrightarrow C([0,T]; [W^{1,2}(\Omega; \bbr_{\mathrm{sym}}^{d \times d})]') $, one sees 
%\begin{equation*}
%    \norm{\widetilde\bbb^N(t) - \bbb^N(t)}_{[W^{1,2}(\Omega; \bbr_{\mathrm{sym}}^{d \times d})]'} 
%    \rightarrow 0 \quad \text{ for all } t \in (0,T),
%\end{equation*}
%as $ N \rightarrow \infty $ (resp. $ h \rightarrow 0 $), which indicates the same limit function
%% \footnote{up to an additive constant? (which will be zero as both hopefully have the same mean value?). Good point, thank you.\\
%% Here I used the continuous w.r.t. time. So I think there should not be a constant?}
%of $ \bbb^N $ as $ \widetilde\bbb^N $ and up to a non-relabeled subsequence,
%\begin{alignat}{4}
%    \bbb^N & \rightarrow \bbb \quad && \text{a.e.} \quad && \text{in } Q_T, \\
%    \label{eqs:trlnBN-a.e.-convergence}
%	\tr \ln \bbb^N & \rightarrow \tr \ln \bbb, \quad && \text{a.e.} \quad && \text{in } Q_T.
%\end{alignat}
Arguing as in Section \ref{sec:B-proof} yields
\begin{equation*}
    \bbb \in C_w([0,T]; L^2(\Omega; \bbr_{\mathrm{sym}}^{d \times d})).
\end{equation*}
%Moreover, in view of \eqref{eqs:B-tilde-L2-strong},
%\begin{equation}
%    \label{eqs:B-tilde-L2-strong-t}
%    \widetilde{\bbb}^N \rightarrow \bbb, \quad \text{strongly} \quad \text{in } L^2(0,T; L^2(\Omega; \bbr_{\mathrm{sym}}^{d \times d})).
%\end{equation}
%By Lemma \ref{lem:C_w}, we have $\widetilde\bbb^N \in C_w([0,T]; L^2(\Omega; \bbr_{\mathrm{sym}}^{d \times d}))$ as well, which, by definitions of $\bbb^N$ and weak continuity, implies
%\begin{equation}
%    \label{eqs:B-tilde-L2-weak-t}
%    \int_0^T \int_\Omega \big( \widetilde\bbb^N(t) - \bbb^N(t) \big) : \bbc(t) \dxdt \rightarrow 0, \quad \text{as } N \rightarrow \infty\  (\text{resp. } h \rightarrow 0),
%\end{equation}
%for all $\bbc \in L^2(0,T; L^2(\Omega; \bbr_{\mathrm{sym}}^{d \times d}))$.
%Then combining \eqref{eqs:B-tilde-L2-strong-t} and \eqref{eqs:B-tilde-L2-weak-t}, one obtains
%\begin{equation}
%    \label{eqs:B-N-L2-weak}
%    \bbb^N \rightarrow \bbb, \quad \text{weakly} \quad \text{in } L^2(0,T; L^2(\Omega; \bbr_{\mathrm{sym}}^{d \times d})).
%\end{equation}

{\textit{\uline{Compactness for $\widetilde{\phi}^N$.}}}
Now let $ \widetilde{\phi}^N $ be the piecewise linear interpolant of $ \phi^N(kh) $, $ k \in \{0,...,N\}$, i.e., $ \widetilde{\phi}^N = \frac{1}{h} \chi_{[0,h]} *_t \phi^N $, where the convolution only happens regarding time variable $ t $. Then it follows that $ \pt \widetilde{\phi}^N = \ptial{t,h}^- \phi^N $ and, for a.e.~$t\in(0,T)$, %\footnote{for all $t\in(0,T)$?\\ a.e.?}
\begin{equation}
    \label{eqs:phi^N-tilde}
    \normm{\widetilde{\phi}^N - \phi^N}_{[W^{1,2}(\Omega)]'}
    \leq C h \normm{\pt \widetilde{\phi}^N}_{[W^{1,2}(\Omega)]'}.
\end{equation}
In view of the weak formulation \eqref{eqs:AGG-discrete-phi-formulation} and the boundedness of $ (\bu^N, \phi^N, q^N) $ in {appropriate} spaces, we see that 
\begin{alignat*}{3}
	\pt \widetilde{\phi}^N & \quad \text{ is bounded in } \quad && L^2(0,T; [W^{1,2}(\Omega)]'),
\end{alignat*}
which, together with Aubin--Lions thereby implies the strong convergence 
\begin{alignat*}{4}
	\widetilde{\phi}^N & \rightarrow \widetilde{\phi}, \quad && \text{strongly} \quad && \text{in } L^2(0,T; W^{1,p}(\Omega)), \quad &&2 \leq p < { \frac{2d}{d-2} }, \\
	\widetilde{\phi}^N & \rightarrow \widetilde{\phi}, \quad && \text{strongly} \quad && \text{in } C([0,T]; L^p(\Omega)), \quad &&2 \leq p < { \frac{2d}{d-2} }, \\
	\widetilde{\phi}^N & \rightarrow \widetilde{\phi}, \quad && \text{a.e.} \quad && \text{in } Q_T,
\end{alignat*}
for some $ \widetilde{\phi} \in L^\infty(0,T; W^{1,2}(\Omega)) \cap L^2(0,T; W^{2,2}(\Omega)) $, thanks to the boundedness of $ \widetilde{\phi}^N $, which can be derived from that of $ \phi^N $. Note that \eqref{eqs:phi^N-tilde} indicates that
\begin{equation*}
    \widetilde{\phi}^N - \phi^N \rightarrow 0 \quad \text{in } L^2(0,T; [W^{1,2}(\Omega)]'), \quad \text{as } N \rightarrow \infty,
\end{equation*}
which gives $ \widetilde{\phi} = \phi $. Then we have $ \pt \phi \in L^2(0,T; [W^{1,2}(\Omega)]') $ and hence 
\begin{equation*}
    \phi \in C_w([0,T]; W^{1,2}(\Omega))
\end{equation*}
due to Lemma \ref{lem:C_w}. Next, we verify the identity $ \phi(0) = \phi_0 $, which can be recorded from
\begin{equation*}
    \widetilde{\phi}^N(0) \rightarrow \widetilde{\phi}(0) = \phi(0), \quad \text{strongly} \text{ in } L^p(\Omega),
\end{equation*}
and the fact that $ \widetilde{\phi}^N(0) = \phi_0^N $ with $ \phi_0^N \rightarrow \phi_0 $ in $ W^{1,2}(\Omega) $. 
In addition, it holds
\begin{equation*}
    \int_\Omega \phi^N(\cdot,\tau - h) \overline{\xi}_\tau \dx
    + \int_\Omega \phi_0^N \overline{\xi}_0 \dx
    \to \int_\Omega \phi(\cdot,\tau) \xi(\cdot,\tau) \dx
    + \int_\Omega \phi_0 \xi(\cdot,0) \dx, 
\end{equation*}
as $N\to \infty$ (resp. $h \to 0$) for a.e.~$\tau \in (0,T)$, concerning the weak convergence of $\phi^N(\tau)$, and strong convergence of $\overline{\xi}_\tau \to \xi(\tau)$ in $L^2(\Omega)$ for fixed $\tau$.
Moreover, the continuity of $ m(\cdot) $, $ \nu(\cdot) $, $ \mu(\cdot) $ and the almost everywhere convergence of $ \phi^N $ yield
\begin{alignat*}{4}
	m(\phi^N) & \rightarrow m(\phi), \quad && \text{a.e.} \quad && \text{in } Q_T, \\
	\nu(\phi^N) & \rightarrow \nu(\phi), \quad && \text{a.e.} \quad && \text{in } Q_T, \\
	\mu(\phi^N) & \rightarrow \mu(\phi), \quad && \text{a.e.} \quad && \text{in } Q_T, \\
	\frac{\mu_\eta(\phi^N) - \mu_\eta(\phi_h^N)}{\sR_\eta \phi^N - \sR_\eta \phi_h^N} & \rightarrow \mu_\eta'(\phi), \quad && \text{a.e.} \quad && \text{in } Q_T. 
\end{alignat*}
Then one can pass to the limit of \eqref{eqs:AGG-discrete-phi-formulation} to \eqref{eqs:weak-reg-phi-formulation} and \eqref{eqs:B-discrete-formulation} to \eqref{eqs:weak-reg-B-formulation} as $ N \rightarrow \infty $, with the help of the mollifier $ \sR_\eta $ and the convergence results from above. %weak convergence of $ \bu^N $ and $\bbb^N$. 
%\footnote{Should we give more arguments, like dominated convergence, or the limit passing in terms like $\nabla \bbb : \nabla (\bbc/\mu_\eta(\phi))$? \\Maybe only this special term to do the argument? \\ No, I would not specify it.}

By virtue of the weak convergences of $q^N$, $\phi^N$, $\pt \phi^N$, $ \widetilde{\bbb}^N $ and $ \tr \ln \widetilde{\bbb}^N $, together with Lemma \ref{lem:mollification}, one knows the right-hand side of \eqref{eqs:AGG-discrete-q-formulation} denoted by $ Q^N $ converges weakly in $ L^2(0,T; L^2(\Omega)) $ to
\begin{equation*}
    Q \coloneqq q + \omega \phi
        - \eta \pt \phi
		- \onehalf \sR_\eta \Big[\mu_\eta'(\phi) \tr(\bbb - \ln \bbb - \bbi)\Big],
\end{equation*}
by applying the convergences result above term by term. Here, the weak convergence of the last term on the right-hand side is valid due to Lemmata \ref{lem:mollification} and \ref{lem:weak-convergence-time-averaged}.
On the other hand, 
\begin{equation*}
    \inner{\partial \widetilde{E}(\phi^N)}{\phi^N}
    = \inner{Q^N}{\phi^N}
    \rightarrow \inner{Q}{\phi}, \text{ as } N \rightarrow \infty
\end{equation*}
due to the strong convergence of $ \phi^N $ in $ L^2(0,T; L^2(\Omega)) $. Therefore, by e.g.~\cite[Theorem 9.13-2]{Ciarlet2013} for monotone operators one knows $ \partial \widetilde{E}(\phi) = Q $, which is exactly \eqref{eqs:weak-reg-q-formulation}. 

{\textit{\uline{Compactness for $\bu^N$.}}}
Next, we are going to get compactness of $\bu^N$ in $ L^2(0,T; L^2(\Omega; \bbr^d)) $, which implies a pointwise almost everywhere convergence. This is in general not a problem in the case of the matched density (constant $\rho$), for which one can use the same strategy as of $ \bbb^N $ to achieve the strong convergence by the Aubin--Lions lemma. However, it is not possible to apply the same argument directly for $ \bu^N $ with unmatched densities, here, instead, we make use of the Helmholtz projection $ \bbp_\sigma $ onto $ L_\sigma^2(\Omega) $ as in \cite{ADG2013} (also called Leray projection). With the uniform boundedness of $ \bu^N $, $ \phi^N $, $ q^N $ and $ \bbb^N $, it follows that
\begin{alignat*}{3}
	\rho_h^N \bu^N \otimes \bu^N & \  \text{ is bounded in } \  && L^2(0,T; L^2(\Omega; \bbr^{d \times d})), \\
	\nabla \bu^N + \tran{\nabla} \bu^N & \  \text{ is bounded in } \  && L^2(0,T; L^2(\Omega; \bbr^{d \times d})), \\
	\bu^N \otimes \nabla q^N & \  \text{ is bounded in } \  && L^{\frac{8}{7}}(0,T; L^{\frac{4}{3}}(\Omega; \bbr^{d \times d})), \\
	q^N \nabla \phi_h^N & \  \text{ is bounded in } \  && L^2(0,T; L^{\frac{3}{2}}(\Omega; \bbr^d)), \\
	\sR_\eta [\mu_\eta(\phi_h^N)(\overline{\bbb}^N - \bbi)] & \  \text{ is bounded in } \  && L^2(0,T; L^6(\Omega; \bbr_{\mathrm{sym}}^{d \times d})), \\
	\sR_\eta \Big[\frac{\mu_\eta(\phi_h^N)}{2} \nabla G^N\Big] & \  \text{ is bounded in } \  && L^2(0,T; L^2(\Omega; \bbr^d)).
\end{alignat*}
Note that the first four bounds follow directly from \cite[Page 474]{ADG2013}, while the remaining bounds are valid due to the boundedness of $ \widetilde{\bbb}^N $, the uniform upper bound of $ \mu $ and Lemma \ref{lem:mollification}, Remark \ref{rem:uniform-bound-time-averaged}. Then going back to the equation \eqref{eqs:AGG-discrete-u-formulation}, one may infer that
\begin{alignat*}{3}
	\pt \big( \bbp_\sigma(\widetilde{\rho \bu}^N) \big) & \quad  \text{ is bounded in } \  && L^\frac{8}{7}(0,T; W^{-1,4}(\Omega; \bbr^d)),
\end{alignat*}
where $ \widetilde{\rho \bu}^N $ is the piecewise linear interpolant of $ \rho^N \bu^N(kh) $ for $ k \in \{0,...,N\}$, which fulfulls $ \pt \big( \widetilde{\rho \bu}^N \big) = \ptial{t,h}^-(\rho^N \bu^N) $.
Moreover,%\footnote{in $L^2(W^{1,2})$ instead of $L^2(W^{1,4})$, right? I changed that.\\yes, thank you.}
\begin{alignat*}{3}
	\bbp_\sigma(\widetilde{\rho \bu}^N) & \quad  \text{ is bounded in } \  && L^2(0,T; W^{1,2}(\Omega; \bbr^d)).
\end{alignat*}
In light of the Aubin--Lions lemma, one arrives at the strong convergence
\begin{alignat*}{4}
    % \label{eqs:strong-convergence_Prhou}
	\bbp_\sigma(\widetilde{\rho \bu}^N) & \rightarrow \overline{\bbp_\sigma(\rho(\phi) \bu)}, \quad && \text{strongly} && \text{ in }  L^2(0,T; L^2(\Omega; \bbr^d)),
\end{alignat*}
for some function $ \overline{\bbp_\sigma(\rho(\phi) \bu)} \in L^\infty(0,T; L^2(\Omega; \bbr^d)) $. Indeed, 
\begin{equation}
    \label{eqs:rho-u-N-weak}
    \widetilde{\rho \bu}^N \rightarrow \rho(\phi) \bu, \quad \text{weakly} \text{ in }  L^2(0,T; L^2(\Omega; \bbr^d)).
\end{equation}
Then by virtue of the weak continuity of the Leray projection $ \bbp_\sigma: L^2(0,T; L^2(\Omega; \bbr^d)) \rightarrow L^2(0,T; L_\sigma^2(\Omega)) $, we obtain
\begin{equation*}
	\overline{\bbp_\sigma(\rho(\phi) \bu)} = \bbp_\sigma(\rho(\phi) \bu).
\end{equation*}
Now in view of the strong convergence of $ \bbp_\sigma (\rho^N \bu^N) $ and the weak convergence of $ \bu^N $ both in $ L^2(0,T; L^2(\Omega; \bbr^d)) $, one ends up with
\begin{align*}
	\int_0^T \int_\Omega \rho^N \abs{\bu^N}^2 \dxdt
	& = \int_0^T \int_\Omega \bbp_\sigma(\rho^N \bu^N) \cdot \bu^N \dxdt \\
	& \rightarrow
	\int_0^T \int_\Omega \bbp_\sigma(\rho(\phi) \bu) \cdot \bu \dxdt
	= \int_0^T \int_\Omega \rho(\phi) \abs{\bu}^2 \dxdt,
\end{align*}
which implies
\begin{alignat}{4}
    \label{eqs:rho-u-N-strong}
	\sqrt{\rho^N} \bu^N & \rightarrow \sqrt{\rho(\phi)} \bu, \quad && \text{strongly} && \text{ in } L^2(0,T; L^2(\Omega; \bbr^d)),
\end{alignat}
with the weak convergence \eqref{eqs:rho-u-N-weak}.
As $ \rho(r) $ is affine linear regarding $ r $, one gets
\begin{alignat*}{4}
	\rho(\phi^N) & \rightarrow \rho(\phi), \quad && \text{a.e.} \quad && \text{in } Q_T
	\quad
	\text{and}
	\quad
	\rho(\phi^N) \geq C > 0,
\end{alignat*}
which together with the strong convergence \eqref{eqs:rho-u-N-strong} of $ \sqrt{\rho^N} \bu^N $ gives rise to%\footnote{should we mention dominated convergence?\\ I would suggest to keep it as right now, since otherwise everywhere you need to specify it and it's too much.\\Agreed!}
\begin{alignat*}{4}
	\bu^N = \frac{1}{\sqrt{\rho^N}} \sqrt{\rho^N} \bu^N & \rightarrow \frac{1}{\sqrt{\rho(\phi)}} \sqrt{\rho(\phi)} \bu = \bu, \quad  && \text{strongly} && \text{ in }  L^2(0,T; L^2(\Omega; \bbr^d)),
\end{alignat*}
and hence
\begin{alignat*}{4}
	\bu^N & \rightarrow \bu, \quad && \text{a.e.} \quad && \text{in } Q_T.
\end{alignat*}
In addition, it holds
\begin{align*}
    & \int_\Omega \rho(\phi^N(\cdot, \tau-h)) \bu^N(\cdot, \tau-h) \cdot \overline{\bw}_\tau \dx
    \to \int_\Omega \rho(\phi(\cdot, \tau)) \bu(\cdot, \tau) \cdot \bw(\cdot,\tau) \dx, \\
    & \int_\Omega \rho(\phi_0^N) \bu_0 \cdot \overline{\bw}_0 \dx
    \to \int_\Omega \rho(\phi_0) \bu_0 \cdot \bw(\cdot,0) \dx,
\end{align*}
as $ N \to \infty $ (resp. $ h \to 0 $) for a.e.~$\tau \in (0,T)$, concerning the convergences of $\phi^N(\tau)$ and $\bu^N(\tau)$, and the strong convergence of $\overline{\bw}_\tau \to \bw(\tau)$ in $L_\sigma^2(\Omega)$ for fixed $\tau$.
Subsequently, one can pass to the limit in \eqref{eqs:AGG-discrete-u-formulation} to \eqref{eqs:weak-reg-u-formulation} term by term as $ N \rightarrow \infty $ with the strong convergences of $ \bu^N, \phi^N $, except for the terms with respect to $ \widetilde{\bbb}^N $. Let us recall the definition of $ \overline{\bbb}^N $ and $ G^N $
\begin{align*}
    \overline{\bbb}^N(t) & \coloneqq \frac{1}{h} \int_{I_{k+1}} \widetilde{\bbb}_{k+1}(s) \,\d s
    = \frac{1}{h} \int_{I_{k+1}} \widetilde{\bbb}^N(s) \,\d s, \\
    G^N(t) & \coloneqq \frac{1}{h} \int_{I_{k+1}} \tr(\widetilde{\bbb}_{k+1} - \ln \widetilde{\bbb}_{k+1} - \bbi)(s) \,\d s
    = \frac{1}{h} \int_{I_{k+1}} \tr(\widetilde{\bbb}^N - \ln \widetilde{\bbb}^N - \bbi)(s) \,\d s
\end{align*}
for $t \in [t_k, t_{k+1})$, where $k\in\{0,...,N-1\}$. In fact, by Lemma \ref{lem:uniform-bound-time-averaged}, Remark \ref{rem:uniform-bound-time-averaged} and the uniform bounds of $ \widetilde{\bbb}^N $ and $ \tr \ln \widetilde{\bbb}^N $, we know that $ \overline{\bbb}^N $ and $ \nabla G^N $ are uniformly bounded in $ L^2(Q_T) $. With the help of Lemma \ref{lem:weak-convergence-time-averaged}, one concludes
\begin{align*}
    & \int_0^\tau \int_\Omega \big[\mu_\eta(\phi_h^N) (\overline{\bbb}^N - \bbi)\big] : \sR_\eta \nabla \bw \dxdt 
    % \\
    % & \qquad 
        \rightarrow
        \int_0^\tau \int_\Omega \big[\mu_\eta(\phi)( \bbb - \bbi)\big] : \sR_\eta \nabla \bw \dxdt,
	    \\
	& \int_0^\tau \int_\Omega \Big[\frac{\mu_\eta(\phi_h^N)}{2} \nabla G^N\Big] \cdot \sR_\eta \bw \dxdt 
% 	\\
% 	& \qquad 
	    \rightarrow
	    \int_0^\tau \int_\Omega \Big[\frac{\mu_\eta(\phi)}{2} \nabla \tr(\bbb - \ln \bbb - \bbi)\Big] \cdot \sR_\eta \bw \dxdt.
\end{align*}
Using integration by parts over $\Omega$ and the compactness of $ \phi^N $, we get the convergence of $ \int_0^t \int_\Omega q^N \nabla \phi^N \cdot \bw \dxdt $. % analogously as in \cite[Section 5.1]{ADG2013}. 
Then, it remains to verify that $ \bu \in C_w([0,T]; L_\sigma^2(\Omega)) $ and $ \bu(t) \rightarrow \bu_0 $ as $ t \rightarrow 0 $, which are clear by proceeding the same argument as in \cite[Section 5.2]{ADG2013}.

{\textit{\uline{Energy dissipation inequality.}}}
In the final step, one recovers the energy dissipation inequality \eqref{eqs:energy-dissipation-reg-inequality}. To this end, multiplying the discrete energy inequality \eqref{eqs:discrete-energy-interpolation} by a function $ \varsigma \in W^{1,1}(0,T) $ with $ \varsigma \geq 0 $ and $ \varsigma(T) = 0 $, and integrating the resulting inequality over $ (0,T) $ and by parts with respect to the time variable, one obtains
\begin{equation}
	\label{eqs:energy-dissipation-reg-varsigma-eta}
	\cE_{tot}(\bu_0, \phi_0^N, \bbb_0) \varsigma(0)
	+ \int_0^T \cE^N(\tau) \varsigma'(\tau) \dtau
	\geq
	\int_0^T \cD^N(\tau) \varsigma(\tau) \dtau,
\end{equation}
where $ \cE_{tot} $, $ \cE^N $ and $ \cD^N $ are defined in the end of last subsection.
Thanks to the compactness of $ \bu^N $, $ \phi^N $ and $ \widetilde\bbb^N $, we deduce that up to a subsequence (not relabeled), 
% \footnote{The last one? \\
% There are 2 possibilities: As we have a bounded (space-time-)domain, pointwise a.e.-convergence and uniform boundedness of $\tr\ln\widetilde\bbb^n$ in $L^2(L^2)$, we can either apply Vitali's convergence theorem (use H\"older inequality to get uniform integrability of $|\tr\ln\widetilde\bbb^n|^{2-\epsilon}$ with $\epsilon>0$ arbitrarily small), or we can use Lemma of Fatou and Egorov's theorem.
% \\
% Here is also a similar application: 
% https://math.stackexchange.com/questions/183032/lp-space-convergence \\
% @D: Many thanks!
% }
\begin{alignat}{3}
    \bu^N(t) & \rightarrow \bu(t), \quad && \text{in } L_\sigma^2(\Omega), 
    \label{eqs:convergence-u-reg-a.e.} \\
	\phi^N(t) & \rightarrow \phi(t), \quad && \text{in } C(\overline{\Omega}), \\
	\widetilde\bbb^N(t) & \rightarrow \bbb(t), \quad && \text{in } L^2(\Omega; \bbr_{\mathrm{sym}}^{d \times d}),
\end{alignat}
for a.e.~$ t \in (0,T) $, as $ N \to \infty $ (resp. $ h \to 0 $).
On noting the uniform energy estimate \eqref{eqs:model-discrete-energy-dissipation} and $\tr \ln \widetilde{\bbb}^N \to \tr \ln \bbb$ a.e.~in $Q_T$, it follows that
\begin{alignat*}{3}
    \int_\Omega \tr \ln \widetilde{\bbb}^N \dx & \to \int_\Omega \tr \ln \bbb \dx, \quad && \text{weakly-}* \text{ in } L^\infty(0,T).
\end{alignat*}
Then by the weak-$*$ compactness Lemma \ref{lem:weak-*-convergence-time-averaged},
% pointwise (a.e.~in $Q_T$) convergence \eqref{eqs:trlnBN-a.e.-convergence} and the uniform boundedness of $ \tr \ln \widetilde{\bbb}^N(\tau) $ in $L^2(Q_T)$, it follows from Vitali's convergence theorem that $ \tr \ln \widetilde{\bbb}^N $ converges strongly to $ \tr \ln \bbb $ in $ L^{2-\epsilon}(Q_T) $, $ \epsilon \in(0, 1) $, and hence $ \int_\Omega \tr \ln \widetilde{\bbb}^N \dx \to \int_\Omega \tr \ln \bbb \dx $ in $ L^{2-\epsilon}(0,T) $. Up to a subsequence (not relabeled), $ \int_\Omega \tr \ln \widetilde{\bbb}^N(t) \dx \to \int_\Omega \tr \ln \bbb(t) \dx $ for a.e.~$ t \in (0,T) $. Then one infers from Lemma \ref{lem:General-LebesgueDT} and Remark \ref{rem:General-LebesgueDT-Lp} that for a.e.~$ t \in (0,T) $,\footnote{Adapt numbering and the argument, $L^\infty$ bound for the space integral.} 
\begin{equation}
    \label{eqs:convergence-trlnB-reg-L-infty}
	\int_0^T \bigg(\frac{1}{h} \int_{t-h}^t \Big(\int_\Omega \tr \ln \widetilde{\bbb}^N(\tau) \dx\Big) \dtau\bigg) \varsigma'(t) \dt \rightarrow \int_0^T \int_\Omega \tr \ln \bbb(t)  \dx \, \varsigma'(t) \dt.
\end{equation}
for any $ \varsigma \in W^{1,1}(0,T) $ with $ \varsigma \geq 0 $.
Therefore with \eqref{eqs:convergence-u-reg-a.e.}--\eqref{eqs:convergence-trlnB-reg-L-infty}, we have
\begin{equation*}
	\int_0^T \cE^N(t) \varsigma'(t) \dt \rightarrow \int_0^T \cE_\eta(t) \varsigma'(t) \dt,
\end{equation*}
with $ \cE_\eta(t) $ defined in \eqref{eqs:energy-reg}. In view of the lower semicontinuity of norms, the positivity of $m(\cdot)$, $\nu(\cdot)$ and $\mu(\cdot)$, and the almost everywhere convergence of $ \phi^N $ to $ \phi $, one has
\begin{equation*}
	\liminf_{N \rightarrow \infty} \int_0^T \cD^N(\tau) \varsigma(\tau) \dtau \geq \int_0^T \cD_\eta(\tau) \varsigma(\tau) \dtau,
\end{equation*}
for any $ \varsigma \in W^{1,1}(0,T) $ with $ \varsigma \geq 0 $, where, for a.e.~$t\in(0,T)$,
\begin{align*}
	\cD_\eta(t) \coloneqq \int_\Omega \Big( \frac{\nu(\phi)}{2} & \abs{\nabla \bu + \tran{\nabla} \bu}^2 
	+ m(\phi) \abs{\nabla q}^2 
	+ \eta \abs{\pt \phi}^2 \Big) \dx \\
	&  
	+ \int_\Omega \frac{\mu_\eta(\phi)}{2} \tr(\bbb
	+ \inv{\bbb} - 2\bbi) \dx
	+ \frac{\kappa}{d} \int_\Omega \abs{\nabla \tr \ln \bbb}^2 \dx.
\end{align*}
Passing to the limit in \eqref{eqs:energy-dissipation-reg-varsigma-eta}, as $ N \rightarrow \infty $, yields
\begin{equation*}
% 	\label{eqs:energy-dissipation-reg-eta}
	\cE_\eta(0) \varsigma(0)
	+ \int_0^T \cE_\eta(\tau) \varsigma'(\tau) \dtau
	\geq
	\int_0^T \cD_\eta(\tau) \varsigma(\tau) \dtau,
\end{equation*}
for any $ \varsigma \in W^{1,1}(0,T) $ with $ \varsigma \geq 0 $ and $\varsigma(T)=0$. On account of Lemma \ref{lem:energy-dissipation}, we obtain the desired energy dissipation inequality \eqref{eqs:energy-dissipation-reg-inequality}.

This completes the proof of Theorem \ref{thm:main-reg}. \qed

\section{Existence of Weak Solutions \texorpdfstring{($ \eta \rightarrow 0 $)}{}}
\label{sec:proof-of-weak-solution}
This section is devoted to prove Theorem \ref{thm:main}, by virtue of Theorem \ref{thm:main-reg} for the regularized system \eqref{eqs:Model_reg}, compactness discussions and limit passages. Due to {substantial mathematical reasons}, the final proof of Theorem \ref{thm:main} is restricted to the two {dimensional} case, as discussed in Sections \ref{sec:state-of-the-art}, \ref{sec:technical-discussions} and Remark \ref{rem:dimension-discussion}.
\subsection{Stronger Uniform Estimate}
\label{sec:a-priori-reg-strong}
{Using only the \textit{a priori} estimate \eqref{eqs:formalEstimate_reg}, we are not able to simply prove the existence of weak solutions by passing to the limit as $ \eta \rightarrow 0 $. Even with stress diffusion, the regularity of $\bbb$ is poor, as only $\normm{\bbb}_{L_t^\infty L_x^1}$ and $\normm{\tr \ln \bbb}_{L_t^\infty L_x^1 \cap L_t^2 W^{1,2}}$ are controlled so far, and that is of no help to obtain the necessary compactness for $\bbb$.} Thus, in this section derive a stronger estimate for {$\mathbb{B}$} {with $L^2$-initial data}, which was also carried out in, e.g., \cite{BB2011,BLS2017,GKT2022}. Note that the restriction of the problem to two dimensions arises precisely from the stronger estimate, even in presence of the stress diffusion term $ \frac{\kappa}{\mu(\phi)} \Delta \bbb $. Taking the Frobenius inner product of \eqref{eqs:fluid_reg-B} with $\bbb$, using the chain rule, integrating over $ \Omega $ and by parts, we have for a.e.~$t\in(0,T)$, that
% \dennis{estimate convective term directly}
% \yadong{I remember we discussed that it is the same with $\mathbb{B}$. Are there much difference?}
% \dennis{Not, not really. But it looks ``nicer''}
% \dennis{It looks nicer, yes, but it has some disadvantages. (we would need to adapt more)}
\begin{align*}
    & \ddt \onehalf \norm{\bbb}_{L^2}^2
    + \norm{\bbb}_{L^2}^2
    + \int_\Omega \frac{\kappa}{\mu_\eta(\phi)} \abs{\nabla \bbb}^2 \dx \\
    & = 
    \int_\Omega \tr\bbb \dx 
    + \kappa \int_\Omega \frac{\mu_\eta'(\phi)}{\mu_\eta^2(\phi)} (\nabla \sR_\eta \phi \cdot \nabla) \bbb : \bbb \dx
    + 2 \int_\Omega \bbb^2 : \nabla \sR_\eta \bu \dx
    \\
    &\quad
    {- \int_\Omega (\sR_\eta \bu \cdot\nabla) \bbb :  \bbb \dx.}
\end{align*}
Noting the upper-lower bounds of $\mu(\cdot), \mu'(\cdot)$, %integrating by parts over $\Omega$ in the convective term, 
employing H\"older's and Young's inequalities, the 2D Gagliardo--Nirenberg inequality and using Lemma \ref{lem:mollification}, we obtain
% \dennis{todo}
{
\begin{align*}
	& \ddt \onehalf \norm{\bbb}_{L^2}^2
	+ \norm{\bbb}_{L^2}^2
    + \frac{\kappa}{\overline\mu} \norm{\nabla \bbb}_{L^2}^2 \\
	& \leq \onehalf \norm{\bbb}_{L^2}^2
    + \abs{\Omega}
    + \frac{3\kappa}{4 \overline{\mu}} \norm{\nabla \bbb}_{L^2}^2
    + C \norm{\nabla\bu}_{L^2}^2
	+ C \Big(\norm{\nabla \phi}_{L^4}^4 + \norm{\nabla \bu}_{L^2}^2
    + \norm{\bu}_{L^4}^4 \Big) \norm{\bbb}_{L^2}^2.
\end{align*}}
Here, we used
{
\begin{align*}
    \abs{\int_\Omega \tr\bbb\dx } 
    &\leq \frac14 \norm{\tr\bbb}_{L^2}^2
    + \abs{\Omega}
    \\
    &\leq \onehalf \norm{\bbb}_{L^2}^2
    + \abs{\Omega},
    \\
    \abs{ \kappa \int_\Omega \frac{\mu_\eta'(\phi)}{\mu_\eta^2(\phi)}(\nabla \sR_\eta \phi \cdot \nabla) \bbb : \bbb \dx }
    & \leq C\norm{\nabla \bbb}_{L^2} \norm{\bbb}_{L^4} \norm{\nabla \sR_\eta \phi}_{L^4} 
    \\
    & \leq C \norm{\nabla \bbb}_{L^2}^{\frac{3}{2}} \norm{\bbb}_{L^2}^\onehalf \norm{\nabla \sR_\eta \phi}_{L^4}
    \\
    &\leq \frac{\kappa}{4 \overline{\mu}} \norm{\nabla \bbb}_{L^2}^2
    + C \norm{\bbb}_{L^2}^2 \norm{\nabla \phi}_{L^4}^4, 
    \\
    \abs{\int_\Omega \bbb^2 : \nabla \sR_\eta \bu \dx}
    & \leq \norm{\bbb}_{L^4}^2 \norm{\nabla \sR_\eta \bu}_{L^2} 
    \\
    & \leq \norm{\nabla \bbb}_{L^2} \norm{\bbb}_{L^2} \norm{\nabla \sR_\eta \bu}_{L^2}
    \\
    &\leq \frac{\kappa}{4 \overline{\mu}} \norm{\nabla \bbb}_{L^2}^2
    + C \norm{\bbb}_{L^2}^2 \norm{\nabla \bu}_{L^2}^2,
    \\
    \abs{\int_\Omega (\sR_\eta \bu \cdot\nabla) \bbb :  \bbb \dx}
    &\leq 
    \norm{\sR_\eta \bu}_{L^4} \norm{\nabla\bbb}_{L^2} \norm{\bbb}_{L^4}
    \\
    &\leq C \norm{\sR_\eta \bu}_{L^4} \norm{\nabla\bbb}_{L^2}^{3/2} \norm{\bbb}_{L^2}^{1/2}
    \\
    &\leq \frac{\kappa}{4 \overline{\mu}} \norm{\nabla \bbb}_{L^2}^2
    + C \norm{\bbb}_{L^2}^2 \norm{\bu}_{L^4}^4 .
\end{align*}}
{In summary, we have
\begin{align*}
    &\ddt \onehalf \norm{\bbb}_{L^2}^2
    + \onehalf \norm{\bbb}_{L^2}^2
    + \frac{\kappa}{4 \overline\mu} \norm{\nabla \bbb}_{L^2}^2 
    \\
    &\leq \abs{\Omega}
    + C \Big(\norm{\nabla \phi}_{L^4}^4 + \norm{\nabla \bu}_{L^2}^2 + \norm{\bu}_{L^4}^4 \Big) \norm{\bbb}_{L^2}^2,
\end{align*}
which, by integrating over $ (0,t) $ and using a Gronwall argument (cf.~\cite[Lemma 3.1]{GL2017}), implies
% \begin{subequations}
\begin{align}
    \label{eqs:formalEstimate_reg_strong} \nonumber
    &\norm{\bbb(t)}_{L^2}^2
    + \int_0^t \norm{\bbb(\tau)}_{L^2}^2 \d\tau
    + \int_0^t \norm{\nabla \bbb(\tau)}_{L^2}^2 \d\tau
    \\
    &\leq C \big( \norm{\bbb_0}_{L^2}^2 + t \abs{\Omega} \big)
    \big(1 + Ch(t)  \exp(C h(t)) \big),
\end{align}
% \begin{equation}
% 	\label{eqs:formalEstimate_reg_strong-1}
% 	\norm{\bbb(t)}_{L^2}^2
% 	\leq C(\norm{\bbb_0}_{L^2}^2 + 1) \exp(C h(t)),
% \end{equation}
% and
% \begin{equation}
% 	\label{eqs:formalEstimate_reg_strong-2}
% 	\int_0^t \norm{\bbb(\tau)}_{L^2}^2 \d\tau
% 	+ \int_0^t \norm{\nabla \bbb(\tau)}_{L^2}^2 \d\tau
% 	\leq C h(t)(\norm{\bbb_0}_{L^2}^2 + 1) \exp(C h(t)),
% \end{equation}
% \end{subequations}
for a.e.~$ t \in (0,T) $, where $ C > 0 $ does not depend on $ \eta > 0 $. Here, we defined
\begin{align*}
    h(t)\coloneqq 
    \normm{\nabla \phi}_{L_t^4 L_x^4}^4 
    + \normm{\nabla \bu}_{L_t^2 L_x^2}^2 
    + \normm{\bu}_{L_t^4 L_x^4}^4 ,
\end{align*}
%$ h(t)\coloneqq \normm{\nabla \phi}_{L_t^4 L_x^4}^4 + \normm{\nabla \bu}_{L_t^2 L_x^2}^2$, 
which is bounded independently of $\eta>0$ due to \eqref{eqs:formalEstimate_reg} and the embedding \eqref{eqs:interpolation-L4}.
}

\subsection{Proof of Theorem \ref{thm:main}}
\label{sec:proof-weak-limiting}
Let us denote $ (\bu^\eta, \bbb^\eta, \phi^\eta, q^\eta) $ the corresponding regularized solution of \eqref{eqs:Model_reg}, where $\eta>0$. Note that the \textit{a priori} estimate done in Section \ref{sec:a-priori-reg} and \ref{sec:a-priori-reg-strong} are uniform in terms of $ \eta $. We conclude from \eqref{eqs:formalEstimate_reg} and \eqref{eqs:formalEstimate_reg_strong} that
\begin{align}
	& \cE_\eta(\tau)
%	\norm{\sqrt{\rho(\phi^\eta)} \bu^\eta (\tau)}_{L^2}^2
%	+ \norm{\nabla \phi^\eta(\tau)}_{L^2}^2
%	+ \norm{W(\phi^\eta)(\tau)}_{L^1}
%	+ \norm{\tr(\bbb^\eta - \ln \bbb^\eta - \bbi)(\tau)}_{L^1} 
	+ \norm{\bbb^\eta(\tau)}_{L^2}^2 
	+ \int_0^\tau \norm{\nabla \bu^\eta(t)}_{L^2}^2 \dt
	+ \int_0^\tau \norm{\nabla q^\eta(t)}_{L^2}^2 \dt 
	+ \int_0^\tau \norm{\bbb^\eta(t)}_{L^2}^2 \dt
	\nonumber \\
	& 
	+ \int_0^\tau \norm{\nabla \bbb^\eta(t)}_{L^2}^2 \dt
	+ \int_0^\tau \norm{\tr(\bbb^\eta + \inv{(\bbb^\eta)} - 2\bbi)(t)}_{L^1} \dt
	+ \int_0^\tau \norm{\nabla \tr \ln \bbb^\eta(t)}_{L^2}^2 \dt
        \nonumber
	\\
	&
	\leq C (\cE_\eta^0, \norm{\bbb_0}_{L^2}^2)
	\leq C (\cE_0, \norm{\bbb_0}_{L^2}^2), 
	\label{eqs:uniform-estimate}
\end{align}
for any $ \tau \in (0,T) $, where $ C > 0 $ is uniform in terms of $ \eta $, which together with the fact that $ \cE(0), \normm{\bbb_0}_{L^2}^2 $ are bounded, implies the following uniform bounds (in $ \eta $)
\begin{alignat*}{3}
	\bu^\eta & \quad  \text{ is bounded in } \quad && L^2(0,T; W_0^{1,2}(\Omega; \bbr^2)) \text{ and } L^\infty(0,T; L_\sigma^2(\Omega)), \\
	\nabla q^\eta & \quad \text{ is bounded in } \quad && L^2(0,T; L^2(\Omega; \bbr^2)), \\
	\phi^\eta & \quad \text{ is bounded in } \quad && L^\infty(0,T; W^{1,2}(\Omega)), \\
	\bbb^\eta & \quad \text{ is bounded in } \quad && L^2(0,T; W^{1,2}(\Omega; \bbr_{\mathrm{sym}}^{2 \times 2})) \text{ and } L^\infty(0,T; L^2(\Omega; \bbr_{\mathrm{sym}}^{2 \times 2})),
\end{alignat*}
and
\begin{equation*}
	\int_0^T \abs{\int_\Omega q^\eta \dx} \dt \leq  M(T),
\end{equation*}
for a certain monotone function $ M: \bbr^+ \rightarrow \bbr^+ $. 
Moreover, testing \eqref{eqs:weak-reg-q-formulation} with $ \Delta \phi^\eta $ together with $W(r)=W_0(r) - \frac{\omega}{2} r^2$ as
%\footnote{also with the extension $\widetilde W$, or we just keep it like this?\\Here we don't need the extension right? \\ No, we are in the physical range.} 
in Section \ref{sec:a-priori-reg}, integration by parts and Young's inequality yields%\footnote{mention definition of $W_0$?}
\begin{align*}
	& \int_0^T \int_\Omega \big( \abs{\Delta \phi^\eta}^2 
	+ W_0''(\phi^\eta) \abs{\nabla \phi^\eta}^2 \big) \dxdt 
		\\
	& = \int_0^T \int_\Omega \Big( \nabla q^\eta \cdot \nabla \phi^\eta 
		+ \omega \abs{\nabla \phi^\eta}^2
		- \frac{\mu'}{2} \nabla \tr(\bbb^\eta - \ln \bbb^\eta - \bbi) \cdot \nabla \phi^\eta \Big) \dxdt \\
	& \leq \int_0^T \norm{\nabla q^\eta}_{L^2}^2 \dt
		+ C(\omega, \overline{\mu}') 
		+ \int_0^T \big( \norm{\nabla \bbb^\eta}_{L^2}^2 + \norm{\nabla \tr \ln \bbb^\eta}_{L^2}^2 \big) \dt
		\leq C,
\end{align*}
where $ C > 0 $ is uniform in $ \eta $ due to $ W_0'' > 0 $ and \eqref{eqs:uniform-estimate}. Therefore, one obtains further
\begin{alignat*}{3}
	\phi^\eta & \quad \text{ is bounded in } \quad && L^2(0,T; W^{2,2}(\Omega)).
\end{alignat*}

Next, we gain more bounds for the purpose of compactness. By the Sobolev embedding in two dimensions and the interpolation embedding \eqref{eqs:interpolation-L4} we know
\begin{alignat*}{3}
	\bu^\eta & \quad \text{ is bounded in } \quad && L^4(0,T; L^4(\Omega; \bbr^2)), \\
	\bbb^\eta & \quad \text{ is bounded in } \quad && L^4(0,T; L^4(\Omega; \bbr_{\mathrm{sym}}^{2 \times 2})),
\end{alignat*}
%\begin{alignat*}{4}
%	\bbb^\eta & \quad \text{ is bounded in } \quad && L^{\frac{2}{\theta}}(0,T; L^{\frac{2p}{(2 - p)\theta + p}}(\Omega; \bbr_{\mathrm{sym}}^{2 \times 2})), && \quad 1 \leq p < \infty, 0 < \theta < 1,
%\end{alignat*}
which together with bounds above leads to
\begin{alignat*}{3}
	\sR_\eta \bu^\eta \otimes \bbb^\eta & \quad \text{ is bounded in } \quad && L^2(0,T; L^2(\Omega; \bbr^{2 \times 2 \times 2})),  \\
	\nabla \sR_\eta \bu^\eta \bbb^\eta & \quad \text{ is bounded in } \quad && L^{\frac{4}{3}}(0,T; L^{\frac{4}{3}}(\Omega; \bbr^{2 \times 2})).
\end{alignat*}
%\begin{alignat*}{4}
%	\sR_\eta \bu^\eta \otimes \bbb^\eta & \quad \text{ is bounded in } \quad && L^2(0,T; L^{\frac{2p}{p + 2}}(\Omega; \bbr^{2 \times 2 \times 2})), && \quad 1 \leq p < \infty, \\
%	\nabla \sR_\eta \bu^\eta \bbb^\eta & \quad \text{ is bounded in } \quad && L^{\frac{2}{1 + \theta}}(0,T; L^{\frac{2p}{(2 - p)\theta + 2p}}(\Omega; \bbr_{\mathrm{sym}}^{2 \times 2})), && \quad 1 \leq p < \infty, 0 < \theta < 1.
%\end{alignat*}
By making use of the weak formulation \eqref{eqs:weak-reg-B-formulation}, one ends up with
%\footnote{I just checked it. I think now the regularity space is okay. I did something wrong during the talk..}
%\footnote{This exponent is not correct. Here $[L^4(W^{1,2})]' \widetilde = L^{4/3}([W^{1,2}]')$ works fine, but it can also be improved. But this does not change the convergence at all.\\Not it should be okay? here I change $p$ to greater than 2 since $2p/(p+2) > 1$.\\ For the term $\nabla\sR_\eta u^\eta \bbb^\eta$ you are right, but I think due to $\abs{<\partial_t\bbb^\eta, \bbc>} \leq ... + \int_{..}\nabla\bbb^\eta:\nabla (\bbc/\mu_\eta(\phi^\eta))$ we need to check this..\\or we just make things easier...\\ by writing $\partial_t\bbb \in L^4([W^{1,2}]')$?\\yes somehow similar. Agreed.:)}
\begin{alignat*}{3}
	\pt \bbb^\eta & \quad \text{ is bounded in } \quad && L^{\frac43}(0,T; [W^{1,2}(\Omega; \bbr_{\mathrm{sym}}^{2 \times 2})]'),
\end{alignat*}
Because of \eqref{eqs:weak-reg-phi-formulation} and the boundedness of $ \bu^\eta $, $ \nabla q^\eta $, we have
\begin{alignat*}{3}
	\pt \phi^\eta & \quad \text{ is bounded in } \quad && L^2(0,T; [W^{1,2}(\Omega)]'),
\end{alignat*}
Then up to a subsequence ($ \eta_k \rightarrow 0 $ as $ k \rightarrow \infty $) still denoted by the superscript $ \eta $, one obtains
\begin{alignat*}{4}
	\bu^\eta & \rightarrow \bu, \quad && \text{weakly} \quad && \text{in } L^2(0,T; W_0^{1,2}(\Omega; \bbr^2)), \\
	\bu^\eta & \rightarrow \bu, \quad && \text{weakly-}* \quad && \text{in } L^\infty(0,T; L_\sigma^2(\Omega)) \cong [L^1(0,T; L_\sigma^2(\Omega))]', \\
	\phi^\eta & \rightarrow \phi, \quad && \text{weakly} \quad && \text{in } L^2(0,T; W^{2,2}(\Omega)), \\
	\phi^\eta & \rightarrow \phi, \quad && \text{weakly-}* \quad && \text{in } L^\infty(0,T; W^{1,2}(\Omega)) \cong { [L^1(0,T; (W^{1,2}(\Omega))' )]', } \\
	\pt \phi^\eta & \rightarrow \pt \phi, \quad && \text{weakly} \quad && \text{in } L^2(0,T; [W^{1,2}(\Omega)]') , \\
	q^\eta & \rightarrow q, \quad && \text{weakly} \quad && \text{in } L^2(0,T; W^{1,2}(\Omega)), \\
	\nabla q^\eta & \rightarrow \nabla q, \quad && \text{weakly} \quad && \text{in } L^2(0,T; L^2(\Omega; \bbr^2)), \\
	\bbb^\eta & \rightarrow \bbb, \quad && \text{weakly} \quad && \text{in } L^2(0,T; W^{1,2}(\Omega; \bbr_{\mathrm{sym}}^{2 \times 2})), \\
	\bbb^\eta & \rightarrow \bbb, \quad && \text{weakly-}* \quad && \text{in } L^\infty(0,T; L^2(\Omega; \bbr_{\mathrm{sym}}^{2 \times 2})) \cong [L^1(0,T; L^2(\Omega; \bbr_{\mathrm{sym}}^{2 \times 2}))]', \\
	\pt \bbb^\eta & \rightarrow \pt \bbb, \quad && \text{weakly} \quad && \text{in } L^{\frac43}(0,T; [W^{1,2}(\Omega; \bbr_{\mathrm{sym}}^{2 \times 2})]').
\end{alignat*}
%\footnote{here also, adapt the exponents for $\partial_t\bbb$\\I think this should be consistant with the boundedness above. Maybe a typo...after many versions of modifications.}
In view of the Aubin--Lions lemma, one concludes the strong convergence (up to a non-relabeled subsequence)
\begin{alignat*}{4}
	\phi^\eta & \rightarrow \phi, \quad && \text{strongly} \quad && \text{in } L^2(0,T; W^{1,p}(\Omega)), \quad \forall\, 1 \leq p < \infty, \\
	\phi^\eta & \rightarrow \phi, \quad && \text{a.e.} \quad && \text{in } Q_T, \\
	\bbb^\eta & \rightarrow \bbb, \quad && \text{strongly} \quad && \text{in } L^2(0,T; L^p(\Omega; \bbr_{\mathrm{sym}}^{2 \times 2})), \quad \forall\, 2 \leq p < \infty, \\
	\bbb^\eta & \rightarrow \bbb, \quad && \text{a.e.} \quad && \text{in } Q_T.
\end{alignat*}
Arguing in a similar fashion as in Section \ref{sec:B-positive-delta} leads us to
\begin{align*}
	& \bbb \text{ is positive definite a.e.~in } Q_T, \\
	& \tr \ln \bbb^\eta \rightarrow \tr \ln \bbb, \quad \text{weakly} \quad \text{in } L^2(0,T; W^{1,2}(\Omega)),
\end{align*}
Then it follows from the continuity of $ W'(\cdot) $ and $ \tr \ln (\cdot) $ that
\begin{alignat*}{4}
	W'(\phi^\eta) & \rightarrow W'(\phi), \quad && \text{a.e.} \quad && \text{in } Q_T, \\
	\tr \ln \bbb^\eta & \rightarrow \tr \ln \bbb, \quad && \text{a.e.} \quad && \text{in } Q_T.
\end{alignat*}
Again by the uniform boundedness of $\tr \ln \bbb^\eta$ in $L^2(Q_T)$, one concludes the strong convergence of $\tr \ln \bbb^\eta \to \tr \ln \bbb$ in $L^{2-\epsilon}(Q_T)$ for $0< \epsilon < 1$ by Vitali's convergence theorem. Then in view of Lemma \ref{lem:mollification} and strong convergence of $\phi^\eta$, we have
\begin{equation}
    \sR_\eta \Big[\frac{\mu_\eta'(\phi)}{2} \tr(\bbb - \ln \bbb - \bbi)\Big]
    \to \frac{\mu'(\phi)}{2} \tr(\bbb - \ln \bbb - \bbi), \ \text{ in } L^1(Q_T),
\end{equation}
Consequently, up to a non-relabeled subsequence, we end up with 
% \footnote{This is not true, because for mollifier, we don't have pointwse-a.e.~convergence property. Instead, we should still do the same argument as in Section \ref{sec:proof-reg} (by means of weak convergence and maximal monotone operator). So copy paste?\\
% Very similar I would say. Maybe one sentence or two are enough. \\ Ok.}
\begin{equation*}
	q^\eta \rightarrow q = W'(\phi) - \Delta \phi + \frac{\mu'(\phi)}{2} \tr(\bbb - \ln \bbb - \bbi), \quad \text{a.e.~in } Q_T,
\end{equation*}
% by the same argument in Section \ref{sec:proof-reg} and the weak formulations \eqref{eqs:weak-reg-phi-formulation}, \eqref{eqs:weak-reg-B-formulation} converge to \eqref{eqs:weak-full-phi-formulation}, \eqref{eqs:weak-full-B-formulation} respectively 
as $ \eta \rightarrow 0 $. %\footnote{I reformulate the convergence here. As commented above, we don't have the almost everywhere convergence of mollification, so we use the weak convergence and monotone operator. But I think maybe we can argue like this by using the strong convergence. So the monotone operator was only used for regularized system.\\ Ok.}

Now, we are in the position to get the compactness of $ \bu^\eta $, by addressing the problem caused by the variable density $ \rho(\phi^\eta) $ with the Helmholtz projection $ \bbp_\sigma $. With the boundedness of $ \bu^\eta $, $ \phi^\eta $, $ q^\eta $ and $ \bbb^\eta $ in two dimensions, one infers
\begin{alignat*}{3}
	\rho(\phi^\eta) \bu^\eta \otimes \bu^\eta & \  \text{ is bounded in } \  && L^2(0,T; L^2(\Omega; \bbr^{2 \times 2})), \\
	\nabla \bu^\eta + \tran{\nabla} \bu^\eta & \  \text{ is bounded in } \  && L^2(0,T; L^2(\Omega; \bbr^{2 \times 2})), \\
	\bu^\eta \otimes \nabla q^\eta & \  \text{ is bounded in } \  && L^{\frac{4}{3}}(0,T; L^{\frac{4}{3}}(\Omega; \bbr^{2 \times 2})), \\
	q^\eta \nabla \phi^\eta & \  \text{ is bounded in } \  && L^2(0,T; L^{\frac{2p}{2+p}}(\Omega; \bbr^2)), \ \forall\,  2< p < \infty, \\
	\sR_\eta[\mu_\eta(\phi^\eta)(\bbb^\eta - \bbi)] & \  \text{ is bounded in } \  && L^2(0,T; L^p(\Omega; \bbr_{\mathrm{sym}}^{2 \times 2})), \ \forall\,  2< p < \infty, \\
	\sR_\eta\Big[\frac{\mu_\eta(\phi^\eta)}{2} \nabla \tr(\bbb^\eta - \ln \bbb^\eta - \bbi)\Big] & \  \text{ is bounded in } \  && L^2(0,T; L^2(\Omega; \bbr^2)),
\end{alignat*}
which means the all terms above are bounded in $ L^{\frac{4}{3}}(0,T; L^{\frac{4}{3}}(\Omega)) $ (without specifying the dimensions). Then in \eqref{eqs:weak-reg-u-formulation}, the test function can be restricted to 
\begin{equation*}
	\bw, \nabla \bw \in [L^{\frac{4}{3}}(0,T; L^{\frac{4}{3}}(\Omega))]' = L^4(0,T; L^4(\Omega)).
\end{equation*}
Note that $ \bw $ lies in the solenoidal space. Hence, with the help of the Leray projection $ \bbp_\sigma $, we conclude that
\begin{alignat*}{3}
	\pt\big(\bbp_\sigma(\rho(\phi^\eta) \bu^\eta)\big) & \  \text{ is bounded in } \  && [L^4(0,T; W_{0,\sigma}^{1,4}(\Omega; \bbr^2))]' = L^\frac{4}{3}(0,T; W_\sigma^{-1,4}(\Omega; \bbr^2)), \\
	\bbp_\sigma(\rho(\phi^\eta) \bu^\eta) & \  \text{ is bounded in } \  && L^2(0,T; W^{1,2}(\Omega; \bbr^2)),
\end{alignat*}
which together with the Aubin--Lions lemma implies the strong convergence
\begin{alignat}{4}
    \label{eqs:strong-convergence_Prhou}
	\bbp_\sigma(\rho(\phi^\eta) \bu^\eta) & \rightarrow \overline{\bbp_\sigma(\rho(\phi) \bu)}, \quad && \text{strongly} \quad && \text{in } L^2(0,T; L_\sigma^2(\Omega; \bbr^2)),
\end{alignat}
for some function $ \overline{\bbp_\sigma(\rho(\phi) \bu)} \in L^\infty(0,T; L_\sigma^2(\Omega; \bbr^2)) $. Analogously to Section \ref{sec:proof-reg}, we identify $ \overline{\bbp_\sigma(\rho(\phi) \bu)} = \bbp_\sigma(\rho(\phi) \bu) $. Indeed, as $ \rho(\phi^\eta) \bu^\eta \rightarrow \rho(\phi) \bu \text{ weakly} \text{ in } L^2(0,T; L^2(\Omega; \bbr^2)) $, and by virtue of the weak continuity of the Leray projection $ \bbp_\sigma: L^2(0,T; L^2(\Omega; \bbr^2)) \rightarrow L^2(0,T; L_\sigma^2(\Omega)) $, we obtain
\begin{equation*}
	\overline{\bbp_\sigma(\rho(\phi) \bu)} = \bbp_\sigma(\rho(\phi) \bu).
\end{equation*}
Once again, we prove the strong convergence of $ \bu^\eta $ to $ \bu $ through the convergence of the kinetic energy. Namely,
\begin{align*}
	\int_0^T \int_\Omega \rho(\phi^\eta) \abs{\bu^\eta}^2 \dxdt
	& = \int_0^T \int_\Omega \bbp_\sigma(\rho(\phi^\eta) \bu^\eta) \cdot \bu^\eta \dxdt \\
	& \rightarrow
	\int_0^T \int_\Omega \bbp_\sigma(\rho(\phi) \bu) \cdot \bu \dxdt
	= \int_0^T \int_\Omega \rho(\phi) \abs{\bu}^2 \dxdt,
\end{align*}
from which we have
\begin{alignat*}{4}
	\sqrt{\rho(\phi^\eta)} \bu^\eta & \rightarrow \sqrt{\rho(\phi)} \bu, \quad && \text{strongly} \quad && \text{in } L^2(0,T; L^2(\Omega; \bbr^2)).
\end{alignat*}
Because $ \rho(r) $ is affine regarding $ r $, one gets
\begin{alignat*}{4}
	\rho(\phi^\eta) & \rightarrow \rho(\phi), \quad && \text{a.e.} \quad && \text{in } Q_T
	\quad
	\text{and}
	\quad
	\abs{\rho(\phi^\eta)} \geq C > 0.
\end{alignat*}
Then it holds that
\begin{alignat*}{4}
	\bu^\eta = \frac{1}{\sqrt{\rho(\phi^\eta)}} \sqrt{\rho(\phi^\eta)} \bu^\eta & \rightarrow \frac{1}{\sqrt{\rho(\phi)}} \sqrt{\rho(\phi)} \bu = \bu, \quad  && \text{strongly} \quad && \text{in } L^2(0,T; L^2(\Omega; \bbr^2)),
\end{alignat*}
and hence
\begin{alignat*}{4}
	\bu^\eta & \rightarrow \bu, \quad && \text{a.e.} \quad && \text{in } Q_T.
\end{alignat*}

With all the compactness above, one can pass to the limit in \eqref{eqs:weak-reg-u-formulation} to \eqref{eqs:weak-full-u-formulation} as $ \eta \rightarrow 0 $, combining with
\begin{alignat*}{4}
	m(\phi^\eta) & \rightarrow m(\phi), \quad && \text{a.e.} \quad && \text{in } Q_T, \\
	\nu(\phi^\eta) & \rightarrow \nu(\phi), \quad && \text{a.e.} \quad && \text{in } Q_T, \\
	\mu(\phi^\eta) & \rightarrow \mu(\phi), \quad && \text{a.e.} \quad && \text{in } Q_T,
\end{alignat*}
which can be deduced by means of the continuity of $ m(\cdot), \nu(\cdot), \mu(\cdot) $ and the almost everywhere convergence of $ \phi^\eta $. Similarly, we use integration by parts and the compactness of $ \phi^\eta $ to handle the term $ \int_0^t \int_\Omega q^\eta \nabla \phi^\eta \cdot \bw \dxdt $.% as in \cite[Section 5.1]{ADG2013}.

In the final step of the proof, we derive the energy dissipation inequality \eqref{eqs:energy-dissipation-inequality}. To this end, multiplying the differential inequality \eqref{eqs:dt-formal_reg} by a function $ \varsigma \in W^{1,1}(0,T) $ with $ \varsigma \geq 0 $, $ \varsigma(T) = 0 $, and integrating the resulting inequality over $ (0,T) $ and by parts with respect to the time variable, one obtains
\begin{equation}
	\label{eqs:energy-dissipation-varsigma-eta}
	\cE_\eta(0) \varsigma(0)
	+ \int_0^T \cE_\eta(\tau) \varsigma(\tau)' \dtau
	\geq
	\int_0^T \cD_\eta(\tau) \varsigma(\tau) \dtau,
\end{equation}
where $ \cE_\eta(t) $ is given in \eqref{eqs:energy-reg} and, for a.e.~$t\in(0,T)$,
\begin{align*}
	\cD_\eta(t) \coloneqq \int_\Omega \Big( \frac{\nu(\phi^\eta)}{2} & \abs{\nabla \bu^\eta + \tran{\nabla} \bu^\eta}^2 
	+ m(\phi^\eta) \abs{\nabla q^\eta}^2 \Big) \dx \\
	&  
	+ \int_\Omega \frac{\mu(\phi^\eta)}{2} \tr(\bbb^\eta
	+ \inv{(\bbb^\eta)} - 2\bbi) \dx
	+ \frac{\kappa}{2} \int_\Omega \abs{\nabla \tr \ln \bbb^\eta}^2 \dx.
\end{align*}
Thanks to the compactness of $ \bu^\eta $, $ \phi^\eta $, we deduce that up to a subsequence (not relabeled),
\begin{gather*}
	\begin{alignedat}{3}
	    \bu^\eta(t) & \rightarrow \bu(t), \quad && \text{in } L_\sigma^2(\Omega), \\
	    \phi^\eta(t) & \rightarrow \phi(t), \quad && \text{in } C(\overline{\Omega}), \\
	    \bbb^\eta(t) & \rightarrow \bbb(t), \quad && \text{in } L^2(\Omega; \bbr_{\mathrm{sym}}^{2 \times 2}),
	\end{alignedat}\\
    \int_\Omega \tr \ln \bbb^\eta(t) \dx \dtau \rightarrow \int_\Omega \tr \ln \bbb(t) \dx,
\end{gather*}
for a.e.~$ t \in (0,T) $, where the last statement holds true by Vitali's convergence theorem in view of the pointwise convergence of $ \tr \ln \bbb^\eta $ and the uniform boundedness in $L^2(Q_T)$. Therefore, it comes up with
\begin{alignat*}{3}
	\cE_\eta(t) & \rightarrow \cE(t), \quad && \text{a.e.~in } (0,T),
\end{alignat*}
where $ \cE(t) $ is defined in \eqref{eqs:energy}. In view of the lower semicontinuity of norms
%\footnote{here also, nonneg. of $m,\nu,\mu$? Also mention $\liminf_{\eta\to 0} \int_0^T \eta \norm{\partial_t\phi^\eta(\tau)}_{L^2}^2 \varsigma(\tau) \d\tau \geq 0$? Or just leave it like it is?\\I add one sentence below.} 
and the almost everywhere convergence of $ \phi^\eta $ to $ \phi $, one has
\begin{equation*}
	\liminf_{\eta \rightarrow 0} \int_0^T \cD_\eta(\tau) \varsigma(\tau) \dtau \geq \int_0^T \cD(\tau) \varsigma(\tau) \dtau,
\end{equation*}
for any $ \varsigma \in W^{1,1}(0,T) $ with $ \varsigma \geq 0 $, where
\begin{align*}
	\cD(t) \coloneqq \int_\Omega \Big( \frac{\nu(\phi)}{2} & \abs{\nabla \bu + \tran{\nabla} \bu}^2 
	+ m(\phi) \abs{\nabla q}^2 \Big) \dx \\
	&  
	+ \int_\Omega \frac{\mu(\phi)}{2} \tr(\bbb
	+ \inv{\bbb} - 2\bbi) \dx
	+ \frac{\kappa}{2} \int_\Omega \abs{\nabla \tr \ln \bbb}^2 \dx.
\end{align*}
Note that here we employed the positivity of $m(\cdot)$, $\nu(\cdot)$ and $\mu(\cdot)$, and
\begin{equation*}
    \liminf_{\eta\to 0} \int_0^T \eta \norm{\partial_t\phi^\eta(\tau)}_{L^2}^2 \varsigma(\tau) \d\tau \geq 0.
\end{equation*}
Passing to the limit in \eqref{eqs:energy-dissipation-varsigma-eta} as $ \eta \rightarrow 0 $ yields
\begin{equation*}
% 	\label{eqs:energy-dissipation-eta}
	\cE(0) \varsigma(0)
	+ \int_0^T \cE(\tau) \varsigma'(\tau) \dtau
	\geq
	\int_0^T \cD(\tau) \varsigma(\tau) \dtau,
\end{equation*}
for any $ \varsigma \in W^{1,1}(0,T) $ with $ \varsigma \geq 0 $ and $\varsigma(T)=0$. By virtue of Lemma \ref{lem:energy-dissipation}, we get the desired energy dissipation inequality \eqref{eqs:energy-dissipation-inequality}. The additional stronger estimate of $ \bbb $ can be obtained directly from \eqref{eqs:formalEstimate_reg_strong} together with \eqref{eqs:energy-dissipation-inequality}.

This finishes the proof.
\qed

\begin{remark}
    \label{rem:generalcases}
	The case of a general free energy
	\begin{equation*}
		\int_\Omega \Big( \frac{a(\phi)}{2} \abs{\nabla \phi}^2 + W(\phi) \Big) \dx
	\end{equation*}
	with some positive coefficient $ a(\phi) $ %and the case of degenerate mobility $ m(\phi) $ that may attain zero 
 can be achieved by our method with slight modifications. Note that the two-phase incompressible flow with different densities and general free energy was already addressed in \cite{ADG2013}.
 % and the degenerate mobility case was handled in \cite{ADG2013degenerate}
 In our framework, one only needs to adopt a more complicated subgradient with respect to a reparametrized potential as in \cite{ADG2013} 
 % and employ an approximation of the degenerate mobility by a certain strictly positive mobility (see \cite{ADG2013degenerate}) 
 during the proof of the existence of solutions to the regularized problem in Section \ref{sec:regularized}.
\end{remark}

\appendix
\section{Derivation of the Model: Local Dissipation Laws}
\label{sec:derivation}
In this section, we provide the main arguments for a thermodynamically consistent derivation of the model \eqref{eqs:Model}. The general idea is to start from physical balance laws in a closed, isothermal system. 
After that, we state phenomenological assumptions such that the system fulfills the second law of thermodynamics. At that point, there are some general frameworks that can be used for the constitutive assumptions, such as the \textit{Local Dissipation Inequality and
the Lagrange Multiplier Approach} developed by Liu \cite{Liu1972}, the \textit{Onsager's variational principle} \cite{Onsager1932} or the \textit{General Equation for Non-Equilibrium Reversible-Irreversible Coupling} (GENERIC) framework of Gmerla--\"Ottinger \cite{GO1997a, GO1997b}. We also refer to \cite{GKL2018, MP2018}. In our case, we will follow \cite{AGG2012} and use the \textit{Local Dissipation Inequality and
the Lagrange Multiplier Approach}.

% Mass balance equations in local form 
% \begin{align*}
%     \partial_t \rho_i + \Div (\rho_i \bv_i) = 0
% \end{align*}
% Define volume fractions $u_i = \rho_i / \overline{\rho}_i$ where $\overline{\rho}_i$ is the specific constant density of the unmixed fluid $i$. Due to the assumption that the excess volume is zero, one has $u_1 + u_2 = 1$. Define the difference of volume fractions $\phi = u_2 - u_1$ and the volume-averaged velocity $\bv = u_1 \bv_1 + u_2 \bv_2$. One can also define the mass fluxes $\pmb{J}_i = \rho_i \bv_i - \rho_i \bv$ of component $i$ relative to the velocity $\bv$, which then gives
% \begin{align*}
%     \partial_t \rho_i + \Div(\rho_i \bv) + \Div \pmb{J}_i = 0.
% \end{align*}
Following \cite{AGG2012} one can derive the following balance law of mass in local form%\footnote{Numbering: remove or use it.}
\begin{align}
\label{eq:appendix-phi}
    \partial_t \phi + \Div(\phi\bu) + \Div \bJ_\phi = 0,
\end{align}
where $\phi$ is the difference of volume fractions, $\bJ_\phi$ is a diffusive flux and $\bu$ is the volume-averaged velocity satisfying the incompressibility condition 
\begin{align}
\label{eq:appendix-div}
    \Div \bu = 0.
\end{align}
The balance law of linear momentum reads
\begin{align}
\label{eq:appendix-v}
    \partial_t (\rho \bu) + \Div(\rho \bu \otimes \bu) = \Div \tilde \bT,
    % (\partial_t \rho + \bu\cdot\nabla \rho) \bu
    % + \rho (\partial_t \bu + (\bu\cdot\nabla)\bu)
    % = \Div \tilde \bT,
\end{align}
% The total mass balance and the momentum balance read
% \begin{align*}
%     \partial_t \rho + \Div(\rho\bu + \tilde\bJ) &= 0,
%     \\
%     \rho \partial_t \bu + ((\rho\bu + \tilde\bJ) \cdot \nabla) \bu &= \Div \tilde\bT,
% \end{align*}
where $\rho = \hat\rho(\phi) =\tfrac12 \overline\rho_2 (1+\phi) + \tfrac12 \overline\rho_1 (1-\phi)$ is the variable mass density of the mixture %, $\tilde\bJ$ is a relative mass flux 
and $\tilde\bT$ is the full stress tensor of the system, which is symmetric due to the balance law of angular momentum. Here, $\overline\rho_1, \overline\rho_2$ denote the constant mass densities of the pure phases of the mixture. For simplicity, we assume that no external forces are present.
We consider the deformation gradient $\mathbb{F}$ of the mixture, given by
\begin{align}
\label{eq:appendix-F}
    \partial_t \mathbb{F} + \bu\cdot\nabla \mathbb{F} = \nabla\bu \mathbb{F}.
\end{align}
For the derivation for the Oldroyd-B model with relaxation, we assume a virtual multiplicative decomposition of the deformation gradient into one part capturing the irreversible, dissipative processes and another part for the total elastic response of the material, i.e.,
\begin{align}
\label{eq:appendix-decomposition}
    \mathbb{F} = \mathbb{F}_e \mathbb{F}_d,
\end{align}
also see \cite{MP2018}. Then, the left Cauchy--Green tensor associated with the elastic part of the total mechanical response $\bbb \coloneqq \mathbb{F}_e \tran{\mathbb{F}_e}$ is the sought quantity for our model.
Introducing the tensorial quantity $\mathbb{L}_d \coloneqq (\partial_t \mathbb{F}_d + \bu\cdot\nabla \mathbb{F}_d) \mathbb{F}_d^{-1}$, one can obtain with a simple calculation, using \eqref{eq:appendix-F} and \eqref{eq:appendix-decomposition},
\begin{align*}
    \partial_t \mathbb{F}_e + \bu\cdot\nabla \mathbb{F}_e = \nabla \bu \mathbb{F}_e - \mathbb{F}_e \mathbb{L}_d,
\end{align*}
which gives for $\bbb = \mathbb{F}_e \tran{\mathbb{F}}_e$
\begin{align}
\label{eq:appendix-B}
    \partial_t \bbb + \bu\cdot\nabla \bbb = \nabla \bu \bbb + \bbb \tran{\nabla}\bu - \mathbb{F}_e (\mathbb{L}_d + \tran{\mathbb{L}_d}) \tran{\mathbb{F}_e}.
\end{align}
Later, the dependence on $\mathbb{D}_d \coloneqq \frac12 (\mathbb{L}_d + \tran{\mathbb{L}_d})$ will be removed with a constitutive relation.

% The unknowns are the state variables $\phi, \bu, \bbb$ (or $\mathbb{F}_e$, respectively), such as the flux $\bJ$, the stress tensor $\tilde\bT$ and $\mathbb{D}_d$.
% As there are more unknowns than balance laws, some of the unknown quantities have to be specified as the system is not complete.

Assuming a general energy density of the form 
\begin{align*}
    e = \hat e(\phi, \nabla\phi, \bbb) + \frac12 \rho(\phi) \abs{\bu}^2,
\end{align*}
composed of a (general) free energy density and the kinetic energy density of the system, the second law of thermodynamics for a closed physical system in the isothermal case gives
\begin{align*}
    \ddt \int_{V(t)} e(\phi,\nabla\phi, \bu, \bbb) \dx
    \leq - \int_{\partial V(t)} \bJ_e \cdot \bn \,\d\mathcal{H}^{d-1} 
    + \int_{\partial V(t)} (\tilde\bT \bn) \cdot \bu \,\d\mathcal{H}^{d-1},
\end{align*}
where $V(t) \subset \Omega$ is an arbitrary test volume transported by the fluid velocity, $\bn$ is the outer unit normal to $\partial V(t)$ and $\bJ_e$ is a dissipative energy flux yet to be determined. Roughly speaking, the change of the total energy in a test volume $V(t)$ cannot exceed the 
change of energy due to diffusion and the working due to macroscopic stresses. With the help of the divergence theorem and Reynold's transport theorem, and as the test volume $V(t)\subset\Omega$ is arbitrary, one can obtain a local inequality for the dissipation by
\begin{align*}
    - \mathcal{D} \coloneqq 
    (\partial_t^\bullet e + e \Div \bu + \Div \bJ_e) 
    - \Div (\tilde\bT) \cdot \bu
    - \tilde\bT : \nabla\bu 
    - \lambda_\phi( \partial_t^\bullet \phi + \phi \Div\bu + \Div \bJ_\phi ) \leq 0,
\end{align*}
where $\partial_t^\bullet\coloneqq \partial_t + (\bu\cdot\nabla)$ denotes the material derivative and $\lambda_\phi$ is a Lagrange multiplier for the balance law of mass \eqref{eq:appendix-phi} for the order parameter $\phi$. 
Note that in general the unknowns $\bu, \bJ_e, \tilde\bT, \phi, \bJ_\phi, \lambda_\phi$ (and their derivatives appearing in the local dissipation inequality) can attain arbitrary values for a given point in space and time. 
It can be checked with a straightforward computation, using \eqref{eq:appendix-div}, \eqref{eq:appendix-v}, \eqref{eq:appendix-B} and many reformulations (see, e.g., \cite[Section 2.2]{AGG2012} for the AGG part and, e.g., \cite[Section 4.4]{MP2018} for the viscoelastic part) that this local dissipation inequality holds true if the following constitutive assumptions are applied: 
% \footnote{Do you need the number of third line? No.\\Okay.}
\begin{gather*}
    \bJ_e = \lambda_\phi \bJ_\phi - \partial_t^\bullet \phi \frac{\partial \hat e}{\partial\nabla e} - \frac12 \bJ\abs{\bu}^2,
    \quad 
    \bJ = \rho'(\phi) \bJ_\phi,
    %\frac{2 \rho'(\phi)}{\overline\rho_1 + \overline\rho_2} \bJ_\phi,
    \quad 
    \bJ_\phi = - m(\phi) \nabla \lambda_\phi,
    \\
    \lambda_\phi = \frac{\partial \hat e}{\partial \phi} - \Div \frac{\partial \hat e}{\partial \nabla\phi},
    \quad 
    \mathbb{D}_d = \frac{1}{\tilde\lambda(\phi)} \mathbb{F}_e^{-1} \frac{\partial \hat e}{\partial \bbb} \mathbb{F}_e,
    \\
    \tilde \bT = \bbs - p \bbi
    - (\bu\otimes \bJ) 
    - \frac{\partial \hat e}{\partial \nabla\phi} \otimes \nabla\phi ,
    \quad 
    \bbs = \nu(\phi) (\nabla\bu + \tran{\nabla}\bu) 
    + 2 \frac{\partial \hat e}{\partial \bbb} \bbb ,
\end{gather*}
where $m(\phi), \tilde\lambda(\phi), \nu(\phi)$ are positive functions corresponding to a mobility, a relaxation and a viscosity, respectively. 
%and $\tilde\alpha(\phi)$ is a non-negative function that can specified as $\tilde\alpha(\phi) = \alpha(\phi) \mu(\phi)^{-1}$ to be consistent with \eqref{eqs:Model}. 
Here, also a relative mass flux $\bJ$ and the viscoelastic stress tensor $\bbs$ were introduced, also see \cite{AGG2012, MAA2018}.

The constitutive system of equations (with a general energy density) reads
\begin{subequations}
\label{eq:appendix-system}
\begin{alignat}{3}
\begin{split}
    \pt (\rho(\phi) \bu) + \Div ( \rho(\phi) \bu \otimes \bu ) & + \Div ( \bu \otimes \bJ ) + \nabla p 
    \\
    - \Div \big(\bbs(\nabla \bu, \bbb, \phi)\big) & = - \Div \Big( \frac{\partial \hat e}{\partial \nabla\phi} \otimes \nabla \phi \Big),
    \end{split}
   \\
    \Div \bu & = 0 ,
   \\
    \pt \bbb + \bu \cdot \nabla \bbb + \frac{2}{\tilde\lambda(\phi)} \frac{\partial \hat e}{\partial \bbb} \bbb  & = \bbb \tran{\nabla} \bu + \nabla \bu \bbb,
    % 			+ \frac{\kappa}{\mu(\phi)} \Delta \bbb 
   \\
    \pt \phi + \bu \cdot \nabla \phi & = \Div (m(\phi) \nabla q),
   \\
    q &= \frac{\partial \hat e}{\partial \phi} - \Div \frac{\partial \hat e}{\partial \nabla\phi} \,.
\end{alignat}
\end{subequations}
Here, the local dissipation is given by
\begin{align*}
    \mathcal{D} 
    = \frac{\nu(\phi)}{2} \abs{ \nabla\bu + \tran{\nabla}\bu}^2
    + m(\phi) \abs{\nabla q}^2
    + \frac{2 }{\tilde\lambda(\phi)} \abs{ \tran{\bbf_e} \frac{\partial \hat e}{\partial \bbb} \bbf_e }^2 \geq 0.
\end{align*}
We note that the system \eqref{eqs:Model} can be recovered from \eqref{eq:appendix-system} with the specific choice of the free energy density 
\begin{align*}
    \hat e(\phi, \nabla\phi, \bbb) 
    = \frac{\tilde\sigma \epsilon}{2} \abs{\nabla\phi}^2
    + \frac{\tilde\sigma}{\epsilon} W(\phi)
    + \frac{\mu(\phi)}{2} \tr(\bbb - \ln \bbb - \bbi),
\end{align*}
and a rescaling of the relaxation function $\tilde \lambda(\phi) = \lambda(\phi) \mu(\phi) / \alpha(\phi)$.

\section{Compactness of Time-Averaged Terms}
\label{sec:compactness-B}
{
In this section, we give three technical lemmata, which are of much significance for the limit passage concerning time-averaged terms, in particular, the compactness of these terms. First, we give a uniform boundedness and compactness property. 
\begin{lemma}[Compactness in $ L^p $] 
	\label{lem:uniform-bound-time-averaged}
	Let $h = T/N$ with $N\in \bbn$ and $f \in L^p(0,T; X)$ be a bounded function for $ p \in(1, \infty) $, where $X$ is a Banach space.
	Define $\bar{f}^N(t) = \frac{1}{h} \int_{I_{k+1}} f(s) \,\d s $ for $t \in I_{k+1}$ as a piecewise-in-time constant function in $ (0,T) $. Then $ \{\bar{f}^N\}_{N\in\bbn} \subset L^p(0,T; X) $. Moreover, it holds $\bar{f}^N\to f$ in $L^p(0,T;X)$, as $N\to \infty$.
\end{lemma}
\begin{proof}
	By virtue of \eqref{eqs:identity-integral-summation} and Jensen's inequality for the time integral (as $p \in(1, \infty) $), one knows that
	\begin{align*}
		\norm{\bar{f}^N}_{L^p(0,T; X) }^p
		= \int_0^T \norm{\bar{f}^N(s)}_{X}^p \,\d s
		& = h \sum_{k = 0}^{N - 1} \norm{\frac{1}{h} \int_{I_{k+1}} f(s) \,\d s}_{X}^p 
            \\ 
            &\leq h \sum_{k = 0}^{N - 1} \Big(\frac{1}{h} \int_{I_{k+1}} \norm{f(s)}_{X} \,\d s \Big)^p
  \\
		& \leq h \sum_{k = 0}^{N - 1} \frac{1}{h} \int_{I_{k+1}} \norm{f(s)}_{X}^p \,\d s
		= \int_0^T \norm{f(s)}_{X}^p \,\d s
		\leq C,
	\end{align*}
	where $ C > 0 $ is independent of $ h>0 $ (and $N\in\bbn$, respectively). The second part of the lemma follows directly from the uniform boundedness and the Banach--Steinhaus theorem, as the convergence result holds true for any $f\in C^\infty([0,T]; X) \subset L^p(0,T;X)$ densely, i.e., for any $f\in C^\infty([0,T]; X)$, it holds 
\begin{align*}
		\norm{\bar{f}^N - f}_{L^p(0,T; X) }^p
		= \int_0^T \norm{\bar{f}^N(s) - f(s)}_{X}^p \,\d s
		& = \sum_{k = 0}^{N - 1} \int_{I_{k+1}} \norm{\frac{1}{h} \int_{I_{k+1}} f(s) \,\d s - f(t)}_{X}^p \dt
  \\
  &=\sum_{k = 0}^{N - 1} \int_{I_{k+1}} \norm{\frac{1}{h} \int_{I_{k+1}} \big( f(s) - f(t)\big) \,\d s }_{X}^p \dt
  \\
  &\leq C h \sup_{t\in (0,T)} \norm{\partial_t f(t) }_{X}^p  \to 0,
\end{align*}
as $h\to 0$.
\end{proof}
\begin{remark}
    \label{rem:uniform-bound-time-averaged}
    %\footnote{I add this remark, since later we will apply above lemma for a uniform sequence.}
    The first part of the lemma also holds true for $f$ substituted by a uniformly bounded sequence $\{f^N\}_{N \in \bbn}$ in $L^p(0,T; X)$, that is, the sequence $\{\bar{f}^N\}_{N\in\bbn}$ with definition $\bar{f}^N(t) = \frac{1}{h} \int_{I_{k+1}} f^N(s) \,\d s$ for $t \in I_{k+1}$ is uniformly bounded in $L^p(0,T; X)$.
\end{remark}
Now concerning the weakly convergent sequences, we introduce the following lemma for the time-averaged functions.
\begin{lemma}[Weak compactness]
	\label{lem:weak-convergence-time-averaged}
    Let $h = T/N$ with $N\in \bbn$ and $\{f^N\}_{N \in \bbn} \subset L^2(Q_T)$ be a sequence satisfying $f^N \to f $ weakly in $L^2(Q_T)$, as $N \to \infty$ (resp.~$h \to 0$), with $f \in L^2(Q_T)$. Defining a piecewise-in-time constant function $\bar{f}^N(t) = \frac{1}{h} \int_{I_{k+1}} f^N(s) \,\d s $ for $t \in I_{k+1}$, we have for any function $\varphi \in L^2(Q_T)$ independent of $h$, 
    \begin{equation*}
        \int_{Q_T} \bar{f}^N \varphi \dxdt
        \to
        \int_{Q_T} f \varphi \dxdt, \text{ as } N \rightarrow \infty.
    \end{equation*}
\end{lemma}
\begin{proof}
    Define $\bar{\varphi}^N(t) = \frac{1}{h} \int_{I_{k+1}} \varphi(s) \,\d s$ for $t \in I_{k+1}$. In light of Lemma \ref{lem:uniform-bound-time-averaged}, $ \bar{\varphi}^N $ is uniformly bounded in $ L^2(Q_T) $ and it holds 
    % $\bar{\varphi}^N\to \phi$ strongly in $L^2(Q_T)$. 
    \begin{equation*}
        \bar{\varphi}^N \rightarrow \varphi, \text{ strongly in } L^2(Q_T).
    \end{equation*}
    With the help of  \eqref{eqs:identity-integral-summation} and Fubini's theorem, we then have
    \begin{align*}
        \int_0^T \int_\Omega \bar{f}^N \varphi \dxdt
        & = \sum_{k = 0}^{N - 1} \int_{I_{k+1}} \int_{\Omega} \bigg(\Big(\frac{1}{h} \int_{I_{k+1}} f^N(s) \,\d s\Big) \widetilde{\varphi}_{k+1}(t)\bigg) \dxdt \\
        & = \sum_{k = 0}^{N - 1} \int_{\Omega} \Big(\frac{1}{h} \int_{I_{k+1}} f^N(s) \,\d s\Big) \Big(\int_{I_{k+1}} \widetilde{\varphi}_{k+1}(t) \dt\Big) \d x \\
        & = \sum_{k = 0}^{N - 1} \int_{\Omega} \Big(\int_{I_{k+1}} f^N(s) \,\d s \Big) \Big(\frac{1}{h} \int_{I_{k+1}} \widetilde{\varphi}_{k+1}(t) \dt\Big) \dx \\
        & = \int_\Omega \int_0^T f^N(s) \bar{\varphi}^N(s) \,\d s\d x
        \to \int_0^T \int_\Omega f \varphi \dxdt,
    \end{align*}
    as $N \to \infty$, where the last convergence holds true for weakly convergent $f^N$ and strongly convergent $ \bar{\varphi}^N $.
\end{proof}
We now write the time-averaged terms as a convolution with a Dirac sequence. Let $N\in \bbn$ and $h=T/N$. We define $\zeta^N(t) = \frac{1}{h} \chi_{(-\frac h2, \frac h2)}(t)$ for $t\in\bbr$, where $\chi_{I}(\cdot)$ denotes the characteristic function on a given interval $I\subset\bbr$. Let $X$ be a Banach space and $p \in [1,\infty)$. Then, for $f\in L^p(0,T;X)$ and $t\in (\tfrac h2, T-\tfrac h2)$, we rewrite the time average over the interval $(t-\tfrac h2, t+\tfrac h2)$ as a convolution with $\zeta^N$, i.e.,
\begin{align*}
    (\zeta^N * f) (t) \coloneqq
    \int_{\bbr} \zeta^N(t - s) f(s) \, \d s
    = \frac1h \int_{\bbr} \chi_{(t-\tfrac h2, t+\tfrac h2)}(s) f(s) \,\d s
    = \frac1h \int_{t-\tfrac h2}^{t+\tfrac h2} f(s) \,\d s.
\end{align*}
Here, the function $f$ is naturally extended with a zero outside of $(0,T)$. We note some properties from, e.g., \cite[Theorems 4.13 and 4.15]{Alt2016}, that we will use: 
\begin{itemize}
    \item $\{(\zeta^N*f)\}_{N\in\bbn} \subset L^p(0,T;X)$ with $(\zeta^N*f) \to f$ in $L^p(0,T;X)$, as $N\to \infty$.
    \item For $f,g\in L^1(0,T;X)$, it holds $\int_{\bbr} (\zeta^N *f)(t) g(t) \dt = \int_{\bbr} f(t) (\zeta^N *g)(t) \dt$.
\end{itemize}
Note that in fact the first one can be derived by the Lebesgue differentiation theorem for $p = 1$ and additionally with Jensen's inequality for $1 < p <\infty$. 
The second one is a direct consequence, since $\zeta^N$ is an even function by construction. 
}

{
Regarding the weak-$*$ convergence, we have the following lemma.
\begin{lemma}[Weak-$*$ compactness]
\label{lem:weak-*-convergence-time-averaged}
Let $h=T/N$ with $N\in \bbn$. Moreover, let $f^N \in L^\infty(-h,T)$, $N\in\bbn$, with $f^N \to f$ weakly-$*$ in $L^\infty(0,T)$, as $N\to \infty$ (resp.~$h\to 0$). 
Then, for all $\varsigma\in L^1(0,T)$, it holds as $N\to \infty$ (resp.~$h\to0$),
\begin{align*}
    \int_0^T \Big( \frac1h \int_{t-h}^t f^N(s) \,\d s  \Big) \varsigma(t) \dt \to \int_0^T f(t) \varsigma(t) \dt.
\end{align*}
\end{lemma}
\begin{proof}
To prove the lemma, we rewrite the inner integral with the help of the Dirac sequence defined above with a shift of $\tfrac h2$. Then one applies the associativity of the convolution with the Dirac sequence, the second property above, to move the convolution to the test function $\varsigma\in L^1(0,T)$. Finally, combining with the weak-$*$ convergence of $f^N$ to $f$ in $L^\infty(0,T)$ and the first property above, we finish the proof.
Namely,
\begin{align*}
    \int_0^T \Big( \frac1h \int_{t-h}^t f^N(s) \,\d s  \Big) \varsigma(t) \dt
    & = \int_\bbr (\zeta^N * f^N) (t) \varsigma(t) \dt \\
    & = \int_\bbr f^N(t) (\zeta^N * \varsigma) (t) \dt 
    \to \int_\bbr f(t) \varsigma (t) \dt
    = \int_0^T f(t) \varsigma(t) \dt,
\end{align*}
as $N \to \infty$ (resp.~$h \to 0$).
\end{proof}
\begin{remark}
In principle, it is possible to generalize Lemmata \ref{lem:weak-convergence-time-averaged} and \ref{lem:weak-*-convergence-time-averaged} for a broader class of Lebesgue spaces. However, we only make use of these special cases in this work.
\end{remark}
We also note a compactness criterion with time translations, which can be found in, e.g., \cite[Section 8, Theorem 5]{Simon1985}.
\begin{lemma}\label{lem:compactness_translation}
Let $p\in [1,\infty]$ and let $X,Y,Z$ be Banach spaces with $X\subset\subset Y$ compactly and $Y\subset Z$ continuously. Moreover, let $\mathcal{F} \subset L^p(0,T;X)$ be a bounded set and $\norm{f(\cdot+h) - f}_{L^p(0,T-h; Z)} \to 0$ uniformly for $f\in \mathcal{F}$, as $h\to 0$. Then, $\mathcal{F}$ is relatively compact in $L^p(0,T;Y)$ if $p\in[1,\infty)$, and in $C([0,T]; Y)$ if $p=\infty$, respectively.
\end{lemma}
}

% \newpage
\section*{Acknowledgments}
{The authors are grateful to the referees for the careful reading of the paper and valuable suggestions, which improved remarkably the presentation of the manuscript.} 
YL and DT would like to thank their supervisors Prof. Dr. Helmut Abels and Prof. Dr. Harald Garcke, respectively, for their fruitful and inspiring discussions, as well as the careful proofreadings. 
%and careful proofreadings. 
The thanks also go to  Jonas Haselb\"{o}ck and Christoph Hurm for useful discussions.% on measure theory.

YL is supported by the startup funding from Nanjing Normal University -- Project-ID 184080H201B90, the Natural Science Foundation of Jiangsu Province (No.~BK20240572), and the Natural Science Foundation of the Jiangsu Higher Education Institutions of China (No.~24KJB110020).
The authors were both partially supported by the Graduiertenkolleg 2339 \textit{IntComSin} of the Deutsche Forschungsgemeinschaft (DFG, German Research Foundation) -- Project-ID 321821685. The support is gratefully acknowledged.

\section*{Data Availability}
Data sharing not applicable to this article as no datasets were generated during the current study.

\section*{Conflict of Interest}
All authors declare that they have no conflicts of interest.


\begin{thebibliography}{999}
%	\addcontentsline{toc}{section}{References}
%	\setlength{\itemsep}{-5pt}	
	
	\bibitem{Abels2009ARMA}
	H. Abels, On a diffuse interface model for two-phase flows of viscous, incompressible fluids with matched densities, Arch. Ration. Mech. Anal. {\bf 194} (2009), no.~2, 463--506.
	
	\bibitem{Abels2009CMP}
	H. Abels, Existence of weak solutions for a diffuse interface model for viscous, incompressible fluids with general densities, Comm. Math. Phys. {\bf 289} (2009), no.~1, 45--73.
	
	\bibitem{ADG2013}
	H. Abels, D. Depner\ and\ H. Garcke, Existence of weak solutions for a diffuse interface model for two-phase flows of incompressible fluids with different densities, J. Math. Fluid Mech. {\bf 15} (2013), no.~3, 453--480.
	
	\bibitem{ADG2013degenerate}
	H. Abels, D. Depner\ and\ H. Garcke, On an incompressible Navier-Stokes/Cahn-Hilliard system with degenerate mobility, Ann. Inst. H. Poincar\'{e} C Anal. Non Lin\'{e}aire {\bf 30} (2013), no.~6, 1175--1190.
	
	\bibitem{AGG2022}
	H. Abels, H. Garcke\ and\ A. Giorgini, Global regularity and asymptotic stabilization for the incompressible Navier-Stokes-Cahn-Hilliard model with unmatched densities, Math. Ann. {\bf 389} (2024), no.~2, 1267--1321.
	
	\bibitem{AGG2012}
	H. Abels, H. Garcke\ and\ G. Gr\"{u}n, Thermodynamically consistent, frame indifferent diffuse interface models for incompressible two-phase flows with different densities, Math. Models Methods Appl. Sci. {\bf 22} (2012), no.~3, 1150013, 40 pp. 
	
	\bibitem{AW2007}
	H. Abels\ and\ M. Wilke, Convergence to equilibrium for the Cahn-Hilliard equation with a logarithmic free energy, Nonlinear Anal. {\bf 67} (2007), no.~11, 3176--3193. 
	
	\bibitem{ACGR2022}
	A. Agosti, P. Colli, H. Garcke\ and\ E. Rocca, A Cahn-Hilliard model coupled to viscoelasticity with large deformations, Commun. Math. Sci. {\bf 21} (2023), no.~8, 2083--2130.

    \bibitem{Alt2016}
    H. W. Alt, {\it Linear functional analysis}, translated from the German edition by Robert N\"{u}rnberg, Universitext, Springer-Verlag London, Ltd., London, 2016.
 
	\bibitem{BB2011}
	J. W. Barrett\ and\ S. Boyaval, Existence and approximation of a (regularized) Oldroyd-B model, Math. Models Methods Appl. Sci. {\bf 21} (2011), no.~9, 1783--1837. 
		
	\bibitem{BLS2017}
	J. W. Barrett, Y. Lu\ and\ E. S\"{u}li, Existence of large-data finite-energy global weak solutions to a compressible Oldroyd-B model, Commun. Math. Sci. {\bf 15} (2017), no.~5, 1265--1323.
	
	\bibitem{BBM2021}
	M. Bathory, M. Bul\'{\i}\v{c}ek\ and\ J. M\'{a}lek, Large data existence theory for three-dimensional unsteady flows of rate-type viscoelastic fluids with stress diffusion, Adv. Nonlinear Anal. {\bf 10} (2021), no.~1, 501--521.

    % \bibitem{BL1976}
    % J. Bergh\ and\ J. L\"{o}fstr\"{o}m, {\it Interpolation spaces. An introduction}, Grundlehren der Mathematischen Wissenschaften, No. 223, Springer-Verlag, Berlin, 1976.
 
	\bibitem{BS2018}
	D. Breit\ and\ S. Schwarzacher, Compressible fluids interacting with a linear-elastic shell, Arch. Ration. Mech. Anal. {\bf 228} (2018), no.~2, 495--562.
	
% 	\bibitem{BP2014}
% 	D. Bresch\ and\ C. Prange, Newtonian limit for weakly viscoelastic fluid flows, SIAM J. Math. Anal. {\bf 46} (2014), no.~2, 1116--1159.
	
	\bibitem{B+2021}
        A. Brunk et al., Analysis of a viscoelastic phase separation model, J. Phys.: Condens. Matter {\bf 33} (2021), no.~23, 234002.
    
    \bibitem{BL2022a}
    A. Brunk\ and\ M. Luk\'{a}\v{c}ov\'{a}-Medvid'ov\'{a}, Global existence of weak solutions to viscoelastic phase separation part: I. Regular case, Nonlinearity {\bf 35} (2022), no.~7, 3417--3458.
    
    \bibitem{BL2022b}
    A. Brunk\ and\ M. Luk\'{a}\v{c}ov\'{a}-Medvid'ov\'{a}, Global existence of weak solutions to viscoelastic phase separation: part II. Degenerate case, Nonlinearity {\bf 35} (2022), no.~7, 3459--3486.
    
	\bibitem{BLL2022}
	A. Brunk, Y. Lu\ and\ M. Luk\'{a}\v{c}ov\'{a}-Medvi\v{d}ov\'{a}, Existence, regularity and weak-strong uniqueness for three-dimensional Peterlin viscoelastic model, Commun. Math. Sci. {\bf 20} (2022), no.~1, 201--230.


    \bibitem{BLM2024}
   M. Bul\'{\i}\v{c}ek, T. Los\ and\ J. M\'{a}lek, On three-dimensional flows of viscoelastic fluids of
Giesekus type, (2024), Preprint, \arxiv{2403.17348}.
 
	\bibitem{Ciarlet2013}
	P. G. Ciarlet, {\it Linear and nonlinear functional analysis with applications}, Society for Industrial and Applied Mathematics, Philadelphia, PA, 2013.
	
	\bibitem{CK2012}
	P. Constantin\ and\ M. Kliegl, Note on global regularity for two-dimensional Oldroyd-B fluids with diffusive stress, Arch. Ration. Mech. Anal. {\bf 206} (2012), no.~3, 725--740. 
	
	\bibitem{FN2017}
	E. Feireisl\ and\ A. Novotn\'{y}, {\it Singular limits in thermodynamics of viscous fluids}, second edition, Advances in Mathematical Fluid Mechanics, Birkh\"{a}user/Springer, Cham, 2017.
	
	\bibitem{Frigeri2016}
	S. Frigeri, Global existence of weak solutions for a nonlocal model for two-phase flows of incompressible fluids with unmatched densities, Math. Models Methods Appl. Sci. {\bf 26} (2016), no.~10, 1955--1993.
	
% 	\bibitem{Frigeri2021}
% 	S. Frigeri, On a nonlocal Cahn-Hilliard/Navier-Stokes system with degenerate mobility and singular potential for incompressible fluids with different densities, Ann. Inst. H. Poincar\'{e} C Anal. Non Lin\'{e}aire {\bf 38} (2021), no.~3, 647--687.
	
	\bibitem{GGW2019}
	C. G. Gal, M. Grasselli\ and\ H. Wu, Global weak solutions to a diffuse interface model for incompressible two-phase flows with moving contact lines and different densities, Arch. Ration. Mech. Anal. {\bf 234} (2019), no.~1, 1--56.
	
	\bibitem{GKT2022}
	H. Garcke, B. Kov\'{a}cs\ and\ D. Trautwein, Viscoelastic Cahn--Hilliard models for tumour growth, Math. Models Methods Appl. Sci. {\bf 32} (2022), no.~13, 2673--2758.
	
% 	\bibitem{GL2017}
% 	H. Garcke\ and\ K. F. Lam, Well-posedness of a Cahn-Hilliard system modelling tumour growth with chemotaxis and active transport, European J. Appl. Math. {\bf 28} (2017), no.~2, 284--316.
	
	\bibitem{GKL2018}
	M.-H. Giga, A. Kirshtein\ and\ C. Liu, Variational modeling and complex fluids, in {\it Handbook of mathematical analysis in mechanics of viscous fluids}, 73--113, Springer, Cham, 2018. 
	
	\bibitem{Giorgini2021}
	A. Giorgini, Well-posedness of the two-dimensional Abels-Garcke-Gr\"{u}n model for two-phase flows with unmatched densities, Calc. Var. Partial Differential Equations {\bf 60} (2021), no.~3, Paper No. 100, 40 pp.
	
	\bibitem{GK2022}
	A. Giorgini\ and\ P. Knopf, Two-phase flows with bulk-surface interaction: thermodynamically consistent Navier-Stokes-Cahn-Hilliard models with dynamic boundary conditions, J. Math. Fluid Mech. {\bf 25} (2023), no.~3, Paper No. 65, 44 pp.

% 	\bibitem{GLT2017}
% 	J. Giesselmann, C. Lattanzio\ and\ A. E. Tzavaras, Relative energy for the Korteweg theory and related Hamiltonian flows in gas dynamics, Arch. Ration. Mech. Anal. {\bf 223} (2017), no.~3, 1427--1484.
	
% 	\bibitem{GD2003}
% 	A. Granas\ and\ J. Dugundji, {\it Fixed point theory}, Springer Monographs in Mathematics, Springer-Verlag, New York, 2003. 

        \bibitem{GO1997a}
        M. Gmerla, H. \"{O}ttinger, Dynamics and thermodynamics of complex fluids. {I}.
              {D}evelopment of a general formalism, Phys. Rev. E (3) {\bf 56} (1997), no.~6, 6620--6632.

        \bibitem{GO1997b}
        M. Gmerla, H. \"{O}ttinger, Dynamics and thermodynamics of complex fluids. {II}.
              {I}llustrations of a general formalism, Phys. Rev. E (3) {\bf 56} (1997), no.~6, 6633--6655.

    \bibitem{GL2017}
    H. Garcke and K. F. Lam, {Well--posedness of a Cahn--Hilliard system modelling tumour growth with chemotaxis and active transport}, European J. Appl. Math. {\bf 28} (2017), no.~2, 284--316.
    
    \bibitem{GM2016}
    G. Gr\"{u}n\ and\ S. Metzger, On micro-macro-models for two-phase flow with dilute polymeric solutions---modeling and analysis, Math. Models Methods Appl. Sci. {\bf 26} (2016), no.~5, 823--866.
    
	\bibitem{GPV1996}
	M. E. Gurtin, D. Polignone\ and\ J. Vi\~{n}als, Two-phase binary fluids and immiscible fluids described by an order parameter, Math. Models Methods Appl. Sci. {\bf 6} (1996), no.~6, 815--831.

    \bibitem{GT2024}
    H. Garcke and D. Trautwein, Approximation and existence of a viscoelastic phase-field model for tumour growth in two and three dimensions, Discrete Contin. Dyn. Syst. Ser. S {\bf 17} (2024), no.~1, 221--284.
 
	\bibitem{HH1977}
	P. C. Hohenberg\ and\ B. I. Halperin, Theory of dynamic critical phenomena, Rev. Mod. Phys. {\bf 49} (1977), 435–479.
	
	\bibitem{HL2007}
	D. Hu\ and\ T. Leli\`evre, New entropy estimates for Oldroyd-B and related models, Commun. Math. Sci. {\bf 5} (2007), no.~4, 909--916.
	
	\bibitem{HL2016Lp}
	X. Hu\ and\ F. Lin, Global solutions of two-dimensional incompressible viscoelastic flows with discontinuous initial data, Comm. Pure Appl. Math. {\bf 69} (2016), no.~2, 372--404.
	
	% \bibitem{HL2016L2}
	% X. Hu\ and\ F. Lin, On the Cauchy problem for two dimensional incompressible viscoelastic flows, \arxiv{1601.03497}, 2016.

    % \bibitem{HLL2018}
    % X. Hu, F. Lin\ and C. Liu, Equations for Viscoelastic Fluids, in {\it Handbook of mathematical analysis in mechanics of viscous fluids}, 1045--1073, Springer, Cham, 2018.
	
	\bibitem{HW2015}
	X. Hu\ and\ H. Wu, Long-time behavior and weak-strong uniqueness for incompressible viscoelastic flows, Discrete Contin. Dyn. Syst. {\bf 35} (2015), no.~8, 3437--3461.
	
	\bibitem{KMS2021}
	M. Kalousek, S. Mitra\ and\ A. Schl\"omerkemper, Existence of weak solutions to a diffuse interface model involving magnetic fluids with unmatched densities, NoDEA Nonlinear Differential Equations Appl. {\bf 30} (2023), no.~4, Paper No.~52, 53 pp.
	
	\bibitem{KTT2022}
	W. Kim, K. Tawri\ and\ R. Temam, Local well-posedness of a three-dimensional phase-field model for thrombus and blood flow, Rev. R. Acad. Cienc. Exactas F\'{\i}s. Nat. Ser. A Mat. RACSAM {\bf 116} (2022), no.~4, Paper No. 149, 23 pp.
	
	\bibitem{LeiLZ2008}
	Z. Lei, C. Liu\ and\ Y. Zhou, Global solutions for incompressible viscoelastic fluids, Arch. Ration. Mech. Anal. {\bf 188} (2008), no.~3, 371--398.
	
	\bibitem{LinLZ2005}
	F. Lin, C. Liu\ and\ P. Zhang, On hydrodynamics of viscoelastic fluids, Comm. Pure Appl. Math. {\bf 58} (2005), no.~11, 1437--1471.
	
	\bibitem{LZ2008}
	F. Lin\ and\ P. Zhang, On the initial-boundary value problem of the incompressible viscoelastic fluid system, Comm. Pure Appl. Math. {\bf 61} (2008), no.~4, 539--558.
	
	\bibitem{LR2014}
	D. Lengeler\ and\ M. R\r{u}\v{z}i\v{c}ka, Weak solutions for an incompressible Newtonian fluid interacting with a Koiter type shell, Arch. Ration. Mech. Anal. {\bf 211} (2014), no.~1, 205--255.
	
	\bibitem{LM2000}
	P. L. Lions\ and\ N. Masmoudi, Global solutions for some Oldroyd models of non-Newtonian flows, Chinese Ann. Math. Ser. B {\bf 21} (2000), no.~2, 131--146.

    \bibitem{Liu1972}
    I S. Liu, Method of Lagrange multipliers for exploitation of the entropy principle, Arch. Rational Mech. Anal. {\bf 46} (1972), no.~2, 131--148.

 
	\bibitem{LT1998}
	J. Lowengrub\ and\ L. Truskinovsky, Quasi-incompressible Cahn-Hilliard fluids and topological transitions, R. Soc. Lond. Proc. Ser. A Math. Phys. Eng. Sci. {\bf 454} (1998), no.~1978, 2617--2654.
	
	\bibitem{LMNR2017}
	M. Luk\'{a}\v{c}ov\'{a}-Medvi\v{d}ov\'{a}, H. Mizerov\'{a}, \v{S}. Ne\v{c}asov\'{a}\ and\ M. Renardy, Global existence result for the generalized Peterlin viscoelastic model, SIAM J. Math. Anal. {\bf 49} (2017), no.~4, 2950--2964.
	
	\bibitem{Masmoudi2013}
	N. Masmoudi, Global existence of weak solutions to the FENE dumbbell model of polymeric flows, Invent. Math. {\bf 191} (2013), no.~2, 427--500.
	
	\bibitem{MP2018}
	J. M\'{a}lek\ and\ V. Pr\r{u}\v{s}a, Derivation of equations for continuum mechanics and thermodynamics of fluids, in {\it Handbook of mathematical analysis in mechanics of viscous fluids}, 3--72, Springer, Cham, 2018.
	
	\bibitem{MPSS2018}
	J. M\'{a}lek, V. Pr\r{u}\v{s}a, T. Sk\v{r}ivan\ and\ E. S\"uli, Thermodynamics of viscoelastic rate-type fluids with stress diffusion, Physics of Fluids {\bf 30} (2018), no.~2, 023101.
	
	\bibitem{MAA2018}
	D. Mokbel, H. Abels\ and\ S. Aland, A phase-field model for fluid-structure interaction, J. Comput. Phys. {\bf 372} (2018), 823--840.
	
% 	\bibitem{MT2008}
% 	L. Molinet\ and\ R. Talhouk, Newtonian limit for weakly viscoelastic fluid flows of Oldroyd type, SIAM J. Math. Anal. {\bf 39} (2008), no.~5, 1577--1594.
    
    \bibitem{Oldroyd1950}
    J. G. Oldroyd, On the formulation of rheological equations of state, Proc. Roy. Soc. London Ser. A {\bf 200} (1950), 523--541.
    
    \bibitem{Onsager1932}
    L. Onsager, Reciprocal relations in irreversible processes I., Phys. Rev. {\bf 37} (1932), no.~4, 405--426.
 
	\bibitem{RT2021}
	M. Renardy\ and\ B. Thomases, A mathematician's perspective on the Oldroyd B model: progress and future challenges, J. Non-Newton. Fluid Mech. {\bf 293} (2021), Paper No. 104573, 12 pp.
	
	
	\bibitem{Sieber2020}
	O. Sieber, On convergent schemes for a two-phase Oldroyd-B type model with variable polymer density, J. Numer. Math. {\bf 28} (2020), no.~2, 99--129.

    \bibitem{Simon1985}
    J. Simon, Compact sets in the space $L^p(0,T;B)$, Ann. Mat. Pura Appl. (4) {\bf 146} (1987), 65--96.
    
	\bibitem{Sohr2001}
	H. Sohr, {\it The Navier-Stokes equations}, [2013 reprint of the 2001 original] [MR 1928881], Modern Birkh\"{a}user Classics, Birkh\"{a}user/Springer Basel AG, Basel, 2001.
	
% 	\bibitem{ZZE2006}
% 	D. Zhou, P. Zhang\ and\ W. E, Modified models of polymer phase separation, Phys. Rev. E {\bf 73} (2006), 061801.
	\end{thebibliography}
\end{document}